\documentclass[12pt]{amsart}
\usepackage[top=1in, bottom=1in, left=1in, right=1in]{geometry}
\usepackage[T2A,T1]{fontenc}
\usepackage{times}

\usepackage{amssymb,amsmath,amsthm}
\usepackage{mathrsfs} 
\usepackage{bbm} 
\usepackage{enumitem}
\usepackage[hyphens]{url}
\usepackage{hyperref}
\usepackage{dsfont}
\usepackage{nicefrac}
\usepackage{mathtools}
\usepackage{xspace}

\usepackage{graphicx,tikz}
\usepackage{color}

\makeatletter
\renewcommand\part{\@startsection{part}{2}%
	\z@{0.5\linespacing\@plus2\linespacing}{\linespacing}%
	{\normalfont\large\scshape\bfseries\centering}}
\makeatother

\renewcommand{\le}{\leqslant}
\renewcommand{\leq}{\leqslant}
\renewcommand{\ge}{\geqslant}
\renewcommand{\geq}{\geqslant}

\definecolor{myblue}{rgb}{0.09,0.32,0.44} 
\hypersetup{pdfborder={0 0 0},pdfborderstyle={},colorlinks=true,linkcolor=myblue,citecolor=myblue,urlcolor=blue}

\DeclareRobustCommand{\cyrtext}{%
  \fontencoding{T2A}\selectfont\def\encodingdefault{T2A}}
\DeclareRobustCommand{\textcyr}[1]{\leavevmode{\cyrtext #1}}

\numberwithin{equation}{section}

\theoremstyle{plain}

\theoremstyle{plain}
\newtheorem{thm}{Theorem}
\newtheorem{cor}{Corollary}

\newtheorem{lem}{Lemma}[section]
\newtheorem{prop}[lem]{Proposition}
\newtheorem{claim}[lem]{Claim}
\newtheorem{thm-n}[lem]{Theorem}

\theoremstyle{remark}
\newtheorem{rem}{Remark}[section]
\newtheorem*{rem*}{Remark}
\newtheorem*{notat*}{Notation}
\newtheorem*{exm*}{Example}

\theoremstyle{definition}
\newtheorem{dfn}[lem]{Definition}

\newcommand{\N}{\mathbb{N}}
\newcommand{\Z}{\mathbb{Z}}
\newcommand{\Q}{\mathbb{Q}}
\newcommand{\R}{\mathbb{R}}
\renewcommand{\C}{\mathbb{C}}

\newcommand{\T}{\mathbb{T}}
\newcommand{\Y}{\Upsilon}

\newcommand{\F}{\mathbb{F}}
\newcommand{\E}{\mathbb{E}}
\renewcommand{\P}{\mathbb{P}}

\newcommand{\M}{\mathrm{Merge}}

\newcommand{\CA}{\mathcal{A}}
\newcommand{\CB}{\mathcal{B}}

\newcommand{\CD}{\mathcal{D}}
\newcommand{\CE}{\mathcal{E}}

\newcommand{\CG}{\mathcal{G}}

\newcommand{\CI}{\mathcal{I}}

\newcommand{\CM}{\mathcal{M}}
\newcommand{\CN}{\mathcal{N}}
\newcommand{\CO}{\mathcal{O}}
\newcommand{\CP}{\mathcal{P}}
\newcommand{\CQ}{\mathcal{Q}}
\newcommand{\CR}{\mathcal{R}}
\newcommand{\CS}{\mathcal{S}}
\newcommand{\CT}{\mathcal{T}}

\newcommand{\CX}{\mathcal{X}}
\newcommand{\CY}{\mathcal{Y}}

\newcommand{\eq}[2]{ \begin{equation} \label{#1}
    #2 
\end{equation} }
\newcommand{\al}[1]{\begin{align} #1 \end{align} }
\newcommand{\als}[1]{\begin{align*} #1 \end{align*} }

\newcommand{\nn}{\nonumber \\}

\newcommand{\dee}{\,\mathrm{d}}

\DeclareMathOperator{\res}{res}
\DeclareMathOperator{\supp}{supp}

\newcommand{\fl}[1]{\left\lfloor#1\right\rfloor}
\newcommand{\fls}[1]{\lfloor#1\rfloor}
\newcommand{\ceil}[1]{\left\lceil#1\right\rceil}

\renewcommand{\tilde}{\widetilde}
\newcommand{\eps}{\varepsilon}
\renewcommand{\phi}{\varphi}

\newcommand{\bs}\boldsymbol{}

\renewcommand{\Re}{{\rm Re}}

\renewcommand{\bar}[1]{\overline{#1}}
\renewcommand{\mod}[1]{\,({\rm mod}\,#1)}

\makeatletter
\def\moverlay{\mathpalette\mov@rlay}
\def\mov@rlay#1#2{\leavevmode\vtop{%
   \baselineskip\z@skip \lineskiplimit-\maxdimen
   \ialign{\hfil$\m@th#1##$\hfil\cr#2\crcr}}}
\newcommand{\charfusion}[3][\mathord]{
    #1{\ifx#1\mathop\vphantom{#2}\fi
        \mathpalette\mov@rlay{#2\cr#3}
      }
    \ifx#1\mathop\expandafter\displaylimits\fi}
\makeatother

\newcommand{\cupdot}{\charfusion[\mathbin]{\cup}{\cdot}}
\newcommand{\bigcupdot}{\charfusion[\mathop]{\bigcup}{\cdot}}

\begin{document}

\title{Irreducibility of random polynomials: general measures}

\author[L. Bary-Soroker]{Lior Bary-Soroker}
\address{LBS: Raymond and Beverly Sackler School of Mathematical Sciences\\
	 Tel Aviv University\\
	 Tel Aviv 69978, Israel.}
\email{{\tt barylior@tauex.tau.ac.il}}

\author[D. Koukoulopoulos]{Dimitris Koukoulopoulos}
\address{DK: D\'epartement de math\'ematiques et de statistique\\
Universit\'e de Montr\'eal\\
CP 6128 succ. Centre-Ville\\
Montr\'eal, QC H3C 3J7\\
Canada}
\email{{\tt dimitris.koukoulopoulos@umontreal.ca}}

\author[G. Kozma]{Gady Kozma}
\address{GK: Department of Mathematics\\
		The Weizmann Institute of Science\\ 
		Rehovot 76100, Israel.}
\email{{\tt gady.kozma@weizmann.ac.il}}


\date{\today}

\begin{abstract} Let $\mu$ be a probability measure on $\mathbb{Z}$ that is not a Dirac mass and that has finite support. We prove that if the coefficients of a monic polynomial $f(x)\in\mathbb{Z}[x]$ of degree $n$ are chosen independently at  random according to $\mu$ while ensuring that $f(0)\neq0$, then there is a positive constant $\theta=\theta(\mu)$ such that $f(x)$ has no divisors of degree $\le \theta n$ with probability that tends to 1 as $n\to\infty$.  

Furthermore, in certain cases, we show that a random polynomial $f(x)$ with $f(0)\neq0$ is irreducible with probability tending to 1 as $n\to\infty$. In particular, this is the case if $\mu$ is the uniform measure on a set of at least 35 consecutive integers, or on a subset  of $[-H,H]\cap\mathbb{Z}$ of cardinality $\ge H^{4/5}(\log H)^2$ with $H$ sufficiently large. In addition, in all of these settings, we show that the Galois group of $f(x)$ is either $\mathcal{A}_n$ or $\mathcal{S}_n$ with high probability.

Finally, when $\mu$ is the uniform measure on a finite arithmetic progression of at least two elements, we prove a random polynomial $f(x)$ as above is irreducible with probability $\ge\delta$ for some constant $\delta=\delta(\mu)>0$. In fact, if the arithmetic progression has step 1, we prove the stronger result that the Galois group of  $f(x)$  is $\mathcal{A}_n$ or $\mathcal{S}_n$ with probability $\ge\delta$.
\end{abstract}

\maketitle

\setcounter{tocdepth}{1}
\tableofcontents


\part{Main results and outline of their proof}\label{part-intro}

\section{Introduction}\label{sec:intro}
Is a random polynomial with integer coefficients irreducible over the rationals with high probability? 
This captivating problem, a forerunner in the effort to understand high-dimensional algebraic phenomena, has a long history. In 1936, van der Waerden \cite{Waerden} was the first to prove that if we choose a polynomial $f(x)\in\Z[x]$ of degree $n$ uniformly at random with coefficients in a box of size $H$, say in $\{1,\ldots, H\}$, then $f$ is irreducible and has Galois group equal to the full symmetric group $\CS_n$ with probability that tends to $1$ as $H\to\infty$. Van der Waerden's estimate on this probability has been steadily improved over the years, most notably in 1976 by Gallagher \cite{Gallagher}, who used the large sieve inequality, and in 2012 by Dietmann \cite{Dietmann}, who used bounds on the number of integral points on certain varieties. In a recent preprint \cite{Bh}, Bhargava established van der Waerden's conjecture that the probablity that $f$ has Galois group different than $\CS_n$ is $O_n(1/H)$. This estimate was previously known in the cases $n\in\{3,4\}$ by work of Chow and Dietmann \cite{chow-dietmann}.

When the size of the box is fixed and the degree grows, progress has been slower. The first important breakthrough was achieved in Konyagin's highly influential work \cite{konyagin}, where he showed that, with high probability, a polynomial whose smallest and largest coefficients are 1 and all others are chosen uniformly at random from $\{0,1\}$ has no divisors of small degree with high probability.  Recently, the first and third author showed that if the coefficients are selected from special sets that satisfy appropriate arithmetic restrictions, then the polynomial is irreducible almost surely \cite{BSK}. Breuillard and Varj\'u extended this result to very general distributions for the coefficients of the random polynomial, but relying on the validity of the Riemann Hypothesis for a family of Dedekind zeta functions \cite{BV}. 

Our purpose in this paper is to replace the arithmetic restrictions of \cite{BSK} with weaker restrictions, more analytic in nature. In general, given a set of integers $\CN$, we let $\Y_\CN(n)$ denote the set of monic polynomials of degree $n$ all of whose coefficients lie in $\CN$ and whose constant coefficient is non-zero. An example of our results is the following:

\begin{thm}\label{main thm}
Let $H\ge1$ and let $\CN$ be a set of $N$ consecutive integers contained in $[-H,H]$. Then there are absolute constants $c,\delta>0$ and $n_0\ge1$ such that if we choose a polynomial $A\in \Y_\CN(n)$ uniformly at random with $n\ge \max\{n_0,(\log H)^3\}$, then the following hold: 
\begin{enumerate}
	\item If $N\ge 35$, then $A$ is irreducible with probability $\ge1-n^{-c}$.
	\item If $2\le N\le 34$, then $A$ is irreducible with probability $\ge\delta$.
\end{enumerate}
\end{thm}

For comparison, assuming the validity of the Riemann Hypothesis for Dedekind zeta functions, the above mentioned result of Breuillard and Varj\'u \cite{BV} is a stronger version of Theorem \ref{main thm}, as they establish for all $N\ge2$ a precise asymptotic formula for the probability that an element of $\Y_\CN(n)$ is reducible. They deduce their theorem as a special case of a more general result.

Similarly, our method produces naturally a more general result than Theorem~\ref{main thm}: instead of sampling the $j^{\textrm{th}}$ coefficient of $A$ uniformly at random from $[1,N]$, we may work with a general sequence of probability measures $(\mu_j)_{j=0}^\infty$ on the integers $\Z$. Then by a ``random monic polynomial'' $A(T)$ of degree $n$ we mean a polynomial 
\[
A(T)=T^n+a_{n-1}T^{n-1}+a_{n-2}T^{n-2}+\cdots +a_0 ,
\]
where the coefficients of the powers of $T$ are independent random variables with $a_j$ sampled according to the measure $\mu_j$. More concretely, we equip the set of polynomials
\[
\CM(n):=\{A(T)\in\Z[T]\ \text{monic}:\deg(A)=n\}
\]
with the measure
\[
\P_{\CM(n)}(A):= \prod_{j=0}^{n-1}\mu_j(a_j) .
\]
Choosing $A\in \Upsilon_{[1,N]}(n)$ uniformly at random corresponds to the above law when
\eq{uniform measure}{
\mu_j(a)=1_{[1,N]}(a)/N 
\quad\text{for all}\ j. 
}

Our more general results take their cleanest form when the measures $\mu_j$ are all the same measure $\mu$ that satisfies certain hypotheses. To state them, we adopt the notation
\[
\|\mu\|_p
	:=\begin{cases}
		(\sum_{a\in\Z}\mu(a)^p)^{1/p}			&\text{if}\ 1\le p<\infty,\\
		\sup_{a\in\Z}\mu(a)&\text{if}\ p=\infty.
	\end{cases}
\]


We prove that there are no divisors of degree $< \theta n$ asymptotically almost surely.

\begin{thm}\label{general main thm - almost irreducible} 
Let $H\ge3$ and $n\ge 3$ be integers, and let $\mu_j=\mu$ for all $j$, where $\mu$ is a probability measure on $\Z$ such that:
	\begin{enumerate}
		\item (support not too large) $\supp(\mu)\subseteq[-H,H]$;
		\item (measure not too concentrated) $\|\mu\|_\infty\le1-\eps$.
	\end{enumerate}
	There are absolute constants $c,C>0$ and a constant $\theta>0$ depending at most on $H,\eps$ such that 
	\begin{equation}\label{prob-lb weak mainthm}
	\P_{\CM(n)}\Big(\mbox{all divisors of $A(T)$ have degree $\ge\theta n$} \,\Big|\, a_0\neq0\Big) \ge 1- n^{-c}
	\end{equation}
	for all $n\ge C\eps^{-20000}(\log H)^{10^6}$. As a matter of fact, we can take $\theta=c'\eps/(\log H)^5$ for some absolute constant $c'>0$.
\end{thm}

Theorem~\ref{general main thm - almost irreducible} strengthens Konyagin's result \cite[Theorem 2]{konyagin} which states that \eqref{prob-lb weak mainthm} holds with $cn/\log n$ replacing $\theta n$ in the special case where $\mu$ taking the values $0,1$ uniformly. 

To get irreducibility one needs to pass the barrier $\theta=1/2$, and we achieve it under some restrictions on $\mu$.

\begin{thm}\label{general main thm} 
Let $H\ge3$ and $n\ge 3$ be integers, and let $\mu_j=\mu$ for all $j$, where $\mu$ is a probability measure on $\Z$ such that:
\begin{enumerate}
\item (support not too large) $\supp(\mu)\subseteq[-H,H]$;
\item (support not too sparse) $\|\mu\|_2^2\le \min\{H^{-4/5},n^{1/16}/H\}/(\log H)^2$.
\end{enumerate}
There are absolute constants $c>0$ and $H_0\ge3$ such that if $H\ge H_0$, then
\begin{equation}\label{prob-lb mainthm}
\P_{\CM(n)}\Big(A(T)\ \text{is irreducible} \,\Big|\, a_0\neq0\Big)
	\ge 1-n^{-c} .
\end{equation}
\end{thm}

\begin{rem}
For fixed $\mu$ and generic values of $n$, we expect that $\P(A(-1)=0)\asymp1/\sqrt{n}$ because the event $A(-1)=0$ is equivalent to the sum of the random variables $a_0-a_1+a_2\mp\cdots+(-1)^{n-1}a_{n-1}$ being exactly equal to $(-1)^{n-1}$. Thus,  \eqref{prob-lb mainthm} is optimal up to the value of the constant $c$. Breuillard and Varj\'u \cite{BV} prove a more precise version of \eqref{prob-lb mainthm} that specifies the secondary main terms coming from cyclotomic factors of $A(T)$, and with condition (b) replaced by the weaker assumption that $\|\mu\|_2<1$ (which is equivalent to having $\|\mu\|_\infty<1$, since $\|\mu\|_\infty\le\|\mu\|_2\le(\|\mu\|_\infty)^{1/2}$). 
\end{rem}

Specializing Theorems~\ref{general main thm - almost irreducible} and~\ref{general main thm} to measures that are uniform on some set of integers, we get:
\begin{cor}\label{general main cor} 
Let $H\ge3$ and $n\ge3$ be integers, and let $\CN\subset[-H,H]$ be a set of $N$ integers. There are absolute constants $c>0$ and $n_0,H_0\ge3$ and a constant $\theta=\theta(H)>0$ such that if we choose a polynomial $A$ from $\Y_\CN(n)$ uniformly at random, then the following hold:
\begin{enumerate}
	\item If $N\ge2$ and $n\ge \max\{n_0,(\log H)^3\}$, then all divisors of $A$ have degree $\ge \theta n$ with probability $\ge1-n^{-c}$.
	\item If $H\ge H_0$, $N\ge H^{4/5}(\log H)^2$, and $n\ge (H/N)^{16}(\log H)^{32}$, 
then $A$ is irreducible with probability $\ge1-n^{-c}$.
\end{enumerate}
\end{cor}

As it is clear from Corollary~\ref{general main cor}, we cannot prove that a random polynomial is irreducible almost surely when the coefficients are sampled according to the measure
\eq{squares}{
\mu(a)=\frac{1_{[1,H]}(a)\cdot 1_{a=\square}}
	{\lfloor{\sqrt{H}\rfloor}}.
}
This is not a mere technicality: our method allows us to take $\theta=1/2$ in Theorem~\ref{general main thm - almost irreducible} only if we can find some primes $p$ modulo which the measure $\mu$ is sufficiently ``close'' to the uniform distribution on $\Z/p\Z$ in the sense that the $L^1$ norm of its Fourier transform mod $p$ has ``better than square-root cancellation''. (The precise condition that we need is stated in Theorem~\ref{master thm} in \S\ref{sec:outline}.) However, the squares fail to satisfy such a condition, since
\[
\bigg|\sum_{a\mod p}e(a^2k/p)\bigg|=\sqrt{p}
\quad\mbox{for all $p>2$ and all $k\not\equiv0\mod p$}.
\]
As a result, we cannot take $\theta=1/2$ in Theorem~\ref{general main thm - almost irreducible} for the measure of \eqref{squares}. 

On the other hand, odd powers become completely equidistributed modulo certain primes.  For instance, if $p\equiv2\mod3$ and $k\not\equiv0\mod p$, then
\[
\sum_{a\mod p}e(a^3k/p)=0.
\]
This allows us to work with the set of cubes and, more generally, with the set of odd powers as it were all of $\Z$ and obtain the following result:

\begin{thm}\label{main thm s-th powers}
 Given  $H\ge1$ and an odd integer $d$, let $\CN=\{k^d:k\in\Z\cap[1,H]\}$. There are constants $c>0$ and $H_0,n_0\ge3$, with $c$ being absolute and $H_0,n_0$ depending only on $d$, such that if $H\ge H_0$, $n\ge\max\{n_0,(\log H)^3\}$ and we choose a polynomial from $\Y_\CN(n)$ uniformly at random, then  it is irreducible with probability $\ge1-n^{-c}$.
\end{thm}

In general, the chances of picking a set that fails to have the needed ``better than square-root-cancellation'' property for some primes are slim. Thus, we can show that Corollary~\ref{general main cor}(b) holds for a generic set $\CN$ that is sufficiently large. This is the content of the following theorem.

\begin{thm}\label{main thm generic}
Let $H\ge1$ and $N\in\Z_{\ge2}$, and let $\CN$ denote a random set chosen uniformly at random among all subsets of $\Z\cap[-H,H]$ of $N$ elements. Then there are absolute constants $c>0$ and $n_0\ge1$ such that the set $\CN$ has the following property with probability $1-O\big(1/\sqrt{N})$: 

If $n\ge\max\{n_0,(\log H)^3\}$ and we choose a polynomial from $\Y_\CN(n)$ uniformly at random, then  it is irreducible with probability $\ge 1-n^{-c}$.
\end{thm}

%
%
%

Let us conclude this introductory section by discussing the Galois group of random polynomials. Recall that a polynomial is irreducible if and only if its Galois group is transitive. Thus it is tempting to try to generalize the above results by characterizing more precisely the Galois group, viewing it as a random subgroup of the symmetric group $\CS_n$. Indeed, this was accomplished in \cite{BSK} and \cite{BV}. As in these cases, we show that the Galois group contains the alternating group $\CA_n$ with high probability, though we obtain a worse estimate for the probability of this event than in \cite{BV}.

\begin{thm}\label{thm:GGs}
In the setting of Theorems~\ref{main thm}(a) and~\ref{general main thm}-\ref{main thm generic}, we have in addition that the Galois group of the random polynomial (given that $a_0\neq 0$) is either $\CS_n$ or $\CA_n$ with probability bigger than $1-n^{-c}$ for some absolute positive constant $c$. In the setting of Theorem~\ref{main thm}(b), the same conclusion holds but with probability that is $\ge \delta-n^{-c}$.
\end{thm}

Large Galois group have many applications, and are closely related to large images of Galois representations -- for example, see \cite{Zarhin}. We do not elaborate on that, and instead we give an application to irreducibility.

A large Galois group implies a high-level irreducibility: Let $A\in \mathbb{Q}[T]$ be a polynomial of degree $n$ with roots $t_1,\ldots, t_n\in \mathbb{C}$. We say that  $A$ is \emph{$k$-fold irreducible} if $A$ is irreducible over $\mathbb{Q}$ and, for all $j = 1,\ldots, k-1$ the polynomial 
\[
A(T)/\prod_{i=1}^j (T-t_i) = \prod_{i=j+1}^n (T-t_i) 
\]
is irreducible in $\mathbb{Q}(t_1,\ldots, t_j)[T]$. 
Note that this definition is independent of the ordering of the roots and that $1$-fold irreducibility is the same as irreducibility. For example $T^{10}+T^9+\cdots +T+1$ is $1$-fold irreducible but not $2$-fold irreducible, while $T^{10}+T^9+\cdots +T-1$ is $10$-fold irreducible. Indeed a polynomial is $k$-fold irreducible if and only if its Galois group is $k$-transitive, and in the first case the Galois group is $C_{10}$ which is not doubly transitive and in the second case the Galois group is $S_{10}$ which is $10$-transitive. Since $\CA_n$ and $\CS_n$ are both $(n-2)$-transitive we get an immediate corollary.

\begin{cor}
	A random polynomial in the  setting of Theorems~\ref{main thm}-\ref{main thm generic} is $(n-2)$-fold irreducible with probability $\ge1-n^{-c}$, with the exception of part (b) of Theorem~\ref{main thm}, where the probability is $\ge \delta-n^{-c}$.
\end{cor}

The proof of Theorem~\ref{thm:GGs} will be discussed in Part~\ref{part-galois} of the paper. Our approach is to apply finite group theory (a \L uczak-Pyber style theorem -- see \S\ref{proof of thm:LP}) to get from irreducibility to a large Galois group, and then to deduce $(n-2)$-fold irreducibility. In contrast, in \cite{BV}, Breuillard and Varj\'u prove directly that a random polynomial is $k$-fold irreducible for some $k>(\log n)^2$, and then they deduce it has a large Galois group. 

\subsection*{Acknowledgments} The authors would like to thank Sam Chow, Vesselin Dimitrov, Andrew Granville, David Hokken and James Maynard for their useful remarks on the paper. They would also like to thank the referees for their thorough reading of the paper.

L.~B.-S. was supported by the Israel Science Foundation (grant no.\ 702/19), D.~K. was supported by Natural Sciences and Engineering Research Council of Canada (Discovery Grant  2018-05699) and by the Fonds de recherche du Qu\'ebec - Nature et technologies (projet de recherche en \'equipe - 256442), and G.~K.~was supported by the Jesselson Foundation and by Paul and Tina Gardner.

This project started during a visit of L.B.-S. to Concordia University of Montreal for the 2017-18 academic year, which was supported by the Simons CRM Scholar-in-Residence Program. In addition, the paper was partly written during D.K.'s visit to the University of Oxford in the Spring of 2019 (supported by Ben Green's Simons Investigator Grant 376201). They would like to thank their hosts for the support and hospitality.

\subsection* {Notation}

We adopt the usual asymptotic notation of Vinogradov: given two functions $f,g\colon X\to\R$ and a set $Y\subseteq X$, we write ``$f(x)\ll g(x)$ for all $x\in Y$'' if there is a constant $c=c(f,g,Y)>0$ such that $|f(x)|\le cg(x)$ for all $x\in Y$. The constant is absolute unless otherwise noted by the presence of a subscript. If $h\colon X\to\R$ is a third function, we use Landau's notation $f=g+O(h)$ to mean that $|f-g|\ll h$.

Finally, below is an index of various symbols we will be using throughout the paper for easy reference.

\begin{list}{}{\settowidth{\labelwidth}{00.00.0000}
    \setlength{\leftmargin}{\labelwidth}
    \addtolength{\leftmargin}{\labelsep}
    \renewcommand{\makelabel}[1]{#1\hfil}}
\item[$\alpha(s,\gamma;P)$] $\displaystyle\max_{\substack{QR=P\\ Q>1}}\max_{\ell\in\Z}\frac{1}{Q^{1-\gamma}}\sum_{k\in\Z/Q\Z}|\hat{\mu}(k/Q+\ell/R)|^s$ where $\mu$ is a probability measure on $\Z$.
\item[$\alpha(P)$]
  $\displaystyle \max_{QR=P,\,Q>1} \max_{\ell\in\Z/R\Z}
  \frac{1}{\sqrt{Q}} \sum_{k\in \Z/Q\Z} |\hat{\mu}(k/Q+\ell/R)|$ with $\mu$ a probability measure on $\Z$.
\item[$\alpha$] $\delta/4-\theta/2$ in \S\ref{proof of thm:LP}.
\item[$\delta_\CP(n;\bs\ell)$]
  $\displaystyle\frac{1}{\prod_{p\in\CP}p^{\ell_p}}
  \sum_{\substack{\bs H\in\CM_\CP(\bs \ell) \\ T\nmid H_p\ \forall p\in\CP}} 
  \sum_{\substack{\bs G\mod{\bs H} \\ (G_p,H_p)=1\ \forall p\in\CP}}
  \sigma_\CP(n;\bs G/\bs H)$ \quad for $\bs\ell=(\ell_p)_{p\in\CP}$.
\item[$\Delta_\CP(n;m)$] $\displaystyle
  \mathop{\sum\cdots \sum}_
         {\substack{{\bs D}\,:\,\deg(D_p)\le  m, \\
             T\nmid D_p ,\,\forall p\in\CP}}
  \max_{\bs C\mod{\bs D}}
  \bigg| \P_{\bs A\in\CM_\CP(n)} (\bs A\equiv \bs C\mod{\bs D} ) 
			-  \frac{1}{\|\bs D\|_\CP}\bigg|$.
\item[$\lambda_0$] The constant $1/(4-4\log2) = 0.8147228\dots$.
\item[$\mu_j$]The distribution of the $j^{\textrm{th}}$ coefficient; see $\P_{\CM(n)}$.
\item[$\hat\mu(\xi)$] The Fourier transform $\sum_{a\in\Z} \mu(a)e(a\xi)$ of the measure $\mu$.
\item[$\sigma_\CP(n;\bs X)$] $\prod_{j=0}^{n-1} |\hat{\mu}_j(\psi_\CP(T^j\bs X))|$, when $\bs X\in\F_\CP((1/T))$.
\item[$\tau(A)$]$\#\{D\in\F_p[T]: D\ \text{monic},\, D|A\}$, when $A\in\F_p[T]\smallsetminus\{0\}$.
\item[$\Y_\CN(n)$] The set of monic polynomials of degree $n$ all of whose coefficients lie in $\CN$, and whose constant coefficient is non-zero.
\item[$\psi_p(X)$] $\res(X_p)/p\mod 1$ with  $X\in\F_p((1/T))$.
\item[$\psi_\CP(\bs X)$] $\sum_{p\in\CP}\res(X_p)/p\mod 1$ with  $\bs X\in\F_\CP((1/T))$.
\item[$\omega(A)$]$\#\{I\in\F_p[T]:I\ \text{monic and irreducible},\, I|A\}$, when $A\in\F_p[T]\smallsetminus\{0\}$.
\item[$\bs A,\bs B,\dotsc$] Bold letters denote sets indexed by primes, e.g. $\bs A=(A_p)_{p\in\CP}$. In addition, 
  $\bs A|\bs B$ means that $A_p|B_p$ for all $p\in\CP$, $\bs A\equiv \bs B\mod{\bs D}$ means that $A_p\equiv B_p\mod{D_p}$ for all $p\in\CP$, etc.
\item[$e(x)$]$e^{2\pi i x}$ with $x\in\R$.
\item[$\F_\CP{[}T{]}$] $\prod_{p\in\CP}\F_p[T]$.
\item[$\F_\CP((1/T))$] $\prod_{p\in\CP}\F_p((1/T))$.
\item[$\CG_A$] The Galois group of the polynomial $A(T)\in\Z[T]$, viewed as a subgroup of the symmetric group $\CS_{\deg(A)}$.
\item[$\CI_p$] A set of monic irreducible polynomials in $\F_p[T]$. See $(A_p,\CI_p)$ and $A_p|\CI_p$ below.
\item[$\CM(n)$]$\{A(T)\in\Z[T]\ \text{monic}:\deg(A)=n\}$.
\item[$\CM_p(n)$]$\{f(T)\in \F_p[T]\ \text{monic}:\deg(f)=n\}$.
\item[$\CM_\CP(\bs n)$]$\prod_{p\in\CP}\CM_p(n_p)$.
\item[$\CM_\CP(n)$]$\prod_{p\in\CP}\CM_p(n)$.
\item[$\M(\rho;y)$] The set of permutations in $\CS_n$ whose cycle structure is a $y$-merging of $\rho$, with $\rho$ a partition of $n$. (See Definition~\ref{dfn-merge} for the notion of ``$y$-merging''.)
\item[$\N$] $\{1,2,3,\dotsc\}$
\item[$\CP$] A set of $r$ (usually 4) primes, often indexed as $p_1<\cdots<p_r$.
\item[$\P_{\CM(n)}$] The measure on $\CM(n)$ given by
  $\P_{\CM(n)}\big(\sum_{j=0}^{n-1}a_jT^j+T^n\big)=\prod_{j=0}^{n-1}\mu_j(a_j)$.
\item[$\P_{\CM_p(n)}$] The projection of $\P_{\CM(n)}$ to $\CM_p(n)$, ditto for $\P_{\CM_\CP(\bs n)}$ and $\P_{\CM_\CP(n)}$.
\item[$\P_{A\in \CM(n)}$] The same measure, where we write ``$A\in\CM(n)$'' to stress that $A$ is the variable of integration. Ditto for $\P_{A\in\CM_p(n)}$, $\P_{\bs A\in\CM_\CP(\bs n)}$ and $\P_{\bs A\in\CM_\CP(n)}$.
\item[$r$] The number of primes in $\CP$, usually 4.
\item[$\res(X)$] For $X=\sum_{j=-\infty}^\infty c_jT^j$,  $\res(X)=c_{-1}$.
\item[$s$] A parameter in $\N\cap[1,n^{1/100}]$.
\item[$\CT_n$] In part IV, the set of permutations lying in a transitive subgroup of $\CS_n$ that is different from $\CS_n$ and $\CA_n$.
\item[$\T_p$] $\displaystyle\big\{X\in\F_p((1/T)):X=\smash{\sum_{j\le -1}}c_jT^j\big\}$.
\item[$\T_\CP$] $\prod_{p\in\CP}\T_p$.
\item[$\|D\|_p$]$p^{\deg(D)}$ when $D$ is a polynomial.
\item[$\|\bs D\|_\CP$]$\prod_{p\in\CP}p^{\deg(D_p)}$ when $\bs D=(D_p)_{p\in\CP}$ is a list of polynomials.
\item[$\|x\|$]The distance of $x$ to the nearest integer, when $x\in\R$.
\item[$(A_p,\CI_p)$] $\prod_{I_p\in\CI_p,I_p|A_p}I_p$ when $\CI_p$ is a family of polynomials.
\item[$(A,B)$] The greatest common divisor of $A$ and $B$, when they are both polynomials or numbers.
\item[{$[A,B]$}] The least common multiple of $A$ and $B$, when they are both polynomials or numbers.
\item[$A_p|\CI_p$] means that $A_p|\prod_{I_p\in\CI_p}I_p$ when $\CI_p$ is a family of polynomials.
\item[{$[n]$}] the set $\{1,2,\dots,n\}$.
\item[$\sim$] $x\sim y$ is the same as $x=(1+o(1))y$.
\item[$\lesssim$] $x\lesssim y$ is the same as $x\le (1+o(1)) y$.
\item[$\asymp$] $x\asymp y$ is the same as $x=O(y)$ and $y=O(x)$.
\item[$\ll$] $x\ll y$ is the same as $x=O(y)$.
\item[$\vdash$] $\rho\vdash n$ means that $\rho$ is a partition of $n$, namely, $\rho=(\rho_1,\dotsc,\rho_r)$ with $\rho_i\in\N$, $\rho_1\le\cdots\le \rho_r$, and $\sum_{i=1}^r\rho_i=n$.
\end{list}

\section{Outline of the proofs}\label{sec:outline}

We present now the main steps of the proof of our theorems. Unlike in the introduction, the results here allow different distributions for different coefficients of our random polynomial (the coefficients would still need to be independent). More formally, given a sequence of probability measures on the integers $\mu_0,\mu_1,\dots,\mu_{n-1}$, we write $\P_{\CM(n)}$ for the probability measure on $\CM(n)$ given by
\[
\P_{\CM(n)}(T^n+a_{n-1}T^{n-1}+\cdots+a_1T+a_0)=\prod_{j=0}^{n-1}\mu_j(a_j). 
\]

We first explain how to prove that 
\eq{irreducibility}{
\P_{\CM(n)}\Big(A(T)\ \text{is reducible} \Big|a_0\neq0\Big)\le n^{-c}
}
under appropriate assumptions on the measures $\mu_j$. Our results on the Galois group will be explained later, in \S~\ref{sec:galois intro}.

Proving \eqref{irreducibility} requires bounding from above the probability that $A$ has a divisor of degree $\le n/2$. For certain measures, we will not be able to prove such a strong result. We will show instead that there are no divisors of degree $\le \theta n$, for some suitable $\theta\in(0,1/2)$.

\subsection{Ruling out factors of small degree} 
The first thing we do is to rule out factors of small degree, say $\le \xi(n)$ for some $\xi(n)\to\infty$. There are many proofs of this fact in the literature, most notably in Konyagin's work \cite{konyagin} that allows taking $\xi(n)\asymp n/\log n$. Konyagin's result is formulated for coefficients $\{0,1\}$ and our coefficients are more general, so we adapt it to our setting. We shall only prove a weak version of his results (what we prove is the analog of the first page in Konyagin's argument, where he works with the function $\xi(n)=n^{1/2-o(1)}$). The large factors will be dealt with later. Here is the exact statement:

\begin{prop}\label{large irr factors} 
Let $n\in\N$ and $\mu_0,\mu_1,\dots,\mu_{n-1}$ be a sequence of probability measures on the integers all of which satisfy the following conditions:
\begin{enumerate}
\item (support not too large) \quad $\supp(\mu_j)\subseteq[-\exp(n^{1/3}),\exp(n^{1/3})]$ \quad for $j\ge0$;
\item (measures not too concentrated) \quad $\|\mu_j\|_\infty\le 1-n^{-1/10}$ \quad for $j\ge1$. 
\end{enumerate} 
Assume further that $\supp(\mu_0)\neq\{0\}$. We then have that
\[
\P_{\CM(n)}\Big(A(T)\ \mbox{has an irreducible factor of degree}\ \le n^{1/10} \,\Big|\,a_0\neq0\Big) \ll n^{-7/20}.
\] 
\end{prop}

 We present the proof of this result in \S~\ref{large irr factors pf}.

\subsection{Ruling out factors of large degree}\label{large deg}
Given Proposition~\ref{large irr factors}, we must rule out factors of $A$ of degree $\in[n^{1/10},\theta n]$, with $\theta =1/2$ for Theorems~\ref{main thm},~\ref{general main thm}--\ref{main thm generic}. In the predecessor paper \cite{BSK}, this was done by using Galois theory and then applying a result of Pemantle, Peres and Rivin \cite{PPR} about the structure of ``random permutations''. Here, instead of passing to the permutation world, we adapt the idea of Pemantle, Peres and Rivin to the polynomial setting.

The argument is simpler to describe in the model case of Theorem~\ref{main thm}(a), which is realized when all measures $\mu_j$ are the uniform counting measure on $N$ consecutive integers, say $\Z\cap[1,N]$. Assume we know that $A$ has a factorisation 
\[
A=BC\quad\text{where}\ B\in\CM(k) .
\]
We may then reduce this equation modulo any prime $p$ and obtain the equation
\[
A_p=B_pC_p,
\]
where $A_p$ denotes the reduction of $A$ mod $p$, and $B_p$ and $C_p$ are defined analogously. In addition,
\[
B_p\in \CM_p(k):=\{f(T)\in \F_p[T]\ \text{monic}:\deg(f)=k\}.
\]
Hence, if $A$ has a degree $k$ divisor, so does $A_p$ for any prime $p$.  To continue, we make two crucial observations:
\begin{itemize}
	\item if $p|N$, then the induced distribution of $A_p$ in $\CM_p(n)$ is the uniform distribution;
	\item if $\CP=\{p_1,\dots,p_r\}$ is any set of distinct prime factors of $N$, the Chinese Remainder Theorem implies that the induced random variables $A_{p_1},\dots,A_{p_r}$ are independent from each other.
\end{itemize}
Hence, for any set $\CP$ of prime divisors of $N$, we have that
\eq{independence}{
\P_{\CM(n)}(A\ \text{has a factor of degree}\ k)
	\le \prod_{p\in\CP} \P_{\CM_p(n)}(A_p\ \text{has a factor of degree}\ k),
}
where $\P_{\CM_p(n)}$ is the uniform counting measure on $\F_p[T]$ here. 

The advantage of working in the set $\CM_p(n)$ instead of the set $\CM(n)$ is that the former has a very well understood arithmetic. In particular, there is a famous analogy that allows us to go back and forth between results for the ring $\Z$ and for the ring $\F_p[T]$. Briefly, integers and polynomials over $\F_p$ share many similar statistical properties, after appropriate normalization. Dividing by units, we restrict our attention to positive integers and to monic polynomials, respectively. 
With this in mind, note that there are about $x$ positive integers of size $\le x$. The ``size'' of a polynomial $A_p\in\F_p[T]$ is measured by its norm
\[
\|A_p\|_p:= p^{\deg(f)}.
\]
And, indeed, we find that $\#\{A_p\in\F_p[T]:A_p\ \text{monic},\ \|A_p\|_p\le p^n\}\asymp p^n$ for each integer $n$. In addition, we note that there are about $ x/\log x$ primes $\le x$, whereas there are about $p^n/n$ monic irreducible polynomials $f\in\F_p[T]$ of norm $\le p^n$. Hence, for our purposes, the role of the natural logarithm in $\Z$ is played by the degree in $\F_p[T]$. Both functions are additive. 

Now, Ford \cite{ford} proved that 
\eq{divisors-int}{
\#\{n\le x:\exists d|n,\, y\le d\le 2y\}\asymp \frac{x}{(\log y)^\eta(\log\log y)^{3/2}}
\quad(3\le y\le\sqrt{x})
}
where
\[
\eta=1-\frac{1+\log\log2}{\log2}= 0.08607\dots
\]
The analogous result\footnote{There is also a famous analogy between statistical properties of integers and those of permutations. The articles \cite{PPR} and \cite{BSK} are set in the world of permutations. The corresponding result to Ford's estimate \eqref{divisors-int} was established by Eberhard, Ford and Green \cite{EFG}.}
in $\F_p[T]$ was proven recently by Meisner \cite{meisner}:
\eq{divisors-poly}{
\#\{A_p\in\CM_p(n) :\exists B_p|A_p,\, \deg(B_p)=k\}
	\asymp \frac{p^n}{k^\eta(\log k)^{3/2}}
	\quad(2\le k\le n/2) .
}
Inserting this bound into \eqref{independence}, we conclude that
\[
\P_{\CM(n)}(A\ \text{has a factor of degree}\ k)
	\ll k^{-r\eta+o(1)}\quad\text{as}\ k\to\infty,
\]
where $r=\#\CP$. If $N$ is divisible by 12 distinct prime factors, we may take $r=12$ in the above estimate. Since $12\eta>1$, we conclude that
\[
\P_{\CM(n)}\big(A\ \text{has a factor of degree}\ \ge n^{1/10}\big)
	\ll \sum_{k\ge n^{1/10}} k^{-12\eta+o(1)}\ll n^{-(12\eta-1)/10+o(1)}.
\]
This completes the proof of Theorem~\ref{main thm}(a) when $N$ has at least 12 distinct prime factors. 

\medskip

It turns out that the above argument is too crude. In comparison, the first and third authors proved in \cite{BSK} that having 4 distinct prime factors is also sufficient. The reason of the deficiency of the above argument is that different $k$ are dependent. Indeed, even though the estimate \eqref{divisors-poly} for a single $k$ is sharp, most of the polynomials counted by it, i.e., polynomials with a degree $k$ divisor mod $p$, have more than their fair share of \emph{irreducible} divisors mod $p$. We may then use other combinations of these irreducible divisors to obtain other values of $k$ as degrees of divisors. Let us make this discussion more quantitative. 

Most polynomials $f\in\CM_p(n)$ that have a divisor of degree $k$ have about $\log k/\log2$ irreducible factors of degree $k$ or less\footnote{Even though this assertion is well-known to experts, going back to Erd\H os's work on the multiplication table problem \cite{erd1,erd2}, its proof does not appear explicitly in the literature. It can be proven by a careful adaptation of \cite[Lemma 4.2]{meisner} followed by an application of \cite[Lemma 4.3]{meisner}.}. On the other hand, it is known that most polynomials $f\in\CM_p(n)$ have about $\log k$ irreducible factors of degree at most $k$, for all sufficiently large $k$. More precisely, let us fix some $\eps\in(0,1/10]$, and let us write $E_p(n;\eps)$ for the event that, for each $k\in[n^{1/10},n]$, the induced polynomial $A_p$ has $\le(1+\eps)\log k$ irreducible factors of degree $\le k$. Then it can be proven that 
\[
\P_{\CM_p(n)}(E_p(n;\eps) \ \text{does not occur})
\ll_\eps n^{-c_\eps}
\]
for some $c_\eps>0$. 
Using the above estimate, we have a relative version of \eqref{independence}:
\als{
&\P_{\CM(n)}(A\ \text{has a factor of degree}\ \in[n^{1/10},n/2])\\
	&\quad= 
	\P_{\CM(n)}\Big(A\ \text{has a factor of degree}\ \in[n^{1/10},n/2]\,\Big|\,A_p \in E_p(n;\eps)\ \forall p\in\CP\Big)
		+ O_{\eps,r}(n^{-c_\eps}) \\
	&\quad\le \sum_{n^{1/10}\le k\le n/2}
		\prod_{p\in\CP} 
		\P_{\CM_p(n)}\Big(A_p\ \text{has a factor of degree}\ k\,\Big|\, E_p(n;\eps)\Big) +
		O_{\eps,r}(n^{-c_\eps}), 
}
where to go from the second to the third line we used the union bound and the independence of the random variables $A_p$ with $p\in\CP$.  Now, if $\P_{\CM_p(n)}$ is the uniform measure on $\CM_p(n)$, then standard techniques about divisors of integers can be adapted to demonstrate that
\[
\P_{\CM_p(n)}\Big(A_p\ \text{has a factor of degree}\ k\,\Big|\, E_p(n;\eps)\Big) \ll_\eps k^{\log2-1+\eps} 
\quad \text{for}\ k\in[n^{1/10},n/2]\cap\Z.
\]
Taking $\eps=1/100$, we have that $1-\log2-\eps>1/4$. 
We thus find that if $N$ is divisible by at least 4 distinct prime factors, then 
\als{
\P_{\CM(n)}\big(A\ \text{has a factor of degree}\ \in[n^{1/10},n/2]\big)
	&\ll_\eps \sum_{k\ge n^{1/10}} k^{-4(1-\log 2-\eps)}
	+ n^{-c_\eps}  \\
	&\ll_\eps n^{-c_\eps'}
}
with
$c_\eps'=\min\{c_\eps,4(1-\log2-\eps)-1\}>0$. 

\medskip

This is the rough outline of the proof of Theorem~\ref{main thm}a in the special case when $N$ has at least four distinct prime factors. To adapt this proof to a general value of $N$ and to the even more general set-up of Theorems~\ref{general main thm}-\ref{main thm generic}, we must circumvent two obstacles:
\begin{itemize}
\item for general measures $\mu$, we cannot always find primes $p$ such that the random variable $A_p$ is uniformly distributed in $\CM_p(n)$;
\item for general measures $\mu$, we cannot always find four primes $p_1,\dots,p_4$ for which the random variables $A_{p_1},\dots,A_{p_4}$ are mutually independent.
\end{itemize}
It turns out, however, that we can find approximate versions of uniformity and independence for rather general measures $\mu_j$, as we explain below.

\subsection{From approximate equidistribution to irreducibility} 
We will prove a general result that allows us to go from an equidistribution statement about the tuple $(A_p)_{p\in\CP}$ to showing that $A$ with $a_0\neq0$ is irreducible with high probability. To state our result, we must introduce some notation.

Given a finite set of primes $\CP$, we use boldface letters to mean a vector indexed by the primes in $\CP$. Thus, $\bs A$ denotes the vector of polynomials $(A_p)_{p\in\CP}$. We further set
\[
\F_\CP[T] := \prod_{p\in\CP}\F_p[T]= \{\bs A : A_p\in \F_p[T]\ \text{for each}\ p\in\CP\}  
\]
for the set of all such vectors. Recall that $\CM_p(n)$ denotes the set of monic polynomials over $\F_p$ of degree $n$. We then also set
\[
\CM_\CP(\bs n) = \{\bs A : A_p\in \CM_p(n_p)\ \text{for each}\ p\in\CP\} .
\]
In the special case when $n_p=n$ for each $p$, we simplify the notation by letting
\[
\CM_\CP(n) = \{\bs A : A_p\in \CM_p(n)\ \text{for each}\ p\in\CP\} .
\]

If the polynomial $A(T)=a_0+a_1T+\cdots+a_{n-1}T^{n-1}+T^n\in\CM(n)$ is distributed according to the measure $\P_{\CM(n)}$, that is to say, it occurs with probability
\[
\P_{\CM(n)}(A) = \prod_{j=0}^{n-1}\mu_j(a_j),
\]
then the vector $\bs A$ is distributed in $\CM_{\CP}(n)$ according to the measure
\[
\P_{\CM_\CP(n)}(\bs A) := \prod_{j=0}^{n-1} \bigg( \sum_{\substack{a\in\Z \\ a\equiv a_{j,p} \mod p\ \forall p\in\CP}} \mu_j(a)\bigg)  ,
\]
where $a_{j,p}$ denotes the coefficient of $T^j$ of $A_p$.

In order to carry out the argument outlined in \S~\ref{large deg}, we will show that for certain choices of measures $\mu_j$, the multiplicative structure of $\bs A$ has approximately the same distribution as if we had selected each $A_p$ independently and uniformly at random with respect to the uniform measure in $\CM_p(n)$. 

More precisely, writing $\bs D|\bs A$ to mean that $D_p|A_p$ for all $p\in\CP$, what we need to show is that 
\[
\P_{\bs A\in\CM_\CP(n)}(\bs D|\bs A)
	\sim \prod_{p\in\CP} \frac{\#\{A_p\in \CM_p(n): D_p|A_p\}}{\#\CM_p(n)}
\]
as $n\to\infty$, for all $\bs D\in\F_\CP[T]$ all of whose components $D_p$ have degree $\le \theta n$, with $\theta = 1/2$ for irreducibility (in fact, we need to go a bit further than $\theta n$ for technical reasons that will be explained later). Indeed, if we have at our disposal such an estimate, then the methods of \S~\ref{large deg} can be adapted to the more general measure $\P_{\CM_\CP(n)}$. 

Note that
\[
\frac{\#\{\bs A\in \CM_p(n):D_p|A_p\}}{\#\CM_p(n)}
	= \frac{1}{p^{\deg(D_p)}} =: \frac{1}{\| D_p\|_p} .
\]
Hence, our task becomes to show that
\eq{approximate equidistribution}{
\P_{\bs A\in\CM_\CP(n)}\big(\bs D|\bs A\big) 
	\sim \frac{1}{\|\bs D\|_\CP} := \prod_{p\in\CP} \frac{1}{\|D_p\|_p} 
}
for $\bs D\in\F_\CP[T]$ all of whose components have degree $\le \theta n$ or a bit larger.

It turns out that we do not actually need \eqref{approximate equidistribution} to hold for {\it all} $\bs D$ of sufficiently large degree but only {\it on average}. For technical reasons\footnote{Notice that
$A_p\equiv a_0\mod T$ for all $p$, and in particular $A_p\mod T$ is distributed according to the projection of the measure $\mu_0$ onto $\Z/p\Z$, which could be rather arbitrary. 
This creates a lot of technical complications that we avoid by only considering congruence classes that are coprime to $T$.}, we exclude $D_p$'s that are divisible by $T$. To state our results, we adopt the notational convention
\[
\bs A\equiv \bs C\mod{\bs D}\qquad\Leftrightarrow\qquad A_p\equiv C_p\mod{D_p}\quad \forall p\in\CP 
\]
and we define
\eq{delta-dfn}{
\Delta_\CP(n;m):= \mathop{\sum\cdots\sum}_{\substack{\bs D=(D_p)_{p\in\CP} \\ D_p\ \text{monic},\, \deg(D_p)\le m, \\ T\nmid D_p\ \forall p\in\CP}}
	\max_{\bs C\mod{\bs D}}
		\bigg| \P_{\bs A\in\CM_\CP(n)} (\bs A\equiv \bs C\mod{\bs D} ) 
			-  \frac{1}{\|\bs D\|_\CP}\bigg| .
}
We also introduce the constant 
\[
\lambda_0:= \frac{1}{4-4\log 2} = 0.8147228\dots
\]
that plays a special role in our results.

\begin{prop}\label{from distr to irr} 
Let $\eps\in(0,1/100]$, $\theta\in(0,1/2]$, $n\in\N$ and $\mu_0,\mu_1,\dots,\mu_{n-1}$ be a sequence of probability measures on the integers satisfying the following conditions:
\begin{enumerate}
\item (support not too large) $\supp(\mu_j)\subseteq[-\exp(n^{1/3}),\exp(n^{1/3})]$ for all $j$.
\item (joint equidistribution modulo four primes) 
There is a set of four primes $\CP$ such that
\begin{equation}\label{eq:thm Delta_P<}
\Delta_\CP(n;\theta n+n^{\lambda_0+\eps}) \le n^{-30}  .
\end{equation}
\item (measure not too concentrated) for all $j\ge1$, we have $\|\mu_j\|_\infty\le 1-n^{-1/10}$, and for all $p\in\CP$, we further have 
 $\sum_{a\equiv0\mod p}\mu_j(a)\le 1-n^{-\eps/200}$.
\end{enumerate} 
Then there are constants $c=c(\eps)>0$ and $C=C(\eps)\ge1$ such that
\[
\P_{\CM(n)}\big(\mbox{$A(T)$ has a divisor in $\Z[T]$ of degree $\le \theta n$},\ a_0\neq0\big) \le Cn^{-c} .
\]
\end{prop}

The above result, that will be proved in Part~\ref{part-irr}, reduces Theorems~\ref{main thm}-\ref{main thm generic} to establishing condition (b) in each setting, except for Theorem~\ref{main thm}(b) that requires one additional argument that allows us to go from having only divisors of degree $\ge \theta n$ to having irreducibility for a positive proportion of polynomials. This argument originates in Konyagin's work \cite{konyagin} and we present it in \S\ref{sec:main thm pf}.

\subsection{Controlling the joint distribution of \texorpdfstring{$(A_p)_{p\in\CP}$}{the A-s}}

Let us now explain how to establish condition (b) of Proposition~\ref{from distr to irr}. Consider the case when 
\[
	\mu(n)=1_{[1,211]}(n)/211.
\] 
The induced measure mod 2 is given by
\[
	\mu^*_2(\ell\mod2):=\sum_{a\equiv \ell\mod2}\mu(a).
\] 
We have $\mu_2(0\mod2)=105/211$ and $\mu_2(1\mod2)=106/211$. So, even though we do not have perfect equidistribution mod 2, we have a distribution that resembles very closely the uniform distribution. Similar observations are true for the primes 3,5,7, as well for the divisors of 210.

The above set-up is reminiscent of the literature on the set of integers whose $g$-ary expansion contains only digits from some prescribed set $\CD$. Call $W_{g,\CD}$ the set of such integers. If we want to count primes in $W_{g,\CD}$ or study other multiplicative properties of it, we need to control its distribution in arithmetic progressions. It is known that when the set $\CD$ has ``nice'' Fourier-analytic properties, then $W_{g,\CD}$ is well-distributed among the different congruence classes of very large moduli. Results of this form has a long history, starting with the work of Erd\H os, Mauduit and S\'ark\"ozy \cite{EMS}, and continuing with the work of Dartyge and Mauduit \cite{DM1}, and Konyagin \cite{konyagin2}. An important breakthrough was accomplished by Dartyge and Mauduit \cite{DM2}, who demonstrated that for appropriate choices of $g$ and $\CD$, the set $W_{g,\CD}\cap[1,x]$ is well-distributed modulo {\it most} numbers $q\le x^{\theta}$ with $\theta>1/2$. Breaking this ``square-root barrier'' is crucial for us, as condition (b) of Proposition~\ref{from distr to irr} indicates. Their results were further improved recently by Maynard \cite{maynard1,maynard2}, who showed that $W_{10,\CD}$ contains infinitely many primes as long as $\#\CD=9$.

Our situation is very similar, so the arguments of Dartyge-Mauduit and Maynard should transfer to our setting. As a matter of fact, Moses \cite{moses} and Porritt \cite{porritt}  have already carried out, independently,  Maynard's argument \cite{maynard1} in the finite field setting: they counted irreducible polynomials over $\F_q$, $q$ being a prime power, all of whose coefficients lie is some restricted subset of $\F_q$ (their argument allows for the omission of up to $\sqrt{q}/2$ coefficients). By adapting their ideas, we can control the quantity $\Delta_\CP(n;m)$ for rather general measures $\mu_j$, as long as their Fourier transform is ``tame''. To state the exact type of condition we must impose, we need to introduce some notation.

\medskip

Given a probability measure $\mu$ on $\Z$, we define its Fourier transform by
\[
\hat{\mu}(\theta):=\sum_{a\in\Z} \mu(a) e(\theta a)
\]
with the usual convention $e(x)=e^{2\pi ix}$. Our main result on $\Delta_\CP(n;m)$ is the following one.

\begin{prop}\label{distr}  Let $\CP=\{p_1,\dots,p_r\}$ be a set of distinct primes and set $P=p_1\cdots p_r$. In addition, consider an integer $n\ge P^4$ and a sequence $\mu_0,\mu_1,\dots,\mu_{n-1}$ of probability measures on the integers for which there are numbers $\gamma\ge1/2$ and $s\in\N\cap[1,n^{1/100}]$ such that
\[
\sum_{k\in\Z/Q\Z} |\hat{\mu}_j(k/Q+\ell/R)|^s \le \big(1-n^{-1/10}\big)\cdot Q^{1-\gamma}
\]
for all $j=1,\dots,n-1$ and all integers $Q,R,\ell$ such that $QR=P$ and $Q>1$. Then, we have
\[
\Delta_\CP\big(n;\gamma n/s+n^{0.88}\big) =O_r(e^{-n^{1/10}} ).
\]
\end{prop}

\begin{rem}\label{rem:gamma}
(a) In the proof we use in a crucial way $\gamma\ge1/2$ (see the last lines of the proof of Lemma \ref{discrete L^1} below). On the other hand, if $\gamma$ satisfies the conditions of Proposition \ref{distr}, it must be strictly less than 1 because $\hat{\mu}(0)=1$. 
	
(b) When the measures $\mu_j$ are all the same, the conclusion of Proposition~\ref{distr} holds when $\sum_{k=0}^{Q-1}|\hat{\mu}(k/Q+\ell/R)|^s<Q^{1-\gamma}$ for all $Q,R,\ell$ as above and $n$ sufficiently large.
\end{rem}

Proposition~\ref{distr} will be proved in Part~\ref{part-distr} of the paper.

\subsection{A master theorem} 
Combining Propositions~\ref{from distr to irr} and~\ref{distr}, we establish the following general result, from which we will deduce Theorems~\ref{main thm}-\ref{main thm generic} in \S\ref{proofs}.

\begin{thm}\label{master thm}
	Let $\mu_0,\mu_1,\dots,\mu_{n-1}$ be a sequence of probability measures on the integers satisfying the following conditions:
	\begin{enumerate}
		\item (support not too large) $\supp(\mu_j)\subseteq[-\exp(n^{1/3}),\exp(n^{1/3})]$ for all $j\ge0$;
		\item (controlled Fourier transform modulo four primes) there is an integer $P\le n^{1/4}$ that is the product of four distinct primes, and numbers $\gamma\ge1/2$ and $s\in\N\cap[1,n^{1/20000}/4]$ such that
		\[
		\sum_{k\in\Z/Q\Z} |\hat{\mu}_j(k/Q+\ell/R)|^s \le \big(1-n^{-1/10}\big)\cdot Q^{1-\gamma}
		\]
		for all $j=1,2,\dots,n-1$ and all integers $Q,R,\ell$ with $QR=P$ and $Q>1$.
\end{enumerate}
Assume further that $\supp(\mu_0)\neq\{0\}$ and let $\theta=\gamma/s$. Then, there are absolute constants $c,C_1>0$ such that
	\[
\P_{\CM(n)}\Big(A(T)\ \text{has no divisors of degree}\ \le \theta n\,\Big| a_0\neq0\Big)\le C_1n^{-c}  .
	\]
\end{thm}

\begin{proof} Without loss of generality, we may replace $\mu_0$ by the conditional measure $\mu_0(\ \cdot\ |a_0\neq0)$. In particular, we have that $a_0\neq0$ with probability 1. In addition, we may assume that $n$ is sufficiently large; otherwise, the result is trivial by adjusting the constant $C_1$.
	
Condition (a) of Proposition~\ref{from distr to irr} holds by condition (a) above. By condition (b), we may apply Proposition~\ref{distr} which then implies that condition (b) of Proposition~\ref{from distr to irr} holds true  (with $\eps= 1/100$ and $\theta_{\textrm{Proposition \ref{distr}}}=\min\{\theta,\frac12\}$).
Next, we show a strong form of condition (c) of Proposition~\ref{from distr to irr}.

For any $j$, any $Q|P$ with $Q>1$, and any $a\in\Z/Q\Z$, we use Fourier inversion to deduce that
\[
\sum_{n\equiv a\mod Q}\mu_j(n) = \sum_{n\in\Z} \mu_j(n)\cdot \frac{1}{Q}\sum_{k\mod Q}e(k(n-a)/Q)
= \frac{1}{Q}\sum_{k\mod Q}e(-ka/Q)\hat{\mu}_j(k/Q).
\]
Taking absolute values, applying the triangle inequality, and then H\"older's inequality, we find that
\begin{equation}
  \label{eq:FT-APs}
  \begin{aligned}
\sum_{n\equiv a\mod Q}\mu_j(n)
	&\le \frac{1}{Q}\sum_{k\mod Q}|\hat{\mu}_j(k/Q)| 
	\le \bigg(\frac{1}{Q}\sum_{k\mod Q}|\hat{\mu}_j(k/Q)|^s\bigg)^{\frac{1}{s}}\\
	&\le Q^{-\gamma/s}\le 2^{-\frac{1}{2s}} \le 1-\frac{1}{4s}
  \end{aligned}
\end{equation}
since $\gamma\ge1/2$, $Q\ge2$, and $e^{-x}\le 1-x/\log 4$ for $0\le x\le (\log 2)/2$. Recalling that $s\le n^{1/20000}/4$, we deduce condition (c) of Proposition~\ref{from distr to irr} with $\eps=1/100$.

In conclusion, we may apply Proposition~\ref{from distr to irr} to find that
\[
\P_{\CM(n)}\big(A(T)\ \text{has a divisor of degree}\ \le\min\{\theta,\tfrac12\} n\big) \le Cn^{-c}
\]
for some absolute constants $c,C>0$, where we used that the condition $a_0\neq0$ holds with probability 1. But if $\theta>\tfrac 12$, then any polynomial with no divisor of degree $\le n/2$ is irreducible, and thus it has no divisors of degree smaller than $\theta n$. This completes the proof.
\end{proof}

\begin{rem}\label{rem:theta}
(a) As per Remark~\ref{rem:gamma}, we have $1/2\le\gamma<1$. Hence, $\theta\ge 1/2$ if $s=1$, and $\theta<1/2$ otherwise. Thus we can only obtain irreducibility with high probability when the Fourier transform of the measures $\mu_j$ at some Farey fractions $a/q$ is bit smaller than $1/\sqrt{q}$, thus excluding the measure given by \eqref{squares}. We will return to this point in \S\ref{sec:pf s-th powers} (see Remark~\ref{limitations} in the end of that section).
	
(b) We can say more things about how the optimal value of $\theta$ varies with $s$. Given a real number $s\ge1$, let us define $\gamma(s)$ to be the largest number $\gamma\in[0,1]$ such that
	\[
	\max_{0\le j<n}\max_{\substack{QR=P,\, Q>1}}\max_{\ell\in\Z} 
	\frac{1}{Q^{1-\gamma}}\sum_{k\in\Z/Q\Z} |\hat{\mu}_j(k/Q+\ell/R)|^s =1.
	\]
	Such a number always exists since the left-hand side is $\ge1$ when $\gamma=1$, and it is $\le1$ when $\gamma=0$. If $1/u+1/v=1$ with $u,v>1$, then H\"older's inequality implies that 
	\als{
		\sum_{k\in\Z/Q\Z} |\hat{\mu}_j(k/Q+\ell/R)|^s	
		&\le \bigg(\sum_{k\in\Z/Q\Z} |\hat{\mu}_j(k/Q+\ell/R)|^{us}\bigg)^{\frac{1}{u}}\bigg(\sum_{k\in\Z/Q\Z} |\hat{\mu}_j(k/Q+\ell/R)|^{vs}\bigg)^{\frac{1}{v}} \\
		&\le Q^{1-\gamma(us)/u-\gamma(vs)/v}
	}
	for all integers $Q,R,\ell,j$ with $QR=P$, $Q>1$ and $0\le j<n$. Hence, $\gamma(s)\ge\gamma(us)/u+\gamma(vs)/v$. If we then set $\theta(s):=\gamma(s)/s$, then we deduce that
	\[
	\theta(s)\ge \theta(us)+\theta(us/(u-1)) 
	\]
	for all $s\ge1$ and all $u>1$. In particular, $\theta$ is a decreasing function such that $\theta(s)\ge 2\theta(2s)$.
\end{rem}

\subsection{From irreducibility to Galois groups}\label{sec:galois intro}  Once we establish that our random polynomial $A(T)$ is irreducible almost surely, we may apply finite group theory to prove that its Galois group must be large in the sense that it contains the alternating group $\CA_n$. The main technical result we need is stated below. In its statement and throughout the paper, we write $\CG_A$ for the Galois group of the polynomial $A(T)$, which we view as a subgroup of $\CS_n$. 

\begin{prop}
	\label{from distr to galois}
	Let $\mu_0,\mu_1,\dots,\mu_{n-1}$ be a sequence of probability measures on the integers for which there is a prime $p$ and a real number $\eps>0$ such that
	\[
	\Delta_p(n;n/2+n^{\lambda_0+\eps})\le n^{-10} 
	\qquad\text{and}\qquad
	\sup_{1\le j<n}\sum_{a\equiv 0\mod p}\mu_j(a)\le 1-1/(\log n)^2.
	\]
	Then there exist some constants $c=c(\eps)>0$ and $C=C(\eps)>0$ such that
	\[
	\P_{\CM(n)}\Big(A(T)\ \text{is irreducible and}\ \CG_A\notin \{\CA_n,\CS_n\} \Big) \le Cn^{-c} .
	\]
\end{prop}

\begin{rem*}
Notice that, unlike Proposition~\ref{from distr to irr}, where we need to control the joint distribution of our random polynomial modulo four distinct primes, Proposition~\ref{from distr to galois} requires input from the reduction of our polynomial modulo a single prime. We formulated Proposition \ref{from distr to galois} for $\theta=\frac 12$ for simplicity. It is also possible to prove a result for smaller $\theta$, but the list of possibilities for the Galois group would become larger.
\end{rem*}

The proof of Proposition~\ref{from distr to galois} goes roughly as follows:
\begin{itemize}
	\item Let $p$ be a prime as in the statement of Proposition~\ref{from distr to galois}, so that if we choose a polynomial $A$ randomly according to the measure $\P_{\CM(n)}$, then its reduction $A_p$ is approximately uniformly distributed in $\CM_p(n)$. 
	\item Each polynomial $f \in \CM_p(n)$ induces a partition $\tau_f\vdash n$, obtained simply by gathering the degrees of  the irreducible factors of $f$. \item The set of partitions of $n$, denoted by $\Pi_n$, is in one-to-one correspondence with the set of conjugacy classes of $\CS_n$. Thus, the uniform measure on $\CS_n$ induces a measure on $\Pi_n$. Let us denote it by $\mu_{\text{unif}}$. 
	\item If $f$ is uniformly distributed in $\CM_p(n)$, then $\tau_f$ is distributed in $\Pi_n$ according to $\mu_{\text{unif}}$, except for factors of small degrees that have slightly distorted distribution.
	\item If $A$ is randomly chosen according to $\P_{\CM(n)}$ satisfying the hypotheses of Proposition~\ref{from distr to galois}, then $f=A_p$ is approximately uniformly distributed, so the distribution of $\tau_f$ in $\Pi_n$ should approximate $\mu_{\text{unif}}$.
	\item Given a polynomial $f \in \CM_p(n)$, the action of the Frobenius automorphism $\alpha\mapsto\alpha^p$ on its roots  induces a permutation whose cycle type is ``close'' to $\tau_f$ (in a precise technical sense that we will specify later). 
	Thus, if $f=A_p$ is as above and we lift the Frobenius to an automorphism of the splitting field of $A$ over $\mathbb{Q}$, then we get a conjugacy class $[\sigma_f]$ in the Galois group of $A$ that is ``close'' to a partition sampled according  to the measure $\mu_{\text{unif}}$, with a small distortion in the distribution of $[\sigma_f]$ due to ramification. 
	\item Let $\CE$ be the event that $A$ is irreducible and its Galois group is different from $\CA_n$ and $\CS_n$. We want to show that $\CE$ occurs with small probability. Recall that the irreducibility of $A$ is equivalent to its Galois group being transitive. On the other hand, \L uczak and Pyber \cite{LP93} showed that, with high probability as $n\to\infty$, a uniform random permutation of $\CS_n$ does not lie in a transitive group other than $\CA_n$ or $\CS_n$. We will show a generalization of this result: if $\tau$ is a random partition of $n$ whose distribution is {\it approximately} $\mu_{\text{unif}}$, then with high probability there is no permutation $\sigma\in\CS_n$ that lies in a transitive subgroup of $\CS_n$ other than $\CA_n$ or $\CS_n$ itself, and whose cycle type is ``close'' to $\tau$. We may thus conclude that the event $\CE$ occurs with small probability. 
\end{itemize}

In order to turn the above sketch into an actual proof, we must address two points. First, we must quantify the statement that if $A$ is sampled randomly, then the partition $\tau_{A_p}$ has a distribution that approximates $\mu_{\text{unif}}$. It turns out that we need a very weak statement of this sort, which we can then insert into the argument of \L uczak-Pyber and establish an appropriate generalization of their result that allows us to complete the proof of Proposition~\ref{from distr to galois}. The details will be given in Part~\ref{part-galois} of the paper. 

We conclude this subsection by using Proposition~\ref{from distr to galois} to establish a general theorem for the Galois group of a random polynomial, from which we will deduce Theorem~\ref{thm:GGs} in \S\ref{proofs-galois}.

\begin{thm}\label{master thm galois}
	Let $\mu_0,\mu_1,\dots,\mu_{n-1}$ be a sequence of probability measures on the integers satisfying the following conditions:
\begin{enumerate}
	\item (support not too large) $\supp(\mu_j)\subseteq[-\exp(n^{1/3}),\exp(n^{1/3})]$ for all $j$;
	\item (controlled Fourier transform modulo four primes) there is an integer $P\le n^{1/4}$ such that
	\begin{equation}\label{FT-L1-condition}
	\sum_{k\in\Z/Q\Z} |\hat{\mu}_j(k/Q+\ell/R)| \le \big(1-n^{-1/10}\big)\cdot Q^{1/2}
	\end{equation}
	for all $j=0,1,\dots,n-1$ and all integers $Q,R,\ell$ with $QR=P$ and $Q>1$.
\end{enumerate}
Then there exists an absolute constant $c>0$ such that
\[
\P_{A\in \CM(n)}\Big(\CG_A \in\{\CA_n,\CS_n\} \, \Big|\, a_0\neq0 \Big) = 1- O(n^{-c}) .
\]
\end{thm}

\begin{proof} We may assume that $n$ is sufficiently large. As in the proof of Theorem~\ref{master thm} when $s=1$ and $\gamma=1/2$, we note that the assumption that \eqref{FT-L1-condition} holds implies that $\Delta_p(n;n/2+n^{\lambda_0+1/100})\le n^{-10}$, $\P_{A\in\CM(n)}(a_0\neq0)\ge1/4$ and $\sum_{a\equiv0\mod p}\mu_j(a) \le 3/4$ for $0\le j<n$. Hence, Theorem~\ref{master thm} implies that a random polynomial $A\in\CM(n)$ with $a_0\neq0$ has no divisors of degree $\le n/2$ (and thus is irreducible) with probability $1-O(n^{-c_1})$, for some $c_1>0$. Combining this result with Proposition~\ref{from distr to galois} completes the proof of Theorem~\ref{master thm galois}.
\end{proof}

\subsection{Summary} The following diagram sums up the discussion of \S~\ref{sec:outline}.

\medskip

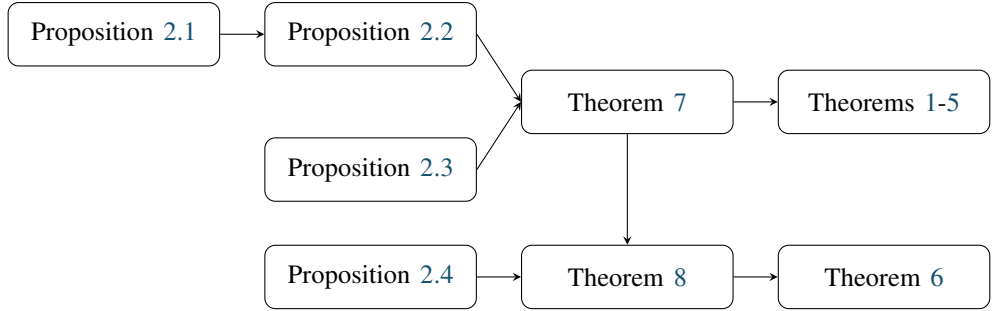
\begin{center}
	\makebox[\textwidth]{\parbox{1.5\textwidth}{
			\begin{center}
				\tikzstyle{interface}=[draw, text width=6em,
				text centered, minimum height=2.0em]
				\tikzstyle{daemon}=[draw, text width=6em,
				minimum height=2em, text centered, rounded corners]
				\tikzstyle{lemma}=[draw, text width=5em,
				minimum height=1.5em, text centered, rounded corners]
				\tikzstyle{dots} = [above, text width=6em, text centered]
				\tikzstyle{wa} = [daemon, text width=6em,
				minimum height=2em, rounded corners]
				\tikzstyle{ur}=[draw, text centered, minimum height=0.01em]
				\def\blockdist{1.3}
				\def\edgedist{0.}
				\begin{tikzpicture}
				\node (thm7)[daemon]  {\footnotesize Theorem~\ref{master thm}};
				
				\path (thm7.east)+(2,0) node(thm-many)[daemon]  {\footnotesize Theorems~\ref{main thm}-\ref{main thm generic} };
				
				\path (thm7.west)+(-2,0.9) node (d1)[daemon] {\footnotesize Proposition~\ref{from distr to irr}};
				\path (thm7.west)+(-2,-0.9) node (d2)[daemon] {\footnotesize Proposition~\ref{distr}};
				\path (d1.west)+(-2,0) node (d3)[daemon] {\footnotesize Proposition~\ref{large irr factors}};
				\path (d2.south)+(0,-1) node (d4)[daemon] {\footnotesize Proposition~\ref{from distr to galois}};
				
				\path (d4.east)+(2,0) node (thm8)[daemon] {\footnotesize Theorem~\ref{master thm galois}};
				\path (thm8.east)+(2,0) node(thm6)[daemon]  {\footnotesize Theorem~\ref{thm:GGs} };
				\path [draw, ->,>=stealth] (d3.east) -- node [above] {} (d1.west) ;
				\path [draw, ->,>=stealth] (d1.east) -- node [above] {} (thm7.west) ;
				\path [draw, ->,>=stealth] (d2.east) -- node [above] {} (thm7.west) ;
				\path [draw, ->,>=stealth] (d4.east) -- node [above] {} (thm8.west) ;
				\path [draw, ->,>=stealth] (thm7.east) -- node [above] {} (thm-many.west) ;
				\path [draw, ->,>=stealth] (thm8.east) -- node [above] {} (thm6.west) ;
				\path [draw, ->,>=stealth] (thm7.south) -- node [above] {} (thm8.north) ;
				\end{tikzpicture}
	\end{center}}}
\end{center}

\medskip

We have already explained how to deduce Theorem~\ref{master thm} from Propositions~\ref{from distr to irr} and~\ref{distr}, as well as Theorem~\ref{master thm galois} from Proposition~\ref{from distr to galois}. We will show how to go from Theorems~\ref{master thm} and~\ref{master thm galois} to Theorems~\ref{main thm}-\ref{thm:GGs} in the next section. Finally, we will prove Proposition~\ref{large irr factors} in Section~\ref{large irr factors pf}, Proposition~\ref{from distr to irr} in Section~\ref{anatomy}, Proposition~\ref{distr} in Part~\ref{part-distr}, and Proposition~\ref{from distr to galois} in Part~\ref{part-galois}.

\section{Deduction of Theorems~\ref{main thm}-\ref{thm:GGs} from Theorems~\ref{master thm} and~\ref{master thm galois}}\label{proofs}

Let us now explain how to use Theorems~\ref{master thm} and~\ref{master thm galois} to deduce Theorems~\ref{main thm}-\ref{thm:GGs}. Note that in all these theorems the measures $\mu_j$ are the same measure $\mu$. 

Let 
\[
	\alpha(s,\gamma;P) := 
	\max_{\substack{QR=P\\ Q>1}} 
	\max_{\ell \in \mathbb{Z}} 
	\frac{1}{Q^{1-\gamma}} 
	\sum_{k\in \mathbb{Z}/Q\mathbb{Z}} | \hat{\mu}(k/Q + \ell/R)|^s .
\]
In most cases, we shall apply Theorems~\ref{master thm} and~\ref{master thm galois} with $s=1$ and $\gamma=1/2$. We thus adopt the notation
\[
\alpha(P):= \alpha(1,1/2;P)= \max_{\substack{QR=P\\ Q>1}} \max_{\ell\in\Z/R\Z} \frac{1}{Q^{1/2}} \sum_{k\in\Z/Q\Z} |\hat{\mu}(k/Q+\ell/R)|. 
\]
It is useful to note the simple bound
\begin{equation}\label{alpha-trivialbound}
\alpha(P)\le \frac{1}{\sqrt{\min\{p|P\}}} \sum_{k\mod P} |\hat{\mu}(k/P)|
\end{equation}
for square-free integers $P$, as it can be easily seen using the Chinese Remainder Theorem.

\subsection{Proof of Theorem~\ref{main thm}(a)}\label{sec:main thm pf a} Our assumption that $\CN\subset[-H,H]$ and that $n\ge (\log H)^3$ implies that condition (a) of Theorem~\ref{master thm} is satisfied. We will now check that $\alpha(210)<1$ for all $N\ge 35$ (210 being the smallest number which is the product of 4 distinct primes; it turns out that the freedom to choose the primes is not useful for Theorem~\ref{main thm}, though it certainly is useful for our other results). We will give a standard proof that works for $N\ge33,\!730$, and a computer-assisted proof for $N\in [35,33729]$.

We start with a bound on $\hat\mu$. Any probability measure satisfies $\hat\mu(0)=1$, and for $\mu$ the uniform measure on a set of $N$ consecutive integers, and for any $k\in\{1,2,\dotsc,P-1\}$ we may calculate
\[
|\hat\mu(k/P)| = \frac{1}{N}\bigg|\sum_{j=1}^N e\Big(\frac{jk}{P}\Big)\bigg|=\bigg|\frac{e(Nk/P)-1}{N(1-e(k/P))}\bigg|\;.
\]
The term $|1-e(k/P)|$ is minimised at $k=1$ and at $k= P-1$. Since $|1-e(1/P)|=2\sin(\pi/P)$, we get that $|\hat\mu(k/P)|\le 1/[\sin(\pi/P)N]$ when $1\le k\le P-1$, and thus
\[
\sum_{k\mod P}|\hat\mu(k/P)|\le 1+\frac{P-1}{N\sin(\pi/P)}.
\]
Together with \eqref{alpha-trivialbound}, and our choice of $P=210$, this implies
\[
\alpha(210)\le \frac{1}{\sqrt{2}}\bigg(1+\frac{209}{N\sin(\pi/210)}\bigg)<1
\]
for $N\ge 33,\!730$.  Finally, when $N\in[35,33729]$, we may check using a computer that $\alpha(210)<1$.\footnote{We used Mathematica\textregistered for this computation. For each $H\in\N$, consider the function \[\textsf{FT[H\_] := Table[If[k == 0 || k == 210, 1, 
		N[Abs[Sin[H*Pi*k/210]/(H*Sin[Pi*k/210])]]], \{k, 0, 210 + 209\}]}.\] This creates a table of all values of the $\hat{\mu}(k/210)$ with $\mu$ the uniform measure on $\{1,2,\dots,H\}$. Given such a table \textsf{F} and an integer \textsf{Q} dividing $210$, we further define 
	\[\textsf{  L1[F\_, Q\_] := Max[ Table [ Sum[ F[[1 + k*210/Q + m*Q]], \{k, 0, Q - 1\}], \{m, 0, 210/Q - 1\}    ] ] / Sqrt[Q] } . \]
	This will calculate $\max_{m\mod R} \sum_{k\mod Q} |\hat{\mu}(k/Q+m/R)|$ with $QR=210$ by taking \textsf{F=FT[N]}. It is important to define \textsf{L1} this way, as this forces \textsf{F} to be precalculated when evaluating \textsf{L1}. Lastly, we define 
	\[
	\textsf{alpha[F\_] := Max[ Table [ L1[F, Divisors[210][[n]]], \{n, 2, 16\} ] ] }
	\]
	and we run 
	\[
	\textsf{Do[Print[ \{N, alpha[FT[N]] \} ], \{N, 35, 33729\} ] }
	\]
	to verify that $\alpha(210)<1$ when $N\in[35,33729]$.
}
The calculation of $\alpha(210)$ involves maximising over finite sets and there are no issues of numerical stability.

In conclusion, we may apply Theorem~\ref{master thm} with $\gamma=1/2$ and $s=1$. This completes the proof of  Theorem~\ref{main thm}(a), since a polynomial of degree $n$ with no divisors of degree $\le n/2$ must be irreducible.

\medskip

We conclude this section by giving a complementary argument in the case when $n\le (\log H)^3$ that builds on a classical method (see \cite[Exer. 266, p. 156, 365]{PS}). This lemma is not used anywhere in the paper, but we think it complements Theorem \ref{main thm} somewhat, leaving only the case that $N$ is small and $H$ is very large.

\begin{lem}
\label{lem:Rivin} Let $\CN$ be a set of $N$ consecutive integers contained in  $[-N^{\log\log(100N)},N^{\log\log(100N)}]$. If $n\le N^{1/200}$ and $\mu_j$ is the uniform measure on $\CN$ for each $j$, then we have
\[
\P_{A\in\CM(n)}\big(A\textrm{ is reducible}\,\big|\,a_0\neq0\big)\ll N^{-0.3}. 
\]
\end{lem}
\begin{proof} Let $\CN_0=\CN\setminus\{0\}$ and $N_0=\#\CN_0$. The number of monic polynomials of degree $n$ with integer coefficients in $\CN$ whose constant coefficient is non-zero is $N_0N^{n-1}$. If $A=BC$ with $B$ and $C$ monic polynomials over $\Z$ of degree $<n$, then the constant coefficients of $A,B$ and $C$, which we denote by $a_0,b_0$ and $c_0$, respectively, must satisfy $a_0=b_0c_0$. The number of possibilities for $b_0$ and $c_0$ is no more than
	\[
	2\sum_{a_0\in\CN_0}\tau(a_0) \le 2N_0T,
	\quad\text{where}\quad
	T:=\max_{a_0\in\CN_0}\tau(a_0). 
	\]
We know that $\tau(a)\le \exp((\log2+o(1))\log a/\log\log a)$ as $a\to\infty$ (e.g., see \cite[\S18.1, Theorem 317]{HW}), so that $T\ll N^{0.695}$ if $\CN\subset[-N^{\log\log(3N)},N^{\log\log(3N)}]$.
	
Let us now fix a choice of $b_0$ and $c_0$ and reduce the equation $A=BC$ modulo $N$.  The number of possibilities for $B$ mod $N$ given $b_0$ and $\deg B=k$ is $N^{k-1}$, and ditto for $C$. Thus, given $b_0$ and $c_0$, we get that the number of possibilities for the couple $(B,C)$ mod $N$ is at most
	\[
	\sum_{k=1}^{n-1}N^{k-1}N^{n-k-1}=(n-1)N^{n-2} .
	\]
In addition, if we are given $B$ and $C$ mod $N$, then there is a unique polynomial $A$ that equals $BC$ modulo $N$ and whose coefficients lie in $\CN$. In conclusion, for each given choice of $b_0$ and $c_0$, the number of possibilities for $A$ is $\le (n-1)N^{n-2}$. Since the number of  choices for $b_0$ and $c_0$ is $\le 2N_0T$, the proof is complete.
\end{proof}

\subsection{Proof of Theorem~\ref{main thm}(b)}\label{sec:main thm pf} Let us first remark that $\alpha(s,\gamma;P)$ does not depend on which $N$ consecutive integers are chosen. Different choices correspond to multiplying $\hat{\mu}$ by a unimodular value and preserve the value of $\alpha$. When $2\le N\le 34$, a numerical calculation reveals that $\alpha(210)>1$ (and larger values of $P$ are even worse). Hence, we cannot apply Theorem~\ref{master thm} with $s=1$ and $\gamma=1/2$ in order to deduce that a polynomial $A\in\Y_\CN(n)$ is irreducible with high probability. However, we may easily check that $\alpha(s,\gamma;210)<1$ for appropriate choices of $s\ge2$ and $\gamma\ge1/2$ as listed in the following table:

\medskip

\begin{center}
\begin{tabular}{| c | c | c | c| } 
	\hline
	$N$ & $s$ & $\gamma$ & $\gamma/s$ \\
	\hline
	2&134&0.50057&0.003736\\
	3&50&0.50045&0.010009\\
	4&27&0.502094&0.018596\\
	5&17&0.503402&0.029612\\
	6&12&0.50681&0.042234\\
	7&9&0.51024&0.056693\\
	8&7&0.51308&0.073297\\
	9&5&0.505506&0.101101\\
	10&4&0.50552&0.12638\\
	11&4&0.52351&0.13088\\
	12&3&0.51283&0.17094\\
	\hline
\end{tabular}
\hspace{0.3em}
\begin{tabular}{| c | c | c | c | } 
	\hline
	$N$ & $s$ & $\gamma$ & $\gamma/s$ \\
	\hline
	13&3&0.52792&0.17597\\
	14&3&0.54188&0.18063\\
	15&2&0.50645&0.25322\\
	16&2&0.51852&0.25926\\
	17&2&0.52986&0.26493\\
	18&2&0.54055&0.27027\\
	19&2&0.55066&0.27533\\
	20&2&0.56025&0.28013\\
	21&2&0.56938&0.28469\\
	22&2&0.57808&0.28904\\
	23&2&0.58639&0.2932\\
	\hline
\end{tabular}
\hspace{0.3em}
\begin{tabular}{| c | c | c | c| } 
	\hline
	$N$ & $s$ & $\gamma$ & $\gamma/s$ \\
	\hline
	24&2&0.59435&0.29718\\
25&2&0.60198&0.30099\\
26&2&0.60932&0.30466\\
27&2&0.61638&0.30819\\
28&2&0.62318&0.31159\\
29&2&0.62974&0.31487\\
30&2&0.63608&0.31804\\
31&2&0.64221&0.321107\\
32&2&0.64815&0.32408\\
33&2&0.65391&0.32695\\
34&2&0.65949&0.32975\\
		\hline
\end{tabular}
\end{center}

\medskip

\noindent 
Hence, Theorem~\ref{master thm} implies that, with probability $\ge 1-n^{-c}$, a polynomial $A$ chosen from $\Y_\CN(n)$ uniformly at random has only irreducible factors of degree $\ge \theta n$ with $\theta=\gamma/s$. In order to pass from this result to a proof of Theorem~\ref{main thm}(b), we use an argument due to Konyagin.

\begin{lem}\label{konyagin-trick-1} Let $n\in\N$, $\theta\in[0,1/2]$, $N\in\Z_{\ge2}$ and $d\in\N$ such that there is at least one prime $p$ that divides $N$ but not $d$. If $\CN$ is an arithmetic progression of step $d$ and $\#\CN=N$, then
\[
\frac{\#\{A\in\Y_\CN(n): \mbox{$A$ has no divisors of degree in $[\theta n,n/2]$}\}}{\#\Y_\CN(n)}\ge -\log(1-\theta)+O(1/n). 
\]
\end{lem}

\begin{proof} Without loss of generality, we may assume that $\theta\ge3/n$; otherwise, the result is trivial since the error term is bigger than the main term. 
	
Let $p$ be as above. If $A$ is uniformly distributed in the set of degree $n$ monic polynomials with coefficients in $\CN$, then its reduction $A_p$ mod $p$ is uniformly distributed in $\CM_p(n)$. Since we are actually sampling $A$ from $\Y_\CN(n)$, there is a small complication regarding the distribution of its constant coefficient mod $p$. Indeed, if $\P$ denotes the uniform probability measure on $\Y_\CN(n)$, then
\[
\P(a_0\equiv b\mod p)=\delta_b:= 
	\begin{cases} 
	1/p&\text{if}\ 0\notin\CN,\\
	(N/p-1)/(N-1)&\text{if}\ 0\in\CN\ \text{and}\ b\equiv0\mod p,\\
	N/(pN-p)&\text{otherwise}.
	\end{cases}
\]	
Hence, if $B\in\CM_p(n)$ has constant coefficient $b$, then $\P_{A\in\Y_\CN(n)}(A_p=B)=\delta_b/p^{n-1}$. 

Now, note that if $A_p$ does not have a divisor of degree in $[\theta n,n/2]$, then neither does $A$. Hence, it suffices to show that
\begin{equation}\label{konyagin:eq1}
	\P_{A\in\Y_\CN(n)}\big(\mbox{$A_p$ has no divisors of degree in $[\theta n,n/2]$}\big) \ge -\log(1-\theta)+O(1/n). 
\end{equation}
Given $a_0\in\F_p$ and $i_0\in\F_p\setminus\{0\}$, let $\mathscr{A}_{a_0,i_0}$ denote the set of polynomials $A_p\in\F_p[T]$ that can be written as $D_pI_p$, where:
\begin{itemize}
	\item $D_p$ is a monic element of $\F_p[T]$ of constant coefficient $a_0i_0^{-1}$ and degree $<\theta n$;
	\item $I_p$ is a monic irreducible element of $\F_p[T]$ of constant coefficient $i_0$ and degree $n-\deg(D_p)$. 
\end{itemize}
Since $\deg(I_p)>n(1-\theta)\ge n/2$, such a representation of $A_p$, if it exists, is unique. Moreover, no $A_p$ of the above form has divisors of degree in $[\theta n,n/2]$. 

Now, we may easily calculate that
\[
\P_{\CM_p(n)}(\mathscr{A}_{a_0,i_0})
	= \frac{\delta_{a_0}}{p^{n-1}} 
		\sum_{0\le m<\theta n} \sum_{\substack{D_p\in\CM_p(m) \\ D_p(0)=a_0i_0^{-1}}} \sum_{\substack{I_p\in \CM_p(n-m) \\ I_p\ \text{irreducible} \\ I_p(0)=i_0}} 1 .
\]
The number of $D_p$ equals $p^{m-1}$, and the number of $I_p$ equals $\frac{p^{n-m}}{(p-1)(n-m)}(1+O(p^{1-(n-m)/2}))$ by \cite[Theorem 4.8]{rosen}. Since $m<\theta n$ and we assumed that $\theta\ge 3/n$, the error term is $O(1/n)$.  Consequently,
\[
\P_{\CM_p(n)}(\mathscr{A}_{a_0,i_0})
 = \frac{\delta_{a_0}}{p-1} \sum_{0\le m<\theta n} \bigg(\frac{1}{n-m}+O(1/n^2)\bigg) 
 = \frac{\delta_{a_0}}{p-1} \big(-\log(1-\theta) +O(1/n)\big),
\]
where we used \cite[Theorem 1.11]{kou}. Since the sets $\mathscr{A}_{a_0,i_0}$ are disjoint, and we also have that $\sum_{a_0\in\F_p}\delta_{a_0}=1$, relation \eqref{konyagin:eq1} follows. This completes the proof of the lemma, and hence also of Theorem \ref{main thm}(b).
\end{proof}

\begin{rem}
When $\theta\le 1/3$ (as is the case when applying Lemma~\ref{konyagin-trick-1} to prove Theorem~\ref{main thm}(b)), it is possible to show that $A_p$ has no divisors of degree in $[\theta n,n/2]$ if, and only if, $A_p=D_pI_p$ with $\deg(D_p)<\theta n$ and $I_p$ irreducible.
\end{rem}

The proof of Lemma~\ref{konyagin-trick-1} has some limitations. For example, it cannot be used when the coefficients are drawn from $\{-1,+1\}$, because this set has two elements that both have the same reduction mod 2. The same problem occurs more generally when $\CN$ is an arithmetic progression of step $d$ that contains $N$ elements, and all prime divisors of $N$ also divide $d$. In these cases, however, we have an alternative argument that follows more closely Konyagin's original idea.

\begin{lem}\label{konyagin-trick-2} 
	Let $n\in\N$, $N\in\Z_{\ge2}$ and $\theta\in[0,1/2]$. If $\CN\subseteq[-H,H]$ is an arithmetic progression such that $\#\CN=N$, then 
	\[
		\frac{\#\{A\in\Y_\CN(n): \mbox{$A$ has no divisors of degree in $[\theta n,n/2]$}\}}{\#\Y_\CN(n)} 
			\ge  -\log(1-\theta)-O\bigg(\frac{\log(nH)}{n^{1/2}\log N}\bigg).
	\]
\end{lem}

We need an auxiliary result:

\begin{lem}\label{lem:mignotte} Let $A(T)$ be  polynomial of degree $n$ all of whose coefficients are in $[-H,H]\cap\Z$. 
If $N\ge2$ and $I(T)$ is an irreducible polynomial over $\Z$ of degree $m$ that divides $A(T)$, then 
	\[
	|I(N)|\le N^m e^{4(1+\sqrt{m})\log(14\sqrt{n}H)}.
	\] 
\end{lem}

\begin{proof} Given a polynomial $f$ with integer coefficients, let $\|f\|_2$ denote the $\ell^2$-norm of its coefficients. Using a result of Mignotte (see Theorem $1'$ in \cite{mignotte} and the remarks below it), we have
	\[
	\|I\|_2\le e^{\sqrt{m}}(m+2\sqrt{m}+2)^{1+\sqrt{m}}\|A\|_2^{1+\sqrt{m}}.
	\]
	Since $\|A\|_2\le H\sqrt{n}$ and $|I(N)|\le (N^{2m}+N^{2m-2}+\cdots)^{1/2}\|I\|_2\le N^m\sqrt{N^2/(N^2-1)}\,\|I\|_2$ by the Cauchy-Schwarz inequality, the lemma follows.
\end{proof}

\begin{proof}[Proof of Lemma~\ref{konyagin-trick-2}] Let us write $\CN=\{a,a+d,\dots,a+(N-1)d\}$, and note that $d,N\le 2H+1$. We recall that $\Y_\CN(n)$ is defined as a set of polynomials whose free coefficient is nonzero. We split it according to the free coefficient, namely, 
given $j\in \CN\setminus\{0\}$, we set $\Y_{\CN,j}(n)=\{A(T)\in\Y_\CN(n):A(0)=j\}$. It suffices to prove that the conclusion of Lemma 
\ref{konyagin-trick-2} holds with $\Y_{\CN,j}(n)$ in place of $\Y_\CN(n)$, for each $j\in\CN\setminus\{0\}$.

The proof revolves around examining the values of $\{A(N):A\in \Y_{\CN,j}(n)\}$. These values form an arithmetic progression of step $dN$, taking each value exactly once. Denote by $x_j\coloneqq j+N^n+a(N^n-N)/(N-1)$ the first element in this arithmetic progression and by $y_j\coloneqq x_j+d(N^n-N)$ the last one. 

Let now $A=I_1\cdots I_k$ denote the decomposition of $A$ into irreducible factors over $\Z$. Assume $A(N)$ has a prime divisor $p$ with
  \[
  p>p_0 \coloneqq N^{(1-\theta)n} \exp\big(4(1+\sqrt{n})\log(14\sqrt{n}H)\big) 
  \]
  Then the prime $p$ divides $I_\ell(N)$ for some $\ell$, and thus $|I_\ell(N)|\ge p>p_0$. Together with Lemma~\ref{lem:mignotte}, this implies that $\deg(I_\ell)>n(1-\theta)$. But then, $A=BI_\ell$ for some $B$ of degree $<\theta n$, and thus $A$ does not have divisors of degree in $[\theta n,n/2]$, which is the property we are interested in.
  Since $\#\CY_{\CN,j}(n)=N^{n-1}$, we conclude that
 \[
	\frac{\#\{A\in\Y_{\CN,j}(n): \mbox{$A$ has no divisors of degree in $[\theta n,n/2]$}\}}{\#\Y_{\CN,j}(n)} \ge \frac{\#\CX_j}{N^{n-1}},
 \]
 where
  \[
  \CX_j\coloneqq\{x_{j}\le kp\le y_{j}:k\in\Z,\ p>p_0\text{ prime},\ kp\equiv x_{j}\mod{dN}\}. 
  \]
  
 To calculate the cardinality of $\CX_j$, we write
\[
	\#\CX_j=
		\sum_{p>p_0}
			\#\{k\in[x_{j}/p,y_{j}/p]\cap\Z:  kp\equiv x_{j}\mod{dN}\} 
\]
(since $p_0>\sqrt{\max\{|y_j|,|x_j|\}}$, any $x\in\CX_j$ is divisible by at most one prime $p>p_0$). Since $p_0> dN$ (in fact, much bigger), we find that $p\nmid dN$ whenever $p>p_0$, and thus the count over $k$'s inside the sum equals $(y_j-x_j)/(pdN)+O(1)$. Let us therefore restrict our attention to $p$ such that $(y_{j}-x_{j})/(pdN)>n$, which will make the $O(1)$ error of smaller order than the main term. Noticing that $(y_{j}-x_{j})/(dN)=N^{n-1}-1$ and summing over such $p$ gives
\[
\#\CX_j\ge \sum_{p_0<p<(N^{n-1}-1)/n} \frac{N^{n-1}-1}{p}-O\bigg(\frac{N^{n-1}}{n}\bigg).
\]
(The big-Oh term not being part of the sum, of course). Using Mertens' theorem \cite[Theorem 3.4(b)]{kou} and the fact that $N\le 2H+1$, we find that
\[
\sum_{p_0<p\le (N^{n-1}-1)/n}\frac{1}{p}=-\log(1-\theta)+O\bigg(\frac{\log(nH)}{\sqrt{n}\log N}\bigg).
\]
Combining the two above estimates using once more that $N\le 2H+1$, we complete the proof of the lemma.
\end{proof}

As a corollary, we can generalise Theorem \ref{main thm} from distribution uniform on $N$ consecutive points to distributions uniform on arithmetic progressions. Here is the precise formulation.

\begin{thm-n}\label{main thm - APs} Let $H\ge1$, $N\in\Z_{\ge2}$ and $d\in\N$. In addition, let $P$ be the product of the four smallest primes that do not divide $d$. Then there are constants $\delta>0$ and $n_0\ge1$  that depend only on $P$ such that the following holds:

If $\CN$ is an arithmetic progression of step $d$ of $N$ elements all contained in $[-H,H]$, and if $n\ge \max\{n_0,(\log H)^3\}$, then
\[
\#\{A\in\Y_\CN(n): \mbox{$A$ is irreducible}\} \ge\delta \#\Y_\CN(n) .
\]	
When $\CN=\{-1,1\}$, we can take $\delta=0.00068053$. 
\end{thm-n}

\begin{proof}
Let us write $\CN=\{a,a+d,\dots,a+(N-1)d\}$, and let $\mu$ denote the uniform measure on $\CN$. As in the proof of Theorem~\ref{main thm}(a), we have
\[
|\hat\mu(k/P)| = \frac{1}{N}\bigg|\sum_{j=0}^{N-1} e\bigg(\frac{(a+dj)k}{P}\bigg)\bigg|=\bigg|\frac{e(dNk/P)-1}{N(1-e(dk/P))}\bigg|\;.
\]
Since $(d,P)=1$ by assumption, the right-hand side is $\le 1/[N\sin(\pi/P)]]\le P/(2N)$ when $P\nmid k$. Hence, if $N\ge P$, then $|\hat{\mu}(k/P)|\le1/2$ for all $k\not\equiv0\mod P$. On the other hand, when $2\le N\le P$, there is some constant $\beta=\beta(P)<1$ such that $|\hat{\mu}(k/P)|\le\beta$ for all $k\not\equiv0\mod P$. Taking $\beta\ge1/2$, as we may, we conclude that $|\hat{\mu}(k/P)|\le\beta$ for all $k\not\equiv0\mod P$ and all $N\ge2$. In conclusion, 
\[
\sum_{k\in\Z/P\Z} |\hat{\mu}(k/P)|^s\le 1+P\beta^s \le 2^{1/4}
\]
as long as $s$ is large enough in terms of $P$. Clearly, this implies that condition (b) of Theorem~\ref{master thm} holds with $\gamma=2/3$ and $n$ sufficient large. Condition (a) also holds by our assumptions on $\CN$ and $n$. Thus, the conclusion of Theorem~\ref{master thm} holds. Combining it with Lemma~\ref{konyagin-trick-2} completes the proof of the theorem for general $\CN$.

Finally, when $\CN=\{-1,+1\}$, note that condition (b) of Theorem~\ref{master thm} is satisfied with $P=3\cdot5\cdot7\cdot 11=1155$, $s=735$ and $\gamma=0.500019700732702471\dots$ We then obtain $\theta=\gamma/s=0.000680298912561\dots$. An application of Lemma~\ref{konyagin-trick-2} completes the proof in this case too.
\end{proof}

\subsection{Proof of Theorem~\ref{main thm s-th powers}}\label{sec:pf s-th powers}
If $p$ is a prime such that $(p-1,d)=1$, then the only $d$-th root of unity mod $p$ is 1 since $(\mathbb{Z}/p\mathbb{Z})^*$ is cyclic of order $p-1$ (see e.g.\ Theorem 1110, \S7.5 in \cite{HW}). As a consequence, the range of the polynomial $f(x)=x^d$ mod $p$ is $\Z/p\Z$.

It is easy to see that there are infinitely many primes such that $(p-1,d)=1$. For instance, we can pick primes in the progression $2\mod d$, which contains infinitely many primes by our assumption that $d$ is odd, using Dirichlet's theorem.

Now, let $P=p_1p_2p_3p_4$, where $p_1<p_2<p_3<p_4$ are the first four primes such that $(p-1,d)=1$. In particular, $p_1=2$. Since the polynomial $f(x)=x^d$ has full range mod each $p_j$, by the Chinese Remainder Theorem it also has full range mod $P$. 

Writing $\mu$ for the uniform measure on $\{k^d:1\le k\le H\}$, we find that
\als{
	\hat{\mu}(\ell/P)
	= \frac{1}{H}\sum_{k=1}^He(k^d\ell/P)
	&= \frac{1}{H}\sum_{a\in\Z/P\Z} e(a^d\ell/P)\cdot \#\{k\le H:k\equiv a\mod P\} .
}
Since $H/P-1< \#\{k\le H:k\equiv a\mod P\} < H/P+1$, we infer that	
\[
|\hat{\mu}(\ell/P)|<
\frac{1}{P}\bigg|\sum_{a\in\Z/P\Z} e(a^d\ell/P)\bigg|+ \frac{P}{H}.
\]
By construction, the residue classes $a^d\mod P$ with $a\in\Z/P\Z$ cover all of $\Z/P\Z$ exactly once. Consequently, the exponential sum on the right hand side of the above inequality vanishes when $P\nmid \ell$. We thus conclude that
\[
|\hat{\mu}(\ell/P)|< P/H \quad\text{when}\ P\nmid \ell.
\]
As a consequence,
\[
\sum_{k\mod P} |\hat{\mu}(k/P)|\le 1+ P(P-1)/H \le 4/3
\]
as long as $H\ge P^2/3$. In particular, $\alpha(P)\le 4/(3\sqrt{2})<1$ by \eqref{alpha-trivialbound} for such $H$. Assuming, as we may, that $n_0\ge P^4$ guarantees that $n\ge P^4$. Since we also supposed that $n\ge (\log H)^3$, we may apply Theorem~\ref{master thm} with $s=1$ and $\gamma=1/2$ and complete the proof of Theorem~\ref{main thm s-th powers}.\qed

\begin{rem}\label{limitations} Let $f(x)\in\Z[x]$ have degree $d\ge1$, and let $\mu$ be the uniform measure on $\CN:=\{f(n):n\in\{1,2,\dots,N\}\cap\Z\}$. For all integers $Q,R,\ell\ge1$, Parseval's identity implies that
	\als{
		\sum_{k\in\Z/Q\Z} |\hat{\mu}(k/Q+\ell/R)|^2 
		&\le Q\mathop{\sum\sum}_{a_1\equiv a_2\mod Q} \mu(a_1)\mu(a_2).
	}
	Now, for any fixed $b\in\Z$, we have
	\[
	\sum_{a\equiv b\mod Q}\mu(a) = \sum_{\substack{k\in\Z/Q\Z\\f(k)\equiv b\mod Q}}\frac{\#\{1\le n\le N:n\equiv k\mod Q\}}{N}
	\le \sum_{\substack{k\in\Z/Q\Z\\f(k)\equiv b\mod Q}} (1/Q+1/N) .
	\]
	If $Q$ is square-free, then the Chinese Remainder Theorem implies that $\#\{k\in\Z/Q\Z:f(k)\equiv b\mod Q\}\le d^{\omega(Q)}$, where $\omega(Q)$ is the number of prime divisors of $Q$. Hence
	\begin{align*}
	  \sum_{k\in\Z/Q\Z} |\hat{\mu}(k/Q+\ell/R)|^2
          &\le Q\sum_{a_1}\mu(a_1)\sum_{\substack{k\in\Z/Q\Z\\f(k)\equiv a_1\mod Q}}(1/Q+1/N) \\
          &\le d^{\omega(Q)}(1+Q/N)\sum_{a_1}\mu(a_1)=d^{\omega(Q)}(1+Q/N).
	\end{align*}
	If, in addition, we assume that $N\ge Q$ and that all prime factors of $Q$ are $\ge d^{4/\eps}$, then the right-hand side is $\le Q^\eps/2$.
	
	In conclusion, if we let $P$ be the product of the four smallest primes $\ge d^{4/\eps}$ and we assume that $N\ge P$, then we may apply Theorem~\ref{master thm} with $s=2$ and $\gamma=1-\eps$. Consequently, with probability $\ge 1-n^{-c}$, an element of $\Y_\CN(n)$ chosen uniformly at random is either irreducible, or it has a divisor of degree in $[n(1-\eps)/2,n/2]$. The latter is a very restrictive condition, and it should only occur for a proportion of polynomials that tends to 0 when $\eps\to0^+$. It is possible to prove the last claim rigorously in some cases. 
	
	For instance, when $f(x)=x^2$, we have $|\hat{\mu}(0)|+|\hat{\mu}(1/2)|\le 1+1_{2\nmid N}/N<\sqrt{2}$ for all $N\ge2$. Hence, Proposition~\ref{distr} applied with $\CP=\{2\}$ implies that $\Delta_2(n;n/2+n^{0.88})\ll \exp(-n^{1/10})$. We may combine this fact with Ford's work \cite{ford} to show that the probability that $A\mod2$ has a divisor of degree in $[n(1-\eps)/2,n/2]$ is $\ll n^{-c}+\eps^c$ for some absolute constant $c>0$. (The case when $\eps=O(1/n)$ follows from Meisner's work \cite{meisner}.) The end result is that if $\CN=\{n^2:1\le n\le N\}$ and we choose $A$ uniformly at random from $\Y_\CN(n)$, then $A$ is irreducible with probability $\ge 1-o_{N,n\to\infty}(1)$. 
\end{rem}

\subsection{Proof of Theorem~\ref{main thm generic}} 
Recall that $\CN$ is a set chosen uniformly at random among all subsets of $[-H,H]\cap\Z$ with $N$ elements. Without loss of generality, we assume throughout that $H\in\N$. We then let $\mu_\CN$ denote the uniform measure on $\CN$ and write $\alpha_\CN$ for the quantity $\alpha(210)$ when $\mu=\mu_\CN$. We claim that $\alpha_\CN\le 3/4$ with probability $1-O(1/\sqrt{N})$. In view of \eqref{alpha-trivialbound} and the fact that $\hat{\mu}_\CN(0)=1$, it suffices to show that $\sum_{k=1}^{209}|\hat{\mu}_\CN(k/210)|\le 3/\sqrt{8}-1$ with probability $1-O(1/\sqrt{N})$. Markov's inequality reduces this claim to proving that
\[
\E\bigg[\sum_{k=1}^{209}|\hat{\mu}_\CN(k/210)|\bigg] \ll \frac{1}{\sqrt{N}} .
\]
The Cauchy-Schwarz inequality reduces the above inequality to proving that
\begin{equation}\label{generic-goal}
\E\bigg[\bigg|\sum_{a\in\CN} e(ak/210)\bigg|^2\bigg] \ll N 
\quad\text{for all}\ k=1,2,\dots,209.
\end{equation}
Let us fix some $k\in\{1,2,\dots,209\}$. Opening the square, we find that
\[
\E\bigg[\bigg|\sum_{a\in\CN} e(ak/210)\bigg|^2\bigg] 
= \sum_{|a_1|,|a_2|\le H} e((a_1-a_2)k/210) \P(a_1,a_2\in\CN).
\]
If $a_1=a_2$, then $\P(a_1,a_2\in\CN)=\P(a_1\in\CN)=\binom{2H}{N-1}\big/\binom{2H+1}{N}=\frac{N}{2H+1}=:\delta_1$; otherwise, $\P(a_1,a_2\in\CN)=\binom{2H-1}{N-2}\big/\binom{2H+1}{N}=\frac{N(N-1)}{2H(2H+1)}=:\delta_2$. We conclude that
\als{
	\E\bigg[\bigg|\sum_{a\in\CN} e(ak/210)\bigg|^2\bigg] 
	&= \delta_2\sum_{|a_1|,|a_2|\le H} e((a_1-a_2)k/210)
	+ (\delta_1-\delta_2)\cdot (2H+1)  \\
	&=\delta_2\bigg|\sum_{|a|\le H} e(ak/210)\bigg|^2+(\delta_1-\delta_2)\cdot (2H+1) \\
	&\ll \delta_2+(\delta_1 - \delta_2)H\ll N 
}
for $k=1,2,\dots,209$. This concludes the proof of \eqref{generic-goal}, and hence of Theorem~\ref{main thm generic}.

\subsection{Proof of Theorem~\ref{general main thm - almost irreducible}}

If we can locate an integer $P$ that is the product of four primes and for which there exists  $\beta<1$ such that $|\hat{\mu}(k/P)|\le \beta$ for all $k\in\{1,\dotsc,P-1\}$, then we argue as in the proof of Theorem~\ref{main thm - APs} to locate $s=s(\beta,P)$ such that $\sum_{k\in\Z/P\Z} |\hat{\mu}(k/P)|\le 2^{1/4}$, which will allow us to apply Theorem~\ref{master thm} with $\gamma=2/3$. In order to locate the necessary $P$, we use the following lemma. 

\begin{lem}\label{lem:l-infinity bound} Let $\eta>0$ and $P\in\Z_{\ge2}$. Assume that $\mu$ is a probability measure on $\Z$ such that 
	\[
	\sum_{a\equiv b\mod p}\mu(a)\le 1-\eta 
	\]
	for all primes $p|P$ and all $b\in\Z$. Then, we have that
	\[
	\max_{k\in\{1,\dotsc,P-1\}}|\hat{\mu}(k/P)|\le 1-4\eta/P^2.
	\]
\end{lem}

\begin{proof} Note that
	\[
	|\hat{\mu}(\theta)|^2
	=\Re\big(\hat\mu(\theta)\overline{\hat\mu(\theta)}\hspace{1pt}\big)
	=\Re\sum_{a,b\in\Z}\mu(a)\overline{\mu(b)}e((a-b)\theta)
	= \sum_{a,b\in\Z}\mu(a)\mu(b)\cos(2\pi(a-b)\theta).
	\]
	Consequently,
	\[
	1-|\hat{\mu}(\theta)|^2
	= \sum_{a,b\in\Z}\mu(a)\mu(b)(1-\cos(2\pi(a-b)\theta))
	\ge    8 \sum_{a,b\in\Z}\mu(a)\mu(b)	
	\cdot \|(a-b)\theta\|^2,
	\]
	where we used the fact that $1-\cos(2\pi y)=2\sin^2(\pi y)\ge 8y^2$ when $|y|\le 1/2$. 
	
	Now, let $\beta=\max\{|\hat{\mu}(k/P)|:k\not\equiv 0\mod P\}$ and let $\theta=k_0/P$ with $k_0\not\equiv0\mod P$ be such that $|\hat{\mu}(k_0/P)|=\beta$. If $k_0/P$ equals $m/Q$ in reduced form, we find that $\|(a-b)\theta\|\ge1/Q$ for all $a\not\equiv b\mod Q$. As a consequence, 
	\[
	1-\beta^2
	\ge    \frac{8}{Q^2} \sum_{\substack{a,b\in\Z \\ a\not\equiv b\mod Q}} \mu(a)\mu(b)  
	=  \frac{8}{Q^2} \sum_{1\le j \le Q} t_j(1-t_j)
	\]
	with
	\[
	t_j =  \sum_{a\equiv j \mod Q}\mu(a) .
	\] 
	If $p$ is any prime dividing $Q$, then $t_j \le \sum_{a\equiv j \mod p}\mu(a) \le 1-\eta$ by assumption. As a consequence,
	\[
	\sum_{1\le j \le Q} t_j(1-t_j) \ge \eta \sum_{1\le j \le Q} t_j = \eta.
	\]
	We conclude that
	\[
	1-\beta\ge \frac{1-\beta^2}{2}  \ge \frac{4}{Q^2} \sum_{1\le j\le Q}t_j(1-t_j) \ge \frac{4\eta}{P^2} ,
	\]
	thus completing the proof of the lemma.
\end{proof}

Let us now see how to use the above lemma to complete the proof of Theorem~\ref{general main thm - almost irreducible}. Recall that $\mu$ is a probability measure on $\Z$ such that $\supp(\mu)\subset[-H,H]$ and $\|\mu\|_\infty\le1-\eps$. We may assume $H$ is sufficiently large, since after increasing $H$ the condition $\supp\mu\subset[-H,H]$ certainly continues to hold, and we then need only adjust the constants $C$ and $c'$ accordingly.

Now, set $x=\log(2H+1)$ and let $\mathscr{P}$ be the set of primes in $(x,3x]$, so that $4+x/\log x\le \#\mathscr{P}\le 3x/\log x$ for $x$ large enough, by the Prime Number Theorem. We claim that there are four primes $p_1,\dots,p_4$ in $\mathscr{P}$ such $P=p_1\cdots p_4$ satisfies the hypothesis of Lemma~\ref{lem:l-infinity bound} with $\eta=\eps\cdot \frac{\log x}{3x}$. To this end, let $\mathscr{Q}$ be the set of primes $p\in\mathscr{P}$ for which there is some congruence class $b_p\mod p$ such that $\sum_{a\equiv b_p\mod p}\mu(a)>1-\eta$. It suffices to prove that $\#\mathscr{Q}< x/\log x$. 

Assume, on the contrary, that $\#\mathscr{Q}\ge x/\log x$ and consider the integer $m=\prod_{p\in\mathscr{Q}}p$. Notice that $m>x^{\#\mathscr{Q}}\ge e^x=2H+1$ by our assumption on $\mathscr{Q}$. On the other hand, the Chinese Remainder Theorem implies that there is some $b\in\Z$ such that $b\mod m$ is the intersection of the residue classes $b_p\mod p$ with $p\in\mathscr{Q}$. Since $\sum_{a\equiv b_p\mod p}\mu(a)>1-\eta$ for each $p\in\mathscr{Q}$, the union bound implies that $\sum_{a\equiv b\mod m}\mu(a) > 1- \#\mathscr{Q}\cdot \eta\ge1-\eps$, where we used that $\#\mathscr{Q}\le \#\mathscr{P}\le 3x/\log x$. However, since $m>2H+1$, there is at most one $a$ that lies in the intersection of the support of $\mu$ with the congruence class $b\mod m$. We have thus arrived at a contradiction. This concludes our proof that $\#\mathscr{Q}\le x/\log x$, and thus that there are four primes $p_1,\dots,p_4$ in $\mathscr{P}$ such $P=p_1\cdots p_4$ satisfies the hypothesis of Lemma~\ref{lem:l-infinity bound} with $\eta=\eps\cdot \frac{\log x}{3x}$.

Now, Lemma~\ref{lem:l-infinity bound} implies that $|\hat{\mu}(k/P)|\le 1-4\eps\cdot (\log x)/(3xP^2)\le 1-\eps\cdot(\log x)/(3^8x^5)$ for all $k\in\Z$ that are not divisible by $P$, where we used that $P\le (3x)^4$. Consequently, 
\[
\sum_{k\in\Z/P\Z}|\hat{\mu}(k/P)|^s \le 1+ (P-1)\cdot \big(1-\eps\cdot(\log x)/(3^8x^5)\big)^s \le 2^{1/4}
\]
by taking $s=\ceil{3^{10}\eps^{-1}x^5}\asymp \eps^{-1}(\log H)^5$, and assuming that $H$ (and thus $x$ and $P$) is large enough. We may now apply Theorem \ref{master thm} with $\gamma=2/3$ and the above value of $s$. The condition $s\le n^{1/20000}/4$ of Theorem \ref{master thm} is satisified since $n^{1/20000}\ge C^{1/20000}\eps^{-1}(\log H)^{5}\ge s$, if $C$ is taken sufficiently large ($C$ from the statement of Theorem \ref{general main thm - almost irreducible}). The condition $n\ge P^4$ holds similarly. This completes the proof of Theorem~\ref{general main thm - almost irreducible}. \qed

\subsection{Proof of Theorem~\ref{general main thm}} Throughout, we fix a measure $\mu$ on the integers and recall that
\[
\alpha(P)= \max_{\substack{QR= P \\ Q>1}}\max_{\ell\in\Z/R\Z} \bigg(\frac{1}{\sqrt{Q}} \sum_{k\in\Z/Q\Z} |\hat{\mu}(k/Q+\ell/R)| \bigg) ,
\]
as well as that $\|\mu\|_2^2=\sum_{a\in\Z}\mu(a)^2$. We will use the large sieve inequality to locate an integer $P$ satisfying $\alpha(P)\le1/2$, so that we may apply Theorem~\ref{master thm}. To this end, given a real number $x\ge2$ and an integer $m\ge0$, let $\CN_m(x)$ denote the set of integers that are the product of $m$ distinct primes from $[x/2,x]$. For future reference, note that 
\eq{N_r properties}{
	\CN_m(x)\subset[(x/2)^m,x^m]
	\quad\text{and}\quad
	\#\CN_m(x)\sim \frac{(x/\log x)^m}{m!2^m}
}
as $x\to\infty$, by a simple application of the Prime Number Theorem \cite[Theorem 8.1]{kou}.

With the above notation, we have the following key estimate. 

\begin{lem}\label{lem:ls} Let $x\ge2$ and $H\ge1$. If $\mu$ is supported on $[-H,H]$, then
	\[
	\sum_{P\in\CN_4(x)} \alpha(P)\ll (x/\log x)^4 \bigg( (x\log x)^2+\Big(\frac{H\log x}x\Big)^{1/2}\bigg)  \|\mu\|_2.
	\]
\end{lem}

\begin{proof} By the large sieve inequality (see \cite[Theorem 25.14]{kou}), we have
	\eq{ls2}{
		\sum_{q\le y}\sum_{a\in(\Z/q\Z)^*} |\hat{\mu}(a/q)|^2
		\ll (y^2+H)\|\mu\|_2^2
	}
	uniformly for all $y\ge1$, where, as usual, $(\Z/q\Z)^*=\{a\in\Z/q\Z: \gcd(a,q)=1\}$. 
	
	Let us now see how to use this bound to prove the lemma. We will be assuming throughout that $x$ is sufficiently large; otherwise, the conclusion of the lemma is trivially true by adjusting the implied constant. 
	
	For brevity, let us write $S$ for the sum in the statement of the lemma.  We then have
	\[
	S\ll 	\sum_{\substack{i+j=4 \\ 1\le i\le 4}} x^{-i/2} 
		\mathop{\sum\sum}_{\substack{Q\in\CN_i(x),\, R\in\CN_j(x) \\ \gcd(Q,R)=1}}
	\max_{\ell\in\Z/R\Z}\sum_{k\in\Z/Q\Z}
        |\hat{\mu}(k/Q+\ell/R)| ,
	\]
	where we used that $Q\asymp x^i$ when $Q\in \CN_i(x)$. Next, let $k_1/Q_1$ and $\ell_1/R_1$ be the fractions $k/Q$ and $\ell/R$, respectively, in reduced form. We then find that
	\[
	\max_{\ell\in\Z/R\Z}\sum_{k\in\Z/Q\Z}
	|\hat{\mu}(k/Q+\ell/R)| 
	\le \sum_{R_1|R}\sum_{Q_1|Q} \max_{\ell_1\in(\Z/R_1\Z)^*} \sum_{k_1\in(\Z/Q_1\Z)^*}
	|\hat{\mu}(k_1/Q_1+\ell_1/R_1)| .
	\]
	Given $Q_1\in\CN_{i_1}(x)$ and $R_1\in\CN_{j_1}(x)$ with $i_1\le i$ and $j_1\le j$, there are $\ll(x/\log x)^{i-i_1}$ choices of $Q$ and $\ll (x/\log x)^{j-j_1}$ choices for $R$. We thus conclude that
	\al{
		S&\ll  \sum_{\substack{i+j=4 \\ 1\le i\le 4}} \sum_{0\le i_1\le i} \sum_{0\le j_1\le j} 
			\frac{(x/\log x)^{4-i_1-j_1}}{x^{i/2}} \nn
		&\quad \times\mathop{\sum\sum}_{\substack{Q_1\in\CN_{i_1}(x),\, R_1\in\CN_{j_1}(x) \\ \gcd(Q_1,R_1)=1}}
		\max_{\ell_1\in(\Z/R_1\Z)^*}  \sum_{k_1\in (\Z/Q_1\Z)^*} |\hat{\mu}(k_1/Q_1+\ell_1/R_1)| .
		\label{eq:ls-beforeCS}
	}
	Using the Cauchy-Schwarz inequality, we find that the sum over $Q_1$ and $R_1$ in \eqref{eq:ls-beforeCS} is
	\[
	\ll (x/\log x)^{(i_1+j_1)/2} \bigg(\mathop{\sum\sum}_{\substack{ Q_1\le x^{i_1},\, R_1\le x^{j_1} \\ \gcd(Q_1,R_1)=1}} 
	\max_{\ell_1\in(\Z/R_1\Z)^*} \bigg(\sum_{k_1 \in (\Z/Q_1\Z)^*} |\hat{\mu}(k_1/Q_1+\ell_1/R_1)| \bigg)^2\bigg)^{1/2} 
	\]
	We majorize $\max_{\ell_1\in(\Z/R_1\Z)^*}$ by $\sum_{\ell_1\in(\Z/R_1\Z)^*}$, and apply again the Cauchy-Schwarz inequality, this time to the sum over $k_1$. We conclude that 
	\als{
		S&\ll  \sum_{\substack{i+j=4 \\ 1\le i\le 4}} \sum_{0\le i_1\le i} \sum_{0\le j_1\le j} 
			\frac{(x/\log x)^{4-i_1-j_1}}{x^{i/2}}
		\cdot (x/\log x)^{(i_1+j_1)/2} \cdot x^{i_1/2}  \nn
		&\quad \times 
		\bigg(\mathop{\sum\sum}_{\substack{Q_1\le x^{i_1},\, R_1\le x^{j_1} \\ (Q_1,R_1)=1}} 
		\sum_{\ell_1\in(\Z/R_1\Z)^*} \sum_{k_1 \in (\Z/Q_1\Z)^*} |\hat{\mu}(k_1/Q_1+\ell_1/R_1)|^2\bigg)^{1/2}.
	}
	Making the change of variables $q=Q_1R_1$ and using the Chinese Remainder Theorem, we deduce that
	\[
	S\ll  \sum_{\substack{i+j=4 \\ 1\le i\le 4}} \sum_{0\le i_1\le i} \sum_{0\le j_1\le j} \frac{x^{4-(i+j_1)/2}}{(\log x)^{4-(i_1+j_1)/2}} 	\bigg( \sum_{q\le x^{i_1+j_1}}\sum_{a\in(\Z/q\Z)^*} |\hat{\mu}(a/q)|^2\bigg)^{1/2}.
	\]
	Employing \eqref{ls2} with $y=x^{i_1+j_1}$, we arrive at the estimate
	\[
	S \ll  \sum_{\substack{i+j=4 \\ 1\le i\le 4}} \sum_{0\le i_1\le i} \sum_{0\le j_1\le j} \frac{x^{4-(i+j_1)/2}}{(\log x)^{4-(i_1+j_1)/2}} 	\cdot (x^{i_1+j_1}+H^{1/2}) \cdot \|\mu\|_2.
	\] 
 If $x$ is sufficiently large, then the expression $x^{i_1+j_1}\cdot x^{4-(i+j_1)/2}/(\log x)^{4-(i_1+j_1)/2}$ is maximized when $i_1=i$, $j_1=j$, in which case it equals $x^4\cdot (x/\log x)^2$ because we are only considering pairs $(i,j)$ with $i+j=4$. On the other hand, since we are ranging over indices $i\ge 1$,  $i_1\in[0,i]$ and $j\ge j_1\ge0$, the expression $x^{4-(i+j_1)/2}/(\log x)^{4-(i_1+j_1)/2}$ is maximized when $i_1=i=1$ and $j_1=0$,  in which case it equals $(x/\log x)^{7/2}$. This completes the proof of the lemma.
\end{proof}

We now explain how to complete the proof of Theorem~\ref{general main thm}. Since $\#\CN_4(x)\asymp(x/\log x)^4$, Lemma~\ref{lem:ls} implies, assuming $x$ is sufficiently large to guarantee that $\CN_4(x)$ is non-empty, that there is some $P\in\CN_4(x)$ with 
\begin{equation}\label{eq:muhat xH}
\alpha(P) \le c_0 \big( (x\log x)^2+ ((H\log x)/x)^{1/2} \big) \|\mu\|_2,
\end{equation}
where $c_0$ is an absolute constant (independent of $x$ and $\mu$). We will show that under the hypotheses of Theorem~\ref{general main thm} we can choose $x$ that makes the right-hand side of \eqref{eq:muhat xH} $\le 1/2$. 

First of all, note that 
\eq{C-S}{
	1=\bigg(\sum_{a\in\supp(\mu)}\mu(a)\bigg)^2
	\le \#\supp(\mu)\|\mu\|_2^2\le (2H+1)\|\mu\|_2^2
	\le 3H\|\mu\|_2^2
}
by the Cauchy-Schwarz inequality and our assumption that $\supp(\mu)\subset[-H,H]$. Next, if we write 
\[
\|\mu\|_2 = N^{-1/2}, 
\]
then we have $N\in[1,3H]$. (To motivate this change of variables, note that if $\mu$ is the uniform measure on $\CN$, then $N=\#\CN$.) In addition, condition (b) of Theorem~\ref{general main thm} is equivalent to $N\ge H^{4/5}(\log H)^2$ and $n\ge (H/N)^{16}(\log H)^{32}$.

We now see that the right-hand side of \eqref{eq:muhat xH} is $\le 1/2$ when
\[
x\le \frac{c_1N^{1/4}}{\log N}
\quad\text{and}\quad 
x\ge \frac{c_2H\log H}{N} ,
\]
where $c_1$ and $c_2$ are appropriate absolute constants. There is such a choice of $x$ precisely when $N\ge c_3H^{4/5}(\log H)^{8/5}$ for some $c_3>0$. This condition holds under the hypotheses of Theorem~\ref{general main thm} if $H$ is sufficiently large (in fact, the $(\log H)^2$ in Theorem~\ref{general main thm} can be improved to $c_3(\log H)^{8/5}$). We then pick the smallest available $x$, that is to say $x=c_2 (H\log H)/N$. If $H$ is sufficiently large then this ensures also that $x\ge 2$ and that $\CN_4(x)$ is non-empty, as they should be. We then see that  the number $P$ we constructed is $\le x^4\le c_2^4(H/N)^4(\log H)^{4}$. Since $n\ge (H/N)^{16}(\log H)^{32}$, we find that $n\ge \max\{P^4,(\log H)^3\}$. As a consequence, an application of Theorem~\ref{master thm} completes the proof of Theorem~\ref{general main thm}.

\subsection{Proof of Theorem~\ref{thm:GGs}}\label{proofs-galois}
In each of the set-ups of Theorems~\ref{main thm}(a) and~\ref{general main thm}-\ref{main thm generic}, we showed that we may find an integer $P\le n^4$ that is the product of four primes and which satisfies $\alpha(P)\le 1-c$ for some fixed $c>0$. Hence, Theorem~\ref{thm:GGs} follows readily from Theorem~\ref{master thm galois}. 

Finally, in the set-up of Theorem~\ref{main thm}(b), we know that our random polynomial is irreducible with probability $\ge\delta$. In order to show Theorem~\ref{thm:GGs} in this case, we fix some prime $p|N$. Thus $\Delta_p=0$ and we appeal 
to Proposition~\ref{from distr to galois} with $p_{\mathrm{Proposition}~\ref{from distr to galois}}=p$.\qedhere

\begin{rem}
More generally, assume that all non-leading coefficients of our polynomial are sampled uniformly at random from a step-$d$ arithmetic progression of $N$ elements. From Theorem 
\ref{main thm - APs}, we know that our random polynomial is irreducible with probability $\ge\delta$. If there exists at least one prime $p|N$ and $p\nmid d$, then we may apply Proposition \ref{from distr to galois} and deduce that the Galois group contains $\CA_n$ with probability $\ge\delta-n^{-c}$.

Note, however, that the above argument cannot be applied to the set $\{-1,+1\}$ without some modification.
\end{rem}

\part{Approximate equidistribution}\label{part-distr}

In this part of the paper, we establish Proposition~\ref{distr}. Throughout, $\CP=\{p_1,\dots,p_r\}$ is a set of primes and $P=p_1\cdots p_r$. We also assume that $p_1<\cdots<p_r$.

\section{The Fourier transform on \texorpdfstring{$\F_\CP[T]$}{F[T]}}\label{FT-F}
In order to capture the condition $\bs A\equiv \bs C\mod{\bs D}$ in the definition of $\Delta_\CP(n;m)$, we will use Fourier inversion over $\F_p[T]$. We begin by recalling a few basic facts about it.

We let $\F_p((1/T))$ denote the field of Laurent series $X(T) = \sum_{-\infty < j\le n} c_jT^j$, where $n\in\Z$ and $c_j\in \F_p$. We set
\[
\res(X):=c_{-1}
\]
and note that $\res$ is an additive function from $\F_p((1/T))$ to $\F_p$. 

Moving from a single prime to a set of primes, we let 
\[
\F_\CP((1/T))=\prod_{p\in\CP} \F_p((1/T))
\quad\text{and}\quad
\res(\bs X)=(\res(X_p))_{p\in\CP} .
\]
We then define the additive function $\psi_\CP:\F_\CP((1/T))\to\R/\Z$ by
\[
\psi_\CP(\bs X) := \sum_{p\in\CP} \frac{\res(X_p)}{p} \ \mod 1.
\]
(Occasionally we will also use a single prime version, $\psi_p\coloneqq\psi_{\{p\}}$.) It is well-known and not hard to check that the functions $A\mapsto e(\res(AB/D)/p)$ form a complete set of characters for the additive group of $\F_p[T]/D\F_p[T]$. We used here the customary notation
\[
e(x):=e^{2\pi i x}.
\]
Hence the same holds replacing a single prime $p$ with a set of primes $\CP$. In other words, the functions $\bs A\mapsto e(\psi_\CP(\bs A\bs B/\bs D))$ form a complete set of characters, where $\bs A\bs B/\bs D$ denotes the tuple $(A_pB_p/D_p)_{p\in\CP}$, which is an element of $\F_\CP((1/T))$.  The orthogonality of characters then gives the inversion formula 
\eq{FI-1}{
	\frac{1}{\|\bs D\|_\CP} \sum_{\bs B\mod{\bs D}} e \big(\psi_\CP(\bs A\bs B/\bs D)\big)
	= 1_{\bs A\equiv \bs0\mod{\bs D}},
}
Applying \eqref{FI-1} to $\bs A-\bs C$ with $\bs A$ random, and then taking expectations gives
\begin{equation}\label{FI-2}
\begin{split}
&\P_{\bs A\in \CM_{\CP}(n)}\big(\bs A\equiv\bs C\mod{\bs D}\big) \\
&\qquad = \frac{1}{\|\bs D\|_\CP} \sum_{\bs B\mod{\bs D}} 
e(\psi_\CP(-\bs C\bs B/\bs D)) 
\E_{\bs A\in\CM_\CP(n)}\big[e(\psi_\CP(\bs A\bs B/\bs D)) \big]  .
\end{split}
\end{equation}
The last term above has a concrete formula, as follows:

\begin{lem} For every $\bs X\in \F_\CP((1/T))$, we have
	\begin{equation}\label{eq:lem def S}
	\E_{\bs A\in \CM_\CP(n)}\big[e(\psi_\CP(\bs A\bs X)) \big]
	= e(\psi_\CP(T^n\bs X)) \prod_{j=0}^{n-1} \hat{\mu}_j(\psi_\CP(T^j\bs X)) .
	\end{equation}
\end{lem}

\begin{proof} Recall that the measure $\P_{\CM_\CP(n)}$ denotes the induced measure by the tuple $\bs A=(A_p)_{p\in\CP}=(A\mod p)_{p\in\CP}$ when $A(T)=T^n+\sum_{j=0}^{n-1} a_jT^j$ is sampled according to the measure $\P_{\CM(n)}$. In particular, the coefficient of $T^j$ of $A_p$ equals the reduction of $a_j$ modulo $p$. We thus find that
	\als{
		e(\psi_\CP(\bs A\bs X))
		&= e\bigg( \sum_{p\in \CP} \frac{\res(A_pX_p)}{p}\bigg)
		= e\bigg(\sum_{p\in \CP} \sum_{j=0}^n \frac{a_j\res(T^jX_p)}{p}\bigg)   \\
		&= e\bigg( \sum_{j=0}^{n} a_j\sum_{p\in \CP} \frac{\res(T^jX_p)}{p}\bigg)   
		= \prod_{j=0}^{n} e( a_j\psi_\CP(T^j\bs X)).
	}
	We now apply expectation to both sides. The $n^\textrm{th}$ term is constant and may be taken out, and we get
	\begin{align*}
	\E_{\bs A\in\CM_\CP(n)}[e(\psi_\CP(\bs A\bs X))]
	&=e(\psi_\CP(T^n\bs X))
	\prod_{j=0}^{n-1}\E_{A\in\CM(n)}( e(a_j\psi_\CP(T^j\bs X)))\\
	&=e(\psi_\CP(T^n\bs X))\prod_{j=0}^{n-1}\hat\mu_j(\psi_\CP(T^j\bs X)),
	\end{align*}
	where the first equality is due to the independence of the coefficients of $A$. 
\end{proof}

It will be convenient to have a notation for the absolute value of the right hand side of \eqref{eq:lem def S}, so let us define
\begin{equation}\label{FT-e1.0}
\sigma_\CP(n;\bs X)
= \prod_{j=0}^{n-1} |\hat{\mu}_j(\psi_\CP(T^j\bs X))| .
\end{equation}
With this notation \eqref{FI-2} and \eqref{eq:lem def S} give
\begin{equation}\label{FI-2S}
\bigg|\P_{\bs A\in \CM_{\CP}(n)}\Big(\bs A\equiv\bs C\mod{\bs D}\Big)
-\frac{1}{\|\bs D\|_\CP}\bigg|
\le \frac{1}{\|\bs D\|_\CP}
\sum_{\substack{\bs B \mod{\bs D} \\ \bs B\not\equiv\bs 0\mod{\bs D}}} \sigma_\CP(n;\bs B/\bs D).
\end{equation}
Selecting $\bs C$ that maximizes the left-hand side of \eqref{FI-2S}, and then summing the resulting inequality over $\bs D$ gives
\[
\Delta_\CP(n;m)
\le \mathop{\sum\cdots \sum}_{\substack{\deg(D_p)\le  m,\, T\nmid D_p \\ \forall p\in\CP}}	
\frac{1}{\|\bs D\|_\CP}  
\sum_{\substack{\bs B \mod{\bs D} \\ \bs B\not\equiv\bs 0\mod{\bs D}}}
\sigma_\CP(n;\bs B/\bs D) .
\]
(here and below we omit the condition of monicity from the sums for brevity).

Our last reduction before starting the bulk of the proof of Proposition~\ref{distr} is to replace the sum over $\bs B$ and $\bs D$ with a sum over coprime polynomials. Denote, therefore, $K_p=(B_p,D_p)$, and write $B_p=K_pG_p$ and $D_p=K_pH_p$, where $K_p$ and $H_p$ are monic polynomials with $\deg(K_p)+\deg(H_p)\le m$, and $(G_p,H_p)=1$. The condition $\bs B\not\equiv\bs 0\mod{\bs D}$ is equivalent to the existence of $p\in\CP$ with $\deg(H_p)\ge1$, which we may abbreviate as $\bs H\ne\bs 1$. Moreover,  since $T\nmid D_p$ for all $p\in\CP$, we have that $T\nmid H_p$ for all $p\in\CP$. 
As a consequence,
\als{
	\Delta_\CP(n;m)
	&\le \mathop{\sum\cdots \sum}_{\substack{\deg(K_p)\le  m \\ \forall p\in\CP}}	
	\frac{1}{\|\bs K\|_\CP}  
	\mathop{\sum\cdots \sum}_{\substack{\deg(H_p)\le  m,\, T\nmid H_p \\ \forall p\in\CP,\,\bs H\ne\bs 1}}	
	\frac{1}{\|\bs H\|_\CP}  
	\sum_{\substack{\bs G\mod{\bs H} \\ (G_p,H_p)=1\ \forall p\in\CP}}
	\sigma_\CP(n;\bs G/\bs H) .
}
Since $\sum_{\deg(K_p)\le m}1/\|K_p\|_p=m+1$, we conclude that
\eq{delta}{
	\Delta_\CP(n;m) 
	\le(m+1)^{r} 
	\sum_{\substack{0\le \ell_p\le m\ \forall p\in\CP \\ \max_{p\in\CP}\ell_p\ge1}}  
	\delta_\CP(n;\bs \ell)
}
(recall that $\#\CP=r$), where
\eq{delta-2}{
	\delta_\CP(n;\bs\ell) \coloneqq
	\frac{1}{\prod_{p\in\CP}p^{\ell_p}}
	\sum_{\substack{\bs H\in\CM_\CP(\bs \ell) \\ T\nmid H_p\ \forall p\in\CP}} 
	\sum_{\substack{\bs G\mod{\bs H} \\ (G_p,H_p)=1\ \forall p\in\CP}}
	\sigma_\CP(n;\bs G/\bs H) .
}
From \eqref{delta} and \eqref{delta-2} it follows that the proof of Proposition~\ref{distr} is reduced to proving that 
\eq{red-thm}{
	\delta_{\CP}(n;\bs \ell) \ll_r n^{-2r} e^{-n^{1/10}}
}
uniformly on $0\leq \ell_p\le \gamma n/s+n^{0.88}$, $p\in \CP$, with $\max_{p\in \CP}\ell_p\geq 1$.

\section{\texorpdfstring{$L^\infty$}{L infinity} bounds}\label{sec-L^infty} 

We begin our course towards proving \eqref{red-thm} by establishing a pointwise estimate on $\sigma_\CP(n;\bs X)$.

\begin{lem}\label{L^infty} Let $\mu_0,\mu_1,\dots,\mu_{n-1}$ be measures on $\Z$, let $\CP$ be a set of primes whose product is $P$, and let $\beta\in[0,1]$ be such that 
	\[
	|\hat{\mu}_j(k/P)| \le \beta \quad\mbox{for all $k\in\Z$ with $P\nmid k$, and for all $j=1,2,\dots,n-1$}.
	\]
	For each $p\in\CP$, let $G_p,H_p\in\F_p[T]$ with $T\nmid H_p$ and $(G_p,H_p)=1$. Assume further there is $q\in\CP$ such that $\ell_q:=\deg(H_q)\ge1$. Then 
	\[
	\sigma_\CP(n;\bs G/\bs H) \le \beta^{\fl{(n-1)/\ell_q}} .
	\]
\end{lem}

\begin{proof} Let $J\in\Z_{\ge0}$. If $\res(T^jG_q/H_q)=0$ for each $j\in\{J,J+1,\dots,J+\ell_q-1\}$, then we have $\res(T^JA_qG_q/H_q)=0$ for any polynomial $A_q$. So $T^JG_q/H_q$ must be a polynomial, which implies that 
	$H_q|T^J G_q$. Since $T\nmid H_q$, we infer that $H_q|G_q$. But this is impossible if $\ell_q\ge1$ and $(G_q,H_q)=1$.
	
	We have thus proven that any subinterval of $\Z_{\ge0}$ of length $\ell_q$ contains at least one $j$ such that $\res(T^jG_q/H_q)\neq0$. Hence, any subinterval of $\{1,\dots,n-1\}$ of length $\ge\ell_q$ contains at least one $j$ such that $\res(T^j\bs G/\bs H)\neq\bs 0$. For such a $j$, we have that 
	\[
	|\hat{\mu}_j(\psi_\CP(T^j \bs G/\bs H))|\le \beta.
	\]
	Otherwise, we use the trivial bound
	\[
	|\hat{\mu}_j(\psi_\CP(T^j \bs G/\bs H))| \le 1.
	\]
	The lemma then follows by the definition of $\sigma_\CP(n;\bs G/\bs H)$ from relation \eqref{FT-e1.0}.
\end{proof}

Clearly, for the above lemma to be useful, we need $\beta$ to be a bit smaller than 1. We will prove this by appealing to Lemma~\ref{lem:l-infinity bound}. Indeed, recall that $\CP=\{p_1,\dots,p_r\}$ and $P=p_1\cdots p_r$ are such that $\sum_{k\in\Z/p\Z}|\hat{\mu}_j(k/p)|^s\le\sqrt{p}$ for all $p\in\CP$ and all $j=1,2,\dots,n-1$. Together with relation \eqref{eq:FT-APs}, this implies that $\sum_{a\equiv b \mod p}\mu(a) \le 1-1/(4s)$ for all $p\in\CP$ and for all $b\in\Z$. Hence, Lemma~\ref{lem:l-infinity bound} implies that
\[
|\hat{\mu}_j(k/P)|\le 1-1/(sP^2) \le e^{-1/(sP^2)} 
\]
for all $k\in\Z$ that are not divisible by $P$, and for all $j=1,2,\dots,n-1$. We then set
\[
L=\max\{\ell_p:p\in\CP\}
\]
and plug in the above bound into Lemma~\ref{L^infty} to conclude that
\begin{align*}
\delta_\CP(n;\bs \ell)
\stackrel{\mathclap{\textrm{\eqref{delta-2}}}}{\le}
\Big(\prod_{p\in\CP}p^{\ell_p}\Big)\max_{\bs G,\bs H} \sigma_{\CP}(n;\bs G/\bs H)
&\leq \Big(\prod_{p\in\CP}p^{\ell_p}\Big) e^{-\fl{(n-1)/L} / (sP^2)} \\
&\ll  \exp\big(L\log P-n/(LsP^{2})\big) .
\end{align*}
According to the hypotheses of Proposition~\ref{distr}, we have $P\le n^{1/4}$ and $s\le n^{1/100}$.  If it so happens that we also have $L\le (n/\log n)^{1/2}/(s^{1/2}P)$, then taking $n$ sufficiently large yields the bound
\eq{L^infty bound}{
  \delta_\CP(n;\bs \ell)
  \ll \exp\Big(\frac{n^{1/2}\log P}{s^{1/2}P\cdot (\log n)^{1/2}}  -\frac{(n\log n)^{1/2}}{s^{1/2}P} \Big) 
  \le \exp\Big(-\frac{3(n\log n)^{1/2}}{4s^{1/2}P} \Big)  \ll e^{-n^{1/9}} .
}
This establishes a stronger version of \eqref{red-thm} for these tuples $\bs\ell$.

It remains to bound $\delta_\CP(n;\bs \ell)$ for those tuples $\bs\ell$ with $L\ge (n/\log n)^{1/2}/(s^{1/2}P)$. This requires different arguments that we develop in the next section.

\section{\texorpdfstring{$L^1$}{L 1} bounds}\label{sec-L^1} 
Here, we prove bounds for various averages of $\sigma_\CP(n;\bs X)$ that will allow us to complete the proof of Proposition~\ref{distr}. 
We begin by discussing a continuous analogue of \eqref{FI-1}. 

Let $\T_p$ denote the subring of $\F_p((1/T))$ composed of those Laurent series $X(T)=\sum_{j\le-1}c_jT^j$. Given any $Y\in \F_p((1/T))$, there is a unique way to write it as $X+A$, where $X\in \T_p$ and $A\in \F_p[T]$. If the coefficients of $X$ are $c_{-1},c_{-2},\dots$, we then set
\[
\|Y\|_{\T_p}:= p^{\sup\{j\in\Z_{\le-1}: c_j\neq0\}}
\]
with the understanding that $\|Y\|_{\T_p}=0$ when $X=0$.  

\begin{rem}\label{norm of quotient}
Let $A,B\in\F_p[T]\setminus\{0\}$ such that $B\nmid A$. We may then uniquely write $A=QB+R$ with $0\le \deg(R)<\deg(B)$, whence $A/B=Q+R/B$. In addition, we have $R=T^{\deg(R)}(r_0+r_1/T+r_2/T^2\cdots)$ and $B=T^{\deg(B)}(b_0+b_1/T+b_2/T^2+\cdots)$ for some coefficients $b_j,r_j\in\F_p$ with $b_0,r_0\neq0$. Using the formula $1/(1-x)=1+x+x^2+\cdots$ to invert $B$ in $\F_p((1/T))$, we conclude that $\|A/B\|_{\T_p}=p^{\deg(R)-\deg(B)}$.
\end{rem}

Let $\dee X_p$ be the Haar measure on $\T_p$ (normalised to have volume 1). We further define $\T_\CP=\prod_{p\in\CP}\T_p$ and write $\dee\bs X = \prod_{p\in\CP} \dee X_p$ for the product measure on $\T_{\CP}$. The continuous analogue of \eqref{FI-1} is that
\[
\int_{\T_\CP} e(\psi_\CP(\bs A\bs X)) \dee \bs X = 1_{\bs A=\bs 0} 
\]
for $\bs A\in\F_\CP[T]$, which follows by the orthogonality of characters.

Next, we show the following simple generalization of  \cite[Lemma 2]{porritt}.

\begin{lem}\label{L^1} Consider $m$ functions $f_0,f_1,\dots,f_{m-1}:\R/\Z\to\C$. For any prime $p$, we have
\[
	\int_{\T_p} \prod_{j=0}^{m-1} f_j\big(\psi_p(T^j X)\big)  \dee X
	= \frac{1}{p^m} \prod_{j=0}^{m-1}\bigg( \sum_{\xi\in\Z/p\Z}f_j(\xi/p) \bigg).
\]
\end{lem}

\begin{proof} If we write $X=\sum_{j\le-1}c_jT^j$, then the function $F(X):= \prod_{j=0}^{m-1} f_j\big(\psi_p(T^j X)\big)$ depends only on the coefficients $c_{-1},\dots,c_{-m}$. In particular, for any $B\in\F_p[T]$ of degree $<m$ and any $R\in \T_p$ such that $\|R\|_{\T_p}<1/p^m$, we have 
	\begin{equation}\label{eq:ConstantInBalls}
	F(R+ B/T^m) = F(B/T^m). 
	\end{equation}
	Since the Haar measure of the set $\{R\in\T_p:\|R\|_{\T_p}<1/p^m\}$ is $1/p^m$, and each $X\in\T_p$ has a unique representation of the form $R+B/T^m$ with $B$ and $R$ as above, we infer that
	\als{
		\int_{\T_p} F(X) \dee X
		= \frac{1}{p^m} \sum_{\deg(B)<m} F(B/T^m)   
		= \frac{1}{p^m} \sum_{\deg(B)<m}
		\prod_{j=0}^{m-1} f_j(\psi_p(T^{j-m}B)) .
	}
	If we write $B(T)=b_0+b_1T+\cdots+b_{m-1}T^{m-1}$, then $\res(T^{j-m}B) = b_{m-1-j}$. Hence,
	\als{
		\int_{\T_p} F(X) \dee X
		&= \frac{1}{p^m} \mathop{\sum\cdots\sum}_{b_0,b_1,\dots,b_{m-1}\in\F_p}
		\prod_{j=0}^{m-1} f_j(b_{m-1-j}/p) ,
	}
	which completes the proof of the lemma.
\end{proof}

Next, we give an inequality of large sieve type in $\F_p[T]$ that generalizes  \cite[Lemma 4]{porritt}.

\begin{lem}\label{L^1-large-sieve} Consider $m$ functions $f_0,f_1,\dots,f_{m-1}\colon \R/\Z\to\R_{\ge0}$. For all $\ell\in\Z_{\ge m/2}$, we have
\[
	\sum_{H\in \CM_p(\ell)} \sum_{\substack{G\mod{H} \\ (G,H)=1}} \prod_{j=0}^{m-1} f_j\big(\psi_p(T^jG/H)\big) 
	\le p^{2\ell-m} \prod_{j=0}^{m-1}\bigg( \sum_{\xi\in\Z/p\Z}f_j(\xi/p) \bigg).
\]
\end{lem}

\begin{proof} As in the proof of Lemma~\ref{L^1}, let $F(X)= \prod_{j=0}^{m-1} f_j\big(\psi_p(T^j X)\big)$ for $X\in\F_p((1/T))$. In addition, consider the $p$-adic ball $\CB(X):=\{Y\in \T_p: \|Y-X\|_{\T_p}<1/p^{2\ell}\}$. 
	
	Arguing as in \eqref{eq:ConstantInBalls} and using our assumption that $\ell\ge m/2$, we find that $F(Y)=F(X)$ for all $Y\in \CB(X)$. Consequently, 
	\[	
	\sum_{H\in \CM_p(\ell)} \sum_{\substack{G\mod{H} \\ (G,H)=1}} F(G/H) = p^{2\ell} \sum_{H\in \CM_p(\ell)} \sum_{\substack{G\mod{H} \\ (G,H)=1}} 
	\int_{\CB(G/H)} F(Y)\dee Y .
	\]
	The balls $\CB(G/H)$ with $\deg(G)<\deg(H)=\ell$ are disjoint, because if $G/H$ and $G'/H'$ are two distinct such Farey fractions, then $\|G/H-G'/H'\|_{\T_p}=\|(GH'-G'H)/HH'\|_{\T_p}\ge 1/p^{2\ell}$ by Remark~\ref{norm of quotient}. Since $F\ge0$ by our assumption that each $f_j$ takes values in $\R_{\ge0}$, we conclude that 
	\[	
	\sum_{H\in \CM_p(\ell)} \sum_{\substack{G\mod{H} \\ (G,H)=1}} F(G/H) \le p^{2\ell} \int_{\T_p} F(Y)\dee Y .
	\]	
	We evaluate the right-hand side using Lemma~\ref{L^1} to complete the proof.
\end{proof}

After applying H\"older's inequality as per \cite{DM1} to shorten the product in the definition of $\sigma_\CP(n;\bs X)$, we shall employ Lemma~\ref{L^1-large-sieve} in an iterative fashion to bound $\delta_\CP(n;\bs \ell)$ (recall its definition, \eqref{delta-2}), applying it to one prime of the set $\CP$ at a time. 

\begin{lem}\label{discrete L^1} 
Suppose there are parameters $s\in\N$, $\alpha\ge0$, $\gamma\ge1/2$, and a finite set of primes $\CP$ with $P=\prod_{p\in\CP}p$ such that
\[
\sum_{k\in\Z/Q\Z}|\hat\mu_j(k/Q+\ell/R)|^s\le \alpha\cdot Q^{1-\gamma}
\]
for all $j=1,2,\dots,n-1$ and all $Q,R,\ell\in\Z$ with $QR=P$ and $Q>1$. If $\ell_p\in\Z_{\ge0}$ for each $p\in\CP$, and we set $L=\max\{\ell_p:p\in\CP\}$ and $m=\fl{(n-1)/s}$, then
	\[
	\delta_\CP(n;\bs \ell)   \le  
	P^{\max\{0,L-\gamma m\}}\alpha^{\min\{2L,m\}}.
	\]
\end{lem}

\begin{proof} First, we use the trivial bound $|\hat{\mu}_j|\le1$ to reduce the product over $\hat\mu$ from $0,\dotsc,n-1$ to $1,\dotsc,sm$ (this, of course, removes very few terms, no more than $s$). Namely, we write
\[
	\delta_\CP(n;\bs\ell) 
	\le 	\frac{1}{\prod_{p\in\CP}p^{\ell_p}}
	\sum_{\substack{\bs H\in\CM_{\CP}(\bs \ell) \\ T\nmid H_p\ \forall p\in\CP}} 
	\sum_{\substack{\bs G\mod{\bs H} \\ (G_p,H_p)=1\ \forall p\in\CP}}
        \prod_{1\le j\le sm} |\hat{\mu}_j(\psi_\CP(T^j\bs G/\bs H))|  .
\]
We then apply H\"older's inequality to deduce that
\[
	\delta_\CP(n;\bs\ell) \le 	\frac{1}{\prod_{p\in\CP}p^{\ell_p}}
	\prod_{t=0}^{s-1} \bigg(\sum_{\substack{\bs H\in\CM_{\CP}(\bs \ell) \\ T\nmid H_p\ \forall p\in\CP}}
	\sum_{\substack{\bs G\mod{\bs H} \\ (G_p,H_p)=1\ \forall p\in\CP}} \prod_{tm<j\le (t+1)m} |\hat{\mu}_j(\psi_\CP(T^j\bs G/\bs H))|^s\bigg)^{1/s} . 
\]
If we write $j=tm+1+j'$ with $0\le j'<m$, then $T^jG_p/H_p=T^{j'}\cdot  (T^{tm+1}G_p/H_p)$. Moreover, if $T\nmid H_p$ and $G_p$ runs through all reduced residue classes mod $H_p$, then $T^{tm+1}G_p$ also runs through all reduced  residue classes mod $H_p$. We thus conclude that 	
\begin{equation}\label{delta-ub-holder}
	\delta_\CP(n;\bs\ell) 
	\le \frac{1}{\prod_{p\in\CP}p^{\ell_p}}
	\prod_{t=0}^{s-1} \bigg(\sum_{
		\bs H\in\CM_{\CP}(\bs \ell) 
		}
	\sum_{\substack{\bs G\mod{\bs H} \\ (G_p,H_p)=1\ \forall p\in\CP}} \prod_{0\le j<m} |\hat{\mu}_{tm+1+j}(\psi_\CP(T^j\bs G/\bs H))|^s\bigg)^{1/s}  ,
\end{equation}
	where we dropped the condition that $T\nmid H_p$ on the right-hand side because we no longer need it. 
	
	Let us now set some notation. Given $\phi\in\R$, $t\in\{0,1,\dots,s-1\}$, $j\in\{0,1,\dots,m-1\}$ and $\CQ\subseteq \CP$, we let 
	\[
	f_{t,j}(\phi;\CQ) :=  \mathop{\sum\cdots\sum}_{a_p\in\Z/p\Z\ \forall p\in \CQ} \Big|\hat{\mu}_{tm+1+j}\Big(\phi+\sum_{p\in\CQ} \frac{a_p}{p}\Big)\Big|^s.
	\]
	Using the Chinese Remainder Theorem and our assumption on $\alpha$ and $\gamma$, we find that
	\begin{equation}\label{eq:max of f_j}
	\sup_{\phi:\, P\phi\in \Z} f_{t,j}(\phi;\CQ) \le \alpha  Q^{1-\gamma} 
	\quad\text{whenever}\ \CQ\neq\emptyset,
	\end{equation}
	where $Q=\prod_{p\in \CQ} p$. In addition, we have that
	\begin{equation}\label{eq:sum of f_j}
	\sum_{a\in\Z/p\Z} f_{t,j}(\phi+ a/p;\CQ)  = f_{t,j}(\phi;\CQ\cup\{p\})
	\qquad \text{for all}\ p\in\CP\smallsetminus\CQ .
	\end{equation}
	Order $\CP=\{p_1,\dots,p_r\}$ according to the size of the $\ell_p$'s, i.e.~assume that $\ell_{p_1}\le \cdots \le \ell_{p_r}$, and set
        \[
        L_i=\ell_{p_i}
        \quad\text{and}\quad
	L_i'=\begin{cases}	
	0 &\text{if}\ i=0,\\
	\min\{L_i,m/2\} &\text{if}\ 1\le i\le r.
	\end{cases} 
	\]
	Finally, let $\CQ_i=\{p_{i+1},\dots,p_r\}$, $\CR_i=\{p_1,\dots,p_i\}$, $\bs \ell_i=( \ell_{p_1}, \ldots, \ell_{p_i})$ and 
	\[
	\mathscr{F}_{t,i}= \sum_{\bs H\in\CM_{\CR_i}(\bs \ell_i)} 
	\sum_{\substack{\bs G\mod{\bs H} \\ (G_p,H_p)=1\ \forall p\in\CR_i}}
	\prod_{j=0}^{2L_i'-1} f_{t,j}\big(\psi_{\CR_i}(T^j\bs G/\bs H); \CQ_i\big) 
	\]
	and let $\mathscr{F}_{t,0}=1$.
        
	For all $t=0,1,\dots,s-1$ and all $i=1,\dots,r$, we claim that 
\begin{equation}\label{eq:delta-inductive-bound}
	\mathscr{F}_{t,i}
	\le p_i^{2L_i-2L_i'}\alpha^{2L_i'-2L_{i-1}'} 
	\bigg( \prod_{j\ge i} p_j^{2(1-\gamma)(L_i'-L_{i-1}')}\bigg)\mathscr{F}_{t,i-1} .
\end{equation}
	
	\begin{proof}[Proof of \eqref{eq:delta-inductive-bound}]
	For brevity, we let $q=p_i$ and note that $\CQ_{i-1}=\CQ_i\cup\{q\}$, as well as that $\CR_{i-1}=\CR_i\smallsetminus\{q\}$. We fix an arbitrary choice of $\phi_1,\phi_2,\dots\in\R$ and apply Lemma~\ref{L^1-large-sieve} with $f_j^{\textrm{Lemma \ref{L^1-large-sieve}}}(x)=f_{t,j}(\phi_j+x;\CQ_i)$, $p^{\textrm{Lemma \ref{L^1-large-sieve}}}=q$, $m^{\textrm{Lemma \ref{L^1-large-sieve}}}=2L_i'$ and $\ell^{\textrm{Lemma \ref{L^1-large-sieve}}}=\ell_q=L_i$. We get that
	\[
	\sum_{H_q\in\CM_q(\ell_q)} \sum_{\substack{G_q\mod{H_q} \\ (G_q,H_q)=1}}
\prod_{0\le j<2L_{i}'} f_{t,j}\big(\phi_j+\psi_q(T^jG_q/H_q); \CQ_i\big) 
	\le q^{2L_i-2L_i'} \prod_{0\le j<2L_{i}'} f_{t,j}\big(\phi_j; \CQ_{i-1}\big) .
	\]
	If $P\phi_j\in\Z$ for all $j$, and we use the bound \eqref{eq:max of f_j} for $2L_{i-1}'\le j<2L_i'$, we conclude that 
	\begin{equation}
	\label{eq:delta-inductive-bound2}
	\begin{split}
	&\sum_{H_q\in\CM_q(\ell_q)} \sum_{\substack{G_q\mod{H_q} \\ (G_q,H_q)=1}}
	\prod_{0\le j<2L_i'} f_{t,j}\big(\phi_j+\psi_q(T^jG_q/H_q); \CQ_i\big)  \\
	&\qquad\le q^{2L_i-2L_i'}\alpha^{2L_i'-2L_{i-1}'} \bigg( \prod_{j\ge i} p_j^{2(1-\gamma)(L_i'-L_{i-1}')}\bigg)  
		\prod_{0\le j<2L_{i-1}'} f_{t,j}\big(\phi_j; \CQ_{i-1}\big) .
	\end{split}
	\end{equation}
        We apply  \eqref{eq:delta-inductive-bound2} with $\phi_j=\psi_{\CR_{i-1}}(T^j\bs G'/\bs H')$, where $\bs H'=(H_p)_{p\in\CR_{i-1}}$ runs over all tuples in $\CM_{\CR_{i-1}}(\bs \ell_{i-1})$ and $\bs G'=(G_p)_{p\in\CR_{i-1}}$ runs over all tuples in $\F_{\CR_{i-1}}[T]$ such that $\deg(G_p)<\deg(H_p)$ and $(G_p,H_p)=1$ for each $p\in\CR_{i-1}$. Summing the resulting inequalities completes the proof of \eqref{eq:delta-inductive-bound}.
	\end{proof}
	
	Let us now see how to use \eqref{eq:delta-inductive-bound} to complete the proof of the lemma. 	Note that when $i=r$, we have $\CQ_r=\emptyset$, and hence $f_j(\phi;\CQ_r)=|\hat\mu_j(\phi)|$. 
        Rewriting \eqref{delta-ub-holder} in the language of $\mathscr{F}_{t,i}$ (again using $|\hat{\mu}_j|\le 1$) gives
	\[
	\delta_\CP(n;\bs\ell)\le \frac{(\mathscr{F}_{0,r}\cdots \mathscr{F}_{s-1,r})^{1/s}}{\prod_{i=1}^r p_i^{L_i}}  .
	\]
	Since we also have that $\mathscr{F}_{t,0}=1$ for all $t$, applying \eqref{eq:delta-inductive-bound} in an iterative fashion yields that
	\[
	\delta_\CP(n;\bs\ell)\le \frac{\prod_{i=1}^r\big(p_i^{2L_i-2L_i'}\alpha^{2L_i'-2L_{i-1}'}\big)}{\prod_{i=1}^r p_i^{L_i}}  
	\prod_{i=1}^r \prod_{j=i}^r p_j^{2(1-\gamma)(L_i'-L_{i-1}')}
	= \alpha^{2L_r'} \prod_{i=1}^r p_i^{L_i-2L_i'+2(1-\gamma)L_i'}.
	\]
	The exponent of $p_i$ is $L_i-2\gamma L_i'=\max\{(1-2\gamma)L_i,L_i-\gamma m\}\le\max\{0,L_r-\gamma m\}$ for all $i=1,\dots,r$, where we used our assumption that $\gamma\ge1/2$. This completes the proof.
\end{proof}

\subsection{Proof of Proposition~\ref{distr}} 
Recall that it suffices to prove \eqref{red-thm}. We have already proven this in \eqref{L^infty bound} when $L:=\max\{\ell_p:p\in\CP\}\le (n/\log n)^{1/2}/(s^{1/2}P)$. 
Next, we consider the case when
\[
(n/\log n)^{1/2}/(s^{1/2}P)\le L\le\gamma\fl{(n-1)/s}.
\]
The hypotheses of Proposition~\ref{distr} allow us to apply Lemma~\ref{discrete L^1} with $\alpha=1-n^{-1/10}$. Denoting $m=\fl{(n-1)/s}$ we get
\[
\delta_\CP(n;\bs\ell)\le P^{\max\{0,L-\gamma m\}}\alpha^{\min\{2L,m\}}.
\]
Our restriction $L\le \gamma m$ gives $\max\{0,L-\gamma m\}=0$ and $\min\{2L,m\}\ge L$ (recall that $\gamma<1$, by Remark \ref{rem:gamma}) so
\[
\delta_\CP(n;\bs \ell) \le (1-n^{-1/10})^{L}\le \exp\big(-(n/\log n)^{1/2}/(n^{1/10}s^{1/2}P)\big)
\]
Since $P\le n^{1/4}$ and $s\le n^{1/100}$, \eqref{red-thm} follows in this case. 

Finally, let us consider the case when $\gamma m \le L\le \gamma n/s+n^{0.88}$. We then have $L-\gamma m\le 2+n^{0.88}$. Hence, Lemma~\ref{discrete L^1} and our assumptions that $P\le n^{1/4}$ and $s\le n^{1/100}$ imply that
\[
\delta_\CP(n;\bs \ell)
\le P^{2+n^{0.88}} (1-n^{-1/10})^{m} 
\le \exp\big(n^{0.88}\log n - m/n^{1/10}\big)
\ll \exp\big(-n^{0.89}/2\big) .
\]
This completes the proof of \eqref{red-thm} in this last case too.\qed

\part{Irreducibility}\label{part-irr}

\section{Ruling out factors of small degree}\label{large irr factors pf}

In this section, we establish Proposition~\ref{large irr factors} by adapting an argument due to Konyagin \cite{konyagin}. Noticing that condition (b) is only assumed for the measures $\mu_1,\dots\mu_{n-1}$, we may replace $\mu_0$ by the conditional measure $\mu_0(\ \cdot\ |a_0\neq0)$ without loss of generality. Throughout, we set
\[
H=\fl{\exp(n^{1/3})}
\]
and recall that $\supp(\mu_j)\subseteq[-H,H]$ for all $j$. 
In particular, all the coefficients $a_j$ of $A(T)$ lie in $[-H,H]$, and we also have $a_0\neq0$. Under these conditions, we have:

\begin{claim}\label{claim1} Any root $z$ of $A$ must satisfy $1/(H+1)<|z|<H+1$.
\end{claim}

\begin{proof} Indeed, if $|z|\ge H+1$, then the highest term $z^{n}$ dominates all the others and the sum cannot be zero. On the other hand,
if $|z|\le\frac{1}{H+1}$, then the lowest term $a_0$ dominates all others.
\end{proof}

A corollary of  Claim~\ref{claim1} is that if
\[
D|A,
\quad 
D\ \text{irreducible}, 
\quad D(T)=d_0+d_1T+\cdots+d_{m-1}T^{m-1}+T^m,
\]
then $D(T)\neq T$ and 
\eq{height condition}{
|d_j|\le \binom{m}{j}(H+1)^{m-j}\le m^j(H+1)^{m-j}\le(H+1)^m ,
}
since $m\le n\le H$ (see also \cite{G90}). Let $\CD(m_0)$ denote the set of monic irreducible polynomials $D(T)\neq T$ that have degree $\le m_0$ and all of whose coefficients satisfy \eqref{height condition}. We infer that
\begin{equation}\label{small factors union bound}
\P_{\CM(n)}\bigg(
	\begin{array}{l}A(T)\ \mbox{has an irreducible factor} \\ 
		\mbox{of degree $\le m_0$},\ a_0\neq0\end{array}\bigg) 
	\le \sum_{D\in\CD(m_0)} \P_{A\in\CM(n)}(D|A) .
\end{equation}

Our next task is to estimate what is the probability that a given irreducible polynomial $D\in\CD(m_0)$ divides a random polynomial $A$. Since $D$ is irreducible, this is equivalent to knowing that $A(z)=0$ for some $z$ that is a root of $D$. The following lemma controls the probability of this happening.

\begin{lem}\label{kolmogorov-rogozin} Let $\mu_0,\mu_1,\dots,\mu_{n-1}$ be probability measures such that
$\|\mu_j\|_\infty\le1-\eps$ for $j=1,2,\dots,n-1$. For each given $z\in\C\smallsetminus\{0\}$, we have that
\[
\P_{A\in\CM(n)}(A(z)=0) \ll \frac{1}{\sqrt{\eps n}} ,
\]
where the implied constant is absolute. 
\end{lem}

\begin{proof} Consider the independent random variables $X_j=a_jz^j$, where $a_j$ is distributed according to $\mu_j$ and note that the probability that $A(z)=0$ equals the probability that
\[
X_0+X_1+\cdots+X_{n-1}=-z^n.
\]
Define the concentration function of a real-valued random variable $X$ by
\[
Q(X;\delta):= \sup_{u\in\R}\P(|X-u|<\delta) .
\]
The Kolmogorov-Rogozin inequality \cite{kolmogorov,R61a,R61b} implies that there is an absolute constant $C$ such that
\[
Q(X_0+X_1+\cdots+X_{n-1};\delta) \le C \cdot \bigg(\sum_{j=0}^{n-1}(1-Q(X_j;\delta))\bigg)^{-1/2} .
\]
When $\delta=\min\{|z|,1\}^n/2$, we have that
\[
Q(X_j;\delta)= \sup_{u\in\R}\P\Big(|a_j-u|<\frac{\min\{|z|,1\}^n}{2|z|^j}\Big)
	\le \sup_{u\in\R}\P(|a_j-u|<1/2) = \|\mu_j\|_\infty \le 1-\eps
\]
for all $j\in\{1,\dots,n-1\}$. Hence, we conclude that
\[
\P(X_0+X_1+\cdots+X_{n-1}=-z^n)
\le Q(X_0+X_1+\cdots+X_{n-1};\delta) \le \frac{C}{\sqrt{\eps(n-1)}} ,
\]
as needed. 
\end{proof}

The rate of decay we obtain for each fixed $z$ in Lemma~\ref{kolmogorov-rogozin} is not strong enough to allow for a proof of Proposition~\ref{large irr factors}. We will use it to rule out cyclotomic divisors of $A$, and argue differently for non-cyclotomic divisors. We denote by $\Phi_d$ the $d^\textrm{th}$ cyclotomic polynomial. Recall that $\deg(\Phi_d)=\phi(d)$, the Euler totient function.

\begin{lem}\label{cyclotomic factors}
Assume the set-up of Lemma~\ref{kolmogorov-rogozin}. We then have that
\[
\sum_{\phi(d)\le m_0} \P_{A\in\CM(n)}(\Phi_d|A) \ll \frac{m_0}{\sqrt{\eps n}} \qquad\forall m_0\in\N.
\]
\end{lem}

\begin{proof}
  Since $\Phi_d(x)=\prod_{1\le j\le d,\, (j,d)=1}(x-e(j/d))$ is irreducible, $\Phi_d|A$ if, and only if, $A(e(1/d))=0$. Hence, Lemma~\ref{kolmogorov-rogozin} implies that $\P_{A\in\CM(n)}(\Phi_d|A)\ll 1/\sqrt{\eps n}$. The lemma is finished using the fact that the number of $d\in\N$ with $\phi(d)\le m_0$ is $O(m_0)$, see e.g.\ \cite{smati}.
\end{proof}

It remains to handle non-cyclotomic irreducible factors $D$ of $A$ of degree $m\le m_0$. Since $A$ is monic, $D$ must also be monic. In general, given a polynomial $f(T)=c(T-w_1)\cdots(T-w_m)$ with $c\in\C\setminus\{0\}$ and $w_1,\dots,w_m\in\C$, we define its {\it Mahler measure} to be
$
M(f):= |c|\prod_{j=1}^m\max\{|w_j|,1\}.
$
If $f\in \Z[X]$, then $|c|\geq 1$, and thus
\begin{equation}\label{mahler_ineq}
    M(f)\geq \prod_{j=1}^m\max\{|w_j|,1\}.
\end{equation}

Let $z_1,\dots,z_m$ denote the roots of $D$, which are all distinct by its irreducibility. Since $D|A$ and we have conditioned on $a_0\neq0$, we must have that $z_j\neq 0$ for all $j$. Since we have assumed that $D$ is not a cyclotomic polynomial, we know from a result of Dobrowolski \cite{dobrowolski} that there are some absolute constants $c,C\geq1$ such that
\[
M(D)\ge \exp(1/L(m)) ,
\quad\text{where}\quad L(m) =\frac12 \bigg(\frac{\log m}{\log\log m}\bigg)^3\quad\text{for all}\ m>C,
\]
and $L(m)=c$ for all $m\in[1,C]$.

In the same paper \cite[Lemma 3]{dobrowolski}, Dobrowolski also proved that, given an algebraic number $\alpha$ of degree $d$, there are $\le\log d/\log 2$ prime numbers $p$ such that the algebraic degree of $\alpha^p$ is $<d$. We apply this result with $\alpha=z_1$, whose degree is $m$. In particular, by the Prime Number Theorem, if $n$ is sufficiently large, then there is a prime number $p$ such that
\[
L(m)\log(2Hn)<p\le 2L(m)\log(2Hn)
\] 
and for which $z_1^p$ has algebraic degree $m$. We deduce that the numbers $z_1^p,\dots,z_m^p$ are distinct (this is because the list $z_1^p,\ldots, z_m^p$ contains all possible conjugates of $z_1^p$, and the number of conjugates of $z_1^p$ equals its degree, which is $m$ here by our choice of $p$).  We let $p=p_D$ be the smallest such prime, which we consider fixed for the rest of this section.

\begin{claim}\label{claim2} 
Let $D$ and $p$ be as above. Given integer coefficients $(c_j)_{0\le j<n,\, p\nmid j}$, there is at most one polynomial $A(T)=a_0+a_1T+\cdots+a_{n-1}T^{n-1}+T^n$ such that $D|A$, $|a_j|\le H$ for all $j$, and $a_j=c_j$ for all $j\not\equiv 0\mod p$.
\end{claim}

\begin{proof}
Assume, on the contrary, that there were two such polynomials, say $A$ and $B$. Their difference $A-B$ is a non-zero polynomial of the form
\[
A(T)-B(T) = \sum_{0\le j<n/p} g_j T^{pj} ,
\quad\text{where}\quad |g_j|\le2H.
\]
In addition, we know that $D|A-B$, whence $z_i^p$ is a root of the polynomial 
\[
G(T)=\sum_{0\le j<n/p}g_jT^j
\]
for all $i$. 

Let $\lambda\in\Z\setminus\{0\}$ be the leading coefficient of $G$, and let us write $G=\lambda \tilde{G}$ so that $\tilde{G}$ is monic. Since the numbers $z_1^p,\dots,z_m^p$ are distinct by our choice of $p$, by \eqref{mahler_ineq}, we infer that
\[
M(\tilde{G})\ge \prod_{i=1}^m \max\{1,|z_i^p|\} = M(D)^p \ge  \exp(p/L(m))>2Hn .
\]
However, by \cite[relation (1.1)]{konyagin}, we have
\[
M(\tilde{G})\le \sum_{0\le j<n/p} |g_j/\lambda|\le 2Hn ,
\]
a contradiction. This proves Claim~\ref{claim2}.
\end{proof}

We may now complete the proof of Proposition~\ref{large irr factors}. Let $D$ and $p$ be as above, with $m\le m_0:=\lfloor n^{1/10}\rfloor$. 
Claim~\ref{claim2} implies that 
\[
\P_{A\in\CM(n)}(D|A,\ a_j=c_j\ \forall j\not\equiv 0\mod p)
	\le \max_{1\le j<n}\|\mu_j\|_\infty^{\fl{(n-1)/p}}\prod_{\substack{0\le j<n \\ j\not\equiv0\mod p}}\mu_j(c_j) ,
\]
since there is at most one possibility for the polynomial $A$. Summing over all possibilities for $c_j$, we conclude that
\als{
\P_{A\in\CM(n)}(D|A) \le \max_{1\le j<n}\|\mu_j\|_\infty^{\fl{(n-1)/p}} \le (1-1/n^{1/10})^{\fl{(n-1)/p}}\ll e^{-n^{0.55}} ,
}
where we used that $p=p_D$ is a prime $\le  2L(m)\log(2Hn)\ll n^{1/3}\log^3 n$ for $m\le n^{1/10}$. Together with \eqref{small factors union bound} and Lemma~\ref{cyclotomic factors}, this implies that
\als{
\P_{\CM(n)}\bigg(
	\begin{array}{l}A(T)\ \mbox{has an irreducible factor} \\ 
		\mbox{of degree $\le n^{1/10}$},\ a_0\neq0\end{array}\bigg) 
	\ll \#\CD(n^{1/10}) \cdot e^{-n^{0.55}}  + n^{-7/20} .
}
The set $\CD(n^{1/10})$ has $\le 2(H+2)^{n^{1/5}}$ elements. To see this, recall the notation $m_0=\lfloor n^{1/10}\rfloor$. We then have two choices for the coefficient of $T^{m_0}$ (either 0 or 1), and $\le 2(H+1)^{m_0}+1$ for the coefficient of $T^m$ for each $m<m_0$  by \eqref{height condition}.  Since $H\le\exp(n^{1/3})$ here, we deduce that $\#\CD(n^{1/10})\ll \exp(n^{0.54})$. This completes the proof of Proposition~\ref{large irr factors}.

\section{An upper bound sieve}\label{sec:sieve}

Our next task is to prove Proposition~\ref{from distr to irr}. But first we develop a bit of sieve theory for $\F_p[T]$. Given the direct analogy between $\Z$ and $\F_p[T]$, it should not come as a surprise that the classical sieve methods over $\Z$ can be carried over to $\F_p[T]$. For example, Selberg's sieve has been ported to the polynomial setting by Webb \cite{Webb}, though he only considers the case when the underlying measure is the uniform counting measure on $\F_p[T]$. Here, we need a more general version of his work, adapted to a general probability measure $\P_{\CM_\CP(n)}$. Developing the full strength of Selberg's sieve is a bit tedious and would actually cause some technical problems in the next section\footnote{In the analogous result to Lemma~\ref{sieve} in the set-up of the Selberg sieve, the summands of the error term would be weighed with $\prod_{p\in\CP}3^{\omega(G_p)}$. In turn, this would require a more general version of Proposition~\ref{distr} that would introduce various unpleasant technicalities.}, so we opt for Brun's pure sieve \cite[Section 6.1]{opera}, which has the added advantage of being simpler and more intuitive.

To state our results, we develop some notation. Let $\CP$ denote a fixed finite set of primes. For each $p\in\CP$, we consider a set of monic irreducible polynomials $\CI_p\subset \F_p[T]$ and we let $\bs \CI = (\CI_p)_{p\in \CP}$. If $\bs A\in \F_\CP[T]$, we write 
\[
(A_p,\CI_p):=\prod_{I_p\in\CI_p,\, I_p|A_p}I_p
\qquad\text{and}\qquad 
(\bs A,\bs\CI):=((A_p,\CI_p))_{p\in\CP}.
\]
We also write $\bs A\bs B:=(A_pB_p)_{p\in\CP}$, $\bs A|\bs B$ if $A_p|B_p$ for all $p$, $\|\bs A\|_\CP=\prod p^{\deg(A_p)}$ and 
\[
\bs A|\bs\CI
\quad\Longleftrightarrow\quad
A_p|\prod_{I_p\in \CI_p} I_p \quad\text{for all}\ p\in\CP .
\]

\begin{rem*} If $\bs A|\bs \CI$, then $A_p$ must be square-free for every $p\in\CP$. 
\end{rem*}

Throughout this and the next section, we will make numerous appeals to the following result, which we record for easy reference.

\begin{prop}[Prime Polynomial Theorem \text{\cite[Proposition 2.1]{rosen}}]\label{PPT}
	If $k\in\N$ and $\pi_p(k)$ denotes the number irreducible elements of $\CM_p(k)$,  then we have
	\[
	\frac{p^k}{k} - \frac{2p^{k/2}}{k} \le \pi_p(k)\leq \frac{p^k}{k} .
	\]
	In particular, $\sum_{\deg I=k} \frac{1}{\|I\|}=\frac{1}{p^{k}} \pi_p(k) \leq 1/k$.
\end{prop}

Let us now state and prove our main sieve estimate. 

\begin{lem}\label{sieve} Let $\CP$ be a finite set of $r$ primes, and let $\P_{\CM_\CP(n)}$ be a probability measure on the set $\CM_\CP(n)$. For each $p\in\CP$, we consider a monic polynomial $D_p\in\F_p[T]$ and a set of monic irreducible polynomials $\CI_p$ in $\F_p[T]$ that have all degree $\le\ell_p$ for some $\ell_p\ge 11$. If $\bs1$ is the vector all of whose coordinates are $1$, then 
\als{
\P_{\bs A\in\CM_\CP(n)}\Big(\bs D|\bs A,\ (\bs A/\bs D,\bs\CI)=\bs 1\Big)
	&\le \frac{2^{r}}{\|\bs D\|_\CP}  \prod_{p\in\CP}
		\prod_{I_p\in\CI_p} \left(1- \frac{1}{\|I_p\|_p}\right) \\
	&\  +
		\mathop{\sum\cdots\sum}_{\substack{\omega(G_p)\le 6\log\ell_p\\ G_p|\CI_p \ \forall p\in\CP}}
			\bigg| \P_{\bs A\in \CM_\CP(n)}( \bs D\bs G|\bs A)
			-  \frac{1}{\|\bs D\bs G\|_\CP} \bigg| ,
}
where $\omega(G_p)$ denotes the number of monic irreducible factors of $G_p$. In particular, we have $\deg(G_p)\le6\ell_p\log\ell_p$ for all $G_p$ in the last sum.
\end{lem}

\begin{proof} 
We will perform inclusion-exclusion to capture the condition that
$(A_p/D_p,\CI_p)=1$ for all $p\in\CP$. 
Let $B$ be a square-free polynomial. Then the inclusion-exclusion principle for the events $J|B$, $J$ irreducible, shows that
\[
1_{B=1}=1-\sum_{J_1}1_{J_1|B}+\sum_{J_1,J_2}1_{J_1J_2|B}-\dotsb ,
\]
where all sums are over irreducible polynomials $J_i$. We write this more compactly as
\begin{equation}\label{eq:incexc}
1_{B=1}=\sum_{G|B}(-1)^{\omega(G)}.
\end{equation}
Stopping the inclusion-exclusion at even or odd steps leads to the following inequalities (sometimes known as Bonferroni inequalities):
\begin{equation}\label{eq:Bonferroni}
  \sum_{G|B,\,\omega(G)\le 2v-1}(-1)^{\omega(G)}
  \le1_{B=1}\le
  \sum_{G|B,\,\omega(G)\le 2v}(-1)^{\omega(G)}
  \qquad\forall v\in\N.
\end{equation}
For each $p\in\CP$, we select a natural number $v_p$ (to be determined shortly), and we apply the right-hand side of \eqref{eq:Bonferroni} with $B=(A_p/D_p,\CI)$ and $v=v_p$. We then multiply the resulting inequalities for all $p\in\CP$ (which we are allowed to do, as both sides are non-negative) to get
\begin{equation}\label{eq:GADI}
  1_{(\bs A/\bs D,\bs\CI)=\bs 1}
  \le \mathop{\sum\cdots \sum}_
      {\substack{\bs G|(\bs A/\bs D,\bs\CI) \\
          \omega(G_p)\le 2v_p\ \forall p\in\CP}}
  (-1)^{\omega(\bs G)}
\end{equation}
Consequently,
\al{
\P_{\bs A\in\CM_\CP(n)}&\Big(\bs D|\bs A,\ (\bs A/\bs D,\bs\CI)=\bs1\Big)	
\stackrel{\textrm{\eqref{eq:GADI}}}{\le}
\E_{\bs A\in\CM_\CP(n)}\bigg[ 1_{\bs D|\bs A} 
		\mathop{\sum\cdots \sum}_{\substack{\bs G|(\bs A/\bs D,\bs\CI) \\ \omega(G_p)\le 2v_p\ \forall p\in\CP}}(-1)^{\omega(\bs G)}\bigg]\nn
	&=\mathop{\sum\cdots \sum}_{\substack{\omega(G_p)\le 2v_p\ \forall p\in\CP\\ \bs G|\bs\CI }}(-1)^{\omega(\bs G)}
		\cdot\P_{\bs A\in\CM_\CP(n)} [\bs D\bs G|\bs A] \nn
	&\le\mathop{\sum\cdots \sum}_{\substack{\omega(G_p)\le 2v_p\ \forall p\in\CP\\ \bs G|\bs\CI }}\frac{(-1)^{\omega(\bs G)}}{\|\bs D\bs G\|_\CP} 
	 +\mathop{\sum\cdots \sum}_{\substack{\omega(G_p)\le 2v_p\ \forall p\in\CP\\ \bs G|\bs\CI }}
		\bigg| \P_{\bs A\in \CM_\CP(n)}( \bs D\bs G|\bs A)
			-  \frac{1}{\|\bs D\bs G\|_\CP} \bigg| .\label{brun-sieve-almost}
}

Let us fix at this point $v_p=\lceil 3/2 + 2\log\ell_p\rceil$. Note that $v_p\le 3\log\ell_p$, since we have assumed that $\ell_p\ge11$ for all $p\in\CP$. With this choice of $v_p$, the second term in \eqref{brun-sieve-almost} is bounded by the corresponding term in the equation in the statement of the lemma. 

Next, we examine the main term that factors as
\[
\frac{1}{\|\bs D\|_\CP}
	\prod_{p\in\CP} \bigg(\sum_{\substack{\omega(G_p)\le 2v_p \\ G_p|\CI_p}} \frac{(-1)^{\omega(G_p)}}{\|G_p\|_p} \bigg).
\]
If we remove the condition $\omega(G_p)\le2v_p$, we have the factorization 
\[
\sum_{G_p|\CI_p} \frac{(-1)^{\omega(G_p)}}{\|G_p\|_p} = \prod_{I_p\in\CI_p} \bigg(1-\frac{1}{\|I_p\|_p}\bigg) .
\]
We now claim that 
\eq{brun upper/lower-bound}{
\sum_{\substack{\omega(G_p)\le 2v_p+1 \\ G_p|\CI_p}} \frac{(-1)^{\omega(G_p)}}{\|G_p\|_p} 
\le \sum_{G_p|\CI_p} \frac{(-1)^{\omega(G_p)}}{\|G_p\|_p}
	\le \sum_{\substack{\omega(G_p)\le 2v_p \\ G_p|\CI_p}} \frac{(-1)^{\omega(G_p)}}{\|G_p\|_p} .
}
To see \eqref{brun upper/lower-bound}, let $N$ be some number. Apply \eqref{eq:incexc}-\eqref{eq:Bonferroni} to $(B_p,\CI_p)$ for all $B_p\in\CM_p(N)$ and sum the resulting inequalities. We get (showing only the upper bound for clarity)
\[
\sum_{B_p\in\CM_p(N)}\sum_{G_p|(B_p,\CI_p)}(-1)^{\omega(G_p)}
\le
\sum_{B_p\in\CM_p(N)}\sum_{\substack{G_p|(B_p,\CI_p) \\ \omega(G_p)\le 2v_p}}(-1)^{\omega(G_p)}.
\]
If $N\ge \sum_{I_p\in\CI_p}\deg(I_p)$, the left hand side equals $p^N\sum_{G_p|\CI_p}(-1)^{\omega(G_p)}/\|G_p\|_p$ and the right hand side equals $p^N\sum_{G_p|\CI_p,\,\omega(G_p)\le2v_p}(-1)^{\omega(G_p)}/\|G_p\|_p$. The lower bound of \eqref{brun upper/lower-bound} follows similarly.

Now, using \eqref{brun upper/lower-bound}, we find that
\begin{equation}\label{eq:omega=2v+1}
  0< \sum_{\substack{\omega(G_p)\le 2v_p \\ G_p|\CI_p}}
      \frac{(-1)^{\omega(G_p)}}{\|G_p\|_p} 
  \le \prod_{I_p\in\CI_p} \bigg(1-\frac{1}{\|I_p\|_p}\bigg) + 
      \sum_{\substack{\omega(G_p)=2v_p+1 \\ G_p|\CI_p}} \frac{1}{\|G_p\|_p} .
\end{equation}
Finally, observe that
\begin{equation}\label{eq:1/Gp}
 \sum_{\substack{\omega(G_p)=2v_p+1 \\ G_p|\CI_p}} \frac{1}{\|G_p\|_p} 
 	\le \frac{1}{(2v_p+1)!} \bigg( \sum_{I_p\in\CI_p}\frac{1}{\|I_p\|_p}\bigg)^{2v_p+1}
	\le\bigg( \frac{e}{2v_p+1}\sum_{I_p\in\CI_p}\frac{1}{\|I_p\|_p}\bigg)^{2v_p+1},
\end{equation}
where we used the inequality $n!\ge(n/e)^n$. Since all polynomials of $\CI_p$ have degree $\le\ell_p$, Proposition~\ref{PPT} implies that
\[
\sum_{I_p\in\CI_p} \frac{1}{\|I_p\|_p}
	\le \sum_{d=1}^{\ell_p} \frac{\#\{I\in\CM_p(d):I\ \text{irreducible}\}}{p^d} 
	\le\sum_{d=1}^{\ell_p} \frac{1}{d} \le1+\log\ell_p.
\]
Recall that we defined $v_p=\lceil 3/2  +2\log\ell_p\rceil$. Thus we conclude that $2v_p+1\ge 4\sum_{I_p\in\CI_p}1/\|I_p\|_p$. Plugging this inequality  into \eqref{eq:1/Gp} gives
\[
 \sum_{\substack{\omega(G_p)=2v_p+1 \\ G_p|\CI_p}} \frac{1}{\|G_p\|_p}
	\le (e/4)^{4\sum_{I_p\in\CI_p}1/\|I_p\|_p}
	= \prod_{I_p\in\CI_p}(e/4)^{4/\|I_p\|_p}
	\le \prod_{I_p\in\CI_p}\bigg(1-\frac{1}{\|I_p\|_p}\bigg),
\]
since $(e/4)^{4x}\le 1-x$ for all $x\in[0,1/2]$. Inserting this last inequality into \eqref{eq:omega=2v+1} gives
\[
 0< \sum_{\substack{\omega(G_p)\le 2v_p \\ G_p|\CI_p}}
      \frac{(-1)^{\omega(G_p)}}{\|G_p\|_p} 
  \le 2\prod_{I_p\in\CI_p} \bigg(1-\frac{1}{\|I_p\|_p}\bigg).
\]
Putting together the above inequalities completes the proof of the lemma.
\end{proof}

We conclude this section with a simple but useful estimate for the product of the statement of Lemma~\ref{sieve}.

\begin{lem}\label{lem:euler product} Let $\CI\subset \F_p[T]$ denote the set of monic irreducible polynomials different from $T$ and of degree $\le m$. Then
\begin{equation*}\label{euler product}
  \prod_{I\in\CI}\Big(1-\frac{1}{\| I \|_p}\Big)\le \frac{2}{m+1}
\end{equation*}
\end{lem}

\begin{proof} With $I$ denoting a generic monic irreducible element of $\F_p[T]$, we have
\begin{align*}
  \prod_{I\in\CI}\bigg(1-\frac{1}{\|I\|_p}\bigg)^{-1}
  &= \bigg(1-\frac{1}{p}\bigg)
     \prod_{\deg(I)\le m} \bigg(1-\frac{1}{\|I\|_p}\bigg)^{-1}\nn
  &=\frac{p-1}{p}\sum_{\substack{A\ \text{monic} \\ I|A\ \Rightarrow\  \deg(I)\le m}}
     \frac{1}{\|A\|_p} \nn
  &\ge\frac{p-1}{p}\sum_{0\le i\le m} \frac{\#\{A\in\CM_p(i)\}}{p^i} 
      \ge \frac{1}{2} \cdot (m+1),
\end{align*}
since $\#\{A\in\CM_p(i)\}=p^i$ for all $i$. This complete the proof.
\end{proof}

\section{Anatomy of polynomials}\label{anatomy}

We conclude Part~\ref{part-irr} of the paper with the proof of Proposition~\ref{from distr to irr}. Our argument relies on an analysis of the multiplicative structure of the reductions of a ``random'' element of $\CM_\CP(n)$. First, we introduce some terminology. 

We write $I_p$ for a generic monic irreducible polynomial over $\F_p$. Moreover, we let
\[
\tau(A_p)=\#\{D_p\in \F_p[T]\ \text{monic}: D_p|A_p\}
\]
for all $A_p\in\F_p[T]\smallsetminus\{0\}$. Note that 
\begin{equation}\label{tau-omega}
\tau(A_p)\ge 2^{\omega(A_p)} ,
\end{equation}
with equality if $A_p$ is square-free. 

The functions $\log\tau$ and $\omega$ are examples of {\it additive functions}. In general, a function $f\colon \F_p[T]\smallsetminus\{0\}\to\C$ is called additive if $f(AB)=f(A)+f(B)$ whenever $A$ and $B$ are coprime elements of $\F_p[T]\smallsetminus\{0\}$.

Finally, given an integer $m\ge0$, note that there is a unique way to decompose a monic polynomial $A_p$ as
\begin{equation}\label{smooth/rough}
A_p=A_p^{\CS(m)} \cdot A_p^{\CR(m)} ,
\quad\text{where}\quad
\begin{cases}
I_p|A_p^{\CS(m)}\ \Rightarrow\ \deg(I_p)\le m\ \text{and}\ I_p\neq T,\\
I_p|A_p^{\CR(m)}\ \Rightarrow\ \deg(I_p)>m\ \text{or}\ I_p=T,
\end{cases}
\end{equation}
and both polynomials $A_p^{\CS(m)}$ and $A_p^{\CR(m)}$ are monic. We call $A_p^{\CS(m)}$ the $m$-{\it smooth} part of $A_p$, and we call $A_p^{\CR(m)}$ its $m$-{\it rough} part\footnote{Normally, we would allow the irreducible factor $T$ in the $A_p^{\CS(m)}$, while forbidding it from $A_p^{\CR(m)}$. Here, we modify the usual notions to accommodate the fact that Proposition~\ref{distr} involves moduli that are coprime to $T$.}.

The next lemma shows that the $m$-smooth part of {\it most} polynomials is not too large.

\begin{lem}\label{normal-smooth} Fix $C\ge1$, and let $p$ be a prime, $n\in\Z_{\ge3}$, $m\in[4C,n]\cap\Z$ and $u\ge 2$. For any choice of probability measures $\mu_0,\mu_1,\dots,\mu_{n-1}$ on $\Z$, we have that
		\[
		\P_{A_p\in\CM(n)} \big( \deg(A_p^{\CS(m)})>um\big) \le O_C\big(m/e^{Cu}\big) + \Delta_p(n;um) .
		\]
\end{lem}

\begin{proof}
If $\deg(A_p^{\CS(m)})>um$, then $A_p$ has an $m$-smooth divisor $D_p$ such that 
\begin{equation}\label{degree of D_p}
(u-1)m <\deg(D_p)\le um ,
\end{equation}
Indeed, among all divisors of $A_p^{\CS(m)}$ of degree $\le um$, let $D_p$ be one of maximal degree. Since $\deg(A_p^{\CS(m)})>um$, there must exist at least one irreducible $I_p$ dividing $A_p^{\CS(m)}/D_p$. By the maximality of the degree of $D_p$, we find that $\deg(I_pD_p)>um$. On the other hand, $\deg(I_p)\le m$ because $I_p|A_p^{\CS(m)}$. Hence, $D_p$ satisfies \eqref{degree of D_p} as needed. 

By the above discussion and by the definition of $\Delta_\CP(n;um)$ (see \eqref{delta-dfn}), we have
\al{
	\P_{A_p\in\CM_p(n)}
	\big(\deg(A_p^{\CS(m)})>um\big)
	&\le  \sum_{\substack{D_p\ m\text{-smooth} \\ (u-1)m<\deg(D_p)\le um}}
	\P_{A_p \in\CM_p(n)}(D_p|A_p) \nn 
	&\le \sum_{\substack{D_p\ m\text{-smooth} \\ (u-1)m<\deg(D_p)\le um}} \frac{1}{\|D_p\|_p} 
	+\Delta_{\CP}(n;um) .
	\label{smooth-e1}
}
To control the main term, we employ Rankin's trick (Chernoff's bound): we have that
\als{
	\sum_{\substack{D_p\ m\text{-smooth} \\ \deg(D_p)>(u-1)m}}
	\frac{1}{\|D_p\|_p}   
	&
	=\sum_{\substack{D_p\ m\text{-smooth}  \\ \deg(D_p)>(u-1)m}}
	\frac{e^{C\deg(D_p)/m} \cdot e^{-C\deg(D_p)/m} }{\|D_p\|_p}
	\\
	&\leq \frac{1}{e^{C(u-1)}} 
	\sum_{D_p\ m\text{-smooth}}
	\frac{e^{C\deg(D_p)/m}}{p^{\deg(D_p)}}    \\
	&=\frac{1}{e^{C(u-1)}}
	\prod_{j=1}^m \left(1-\frac{e^{Cj/m}}{p^j}\right)^{-\pi_p(j)} ,
}
where $\pi_p(j)$ is the number of monic irreducible polynomials of $\F_p[T]$ of degree $j$. Since $m\ge4C$, we have $e^{C/m}\le e^{1/4}<2^{1/2}\le p^{1/2}$. 
Together with Proposition~\ref{PPT}, this implies that
\als{
	\sum_{\substack{D_p\ m\text{-smooth} \\ \deg(D_p)>(u-1)m}}
	\frac{1}{\|D_p\|_p}    
	&\le \frac{1}{e^{C(u-1)}} \exp\bigg\{ \sum_{j=1}^m \bigg(\frac{e^{Cj/m}}{j}+O\Big(\frac{e^{2Cj/m}}{jp^j}\Big) \bigg) \bigg\}  \\
	&\ll \frac{1}{e^{C(u-1)}} \exp\bigg\{ \sum_{j=1}^m \frac{e^{Cj/m}}{j}\bigg\} .
}
Using the fact that $e^{Cj/m}=1+O_C(j/m)$ for $j\le m$, we conclude that the sum over $j$ is $\log m+O_C(1)$. This proves that the first term of \eqref{smooth-e1} is $\ll_C m/e^{Cu}$, thus completing the proof of the lemma.
\end{proof}

The next lemma shows that the distribution of certain additive functions is concentrated around its mean value. In its statement, we write $I$ for a generic monic irreducible polynomial over $\F_p$. 

\begin{lem}\label{normal-additive} 
	Fix $\theta\in(0,1)$ and $C_1,C_2\ge3$. Consider a prime $p$ and an additive function $f\colon \F_p[T]\smallsetminus\{0\}\to\R_{\ge0}$ such that:
	\begin{enumerate}
		\item[(i)] $f(I)\in\{0,1\}$ for all monic irreducible polynomials $I\in\F_p[T]$;
		\item[(ii)] $0\le f(I^\nu)\le C_1\log \nu$ for all monic irreducible polynomials $I\in\F_p[T]$ and all $\nu\in\Z_{\ge2}$.
	\end{enumerate}
	Let $n\in\N$ and $m\in[1,2\theta n/\log n]\cap\Z$, and set 
	\[
	L_f(m)=\sum_{\substack{\deg(I)\le m \\ f(I)=1\\ I\textrm{ irreducible}}} \frac{1}{\|I\|_p}.
	\]
	Then, for any choice of probability measures $\mu_0,\dots,\mu_{n-1}$ on $\Z$, the following hold:
	\begin{enumerate}
		\item Uniformly for $0\le t\le 1$, we have
		\[
		\P_{A\in\CM(n)} \big( f(A_p^{\CS(m)}) \le tL_f(m) \big) \ll e^{-(t\log t-t+1)L_f(m)}  + n^8\Delta_p(n;\theta n)
		\]
		with the convention that $0\log0=0$. 
		\item Uniformly for $1\le t\le C_2$, we have
		\[
		\P_{A\in\CM(n)} \big(f(A_p^{\CS(m)})\ge tL_f(m)\big) \ll_{C_1,C_2}
			 e^{-(t\log t-t+1)L_f(m)} + n^{\max\{7,t+5\}}\Delta_p(n;\theta n).
		\]
	\end{enumerate}
\end{lem}

\begin{rem*}
For the purposes of Proposition \ref{from distr to irr}, we only use the lemma for two additive functions: $\omega$ and $\frac{\log \tau}{\log 2}$ (the division by $\log 2$ is in order to satisfy the condition $f(I)\in\{0,1\}$). The proof of Proposition \ref{from distr to galois} in Part \ref{part-galois} will necessitate more general choices of $f$.
\end{rem*}

\begin{proof} We first prove a special case of the lemma:

\medskip 

\noindent {\it Proof of part (b) when $f=\omega$ and $t\ge2$.} 
We may assume that $m$ is sufficiently large (depending on $C_1$ and $C_2$) as for $m$ small we also have $L_\omega(m)$ small and the bounds for the probabilities may be made larger than 1 by choosing the constants implicit in the $\ll$ signs sufficiently large. Notice that this also means that $n$ is sufficiently large (depending on $C_1$, $C_2$ and $\theta$), as otherwise there is no $m$ both sufficiently large and satisfying the requirement $m\le 2\theta n/\log n$.

For $L_\omega(m)$ we have the estimate
\begin{equation}
  L_\omega(m)=\sum_{\substack{\deg(I)\le m\\I\textrm{ irreducible}}}1/\|I\|_p=\log m+O(1)
  \label{eq:945}
\end{equation}
by Proposition~\ref{PPT}. Hence we need to show
\[
\P_{A\in\CM(n)} \big(\omega(A_p^{\CS(m)})\ge tL_\omega(m)\big) \ll_{C_2}
m^{-(t\log t-t+1)} + n^{\max\{7,t+5\}}\Delta_p(n;\theta n).
\]
We apply Lemma~\ref{normal-smooth} with $u_{\textrm{Lemma~\ref{normal-smooth}}}=(\theta n)/(2m)\ge \tfrac{1}{4} \log m$ and $C_{\textrm{Lemma~\ref{normal-smooth}}}=4C_2\log C_2$ to find that the probability that $\deg(A_p^{\CS(m)})>\theta n/2$ is $\ll_{C_2} m^{1-C_2\log C_2}+\Delta_p(n;\theta n/2)$. 
Thus, part (b) with $f=\omega$ and $t\ge2$ will follow if we can show that
\begin{equation}\label{normal-g-target}
\rho:= \P_{A\in\CM(n)} 
\bigg( \begin{array}{l} \deg(A_p^{\CS(m)})\le \theta n/2 \\
\omega(A_p^{\CS(m)})\ge tL_\omega(m)\end{array} \bigg) 
	\le O_{C_2}(m^{-(t\log t-t+1)})+n^{t+5} \Delta_p(n;\theta n) . 
\end{equation}

Borrowing an idea of Shiu \cite{shiu}, we order the irreducible factors of $A_p^{\CS(m)}$ (recall that they must be different from $T$) by their degrees, say 
\[
A_p^{\CS(m)}=I_{p,1}I_{p,2}\cdots I_{p,k}
\quad\text{with}\quad 
\deg(I_{p,1})\le\cdots \le \deg(I_{p,k}).
\]
Since $\omega(A_p^{\CS(m)})\ge tL_\omega(m)$, there is a unique $\ell\in[k]$ such that 
\[
\omega(I_{p,1}\cdots I_{p,\ell})\ge tL_\omega(m) > \omega(I_{p,1}\cdots I_{p,\ell-1}).
\] 
Set
\[
B_p=I_{p,1}\cdots I_{p,\ell-1}, \quad
J_p=I_{p,\ell},
\quad\text{and}
\quad
j=\deg(J_{p}),
\] 
so that $B_p$ is $j$-smooth, $A_p/(B_pJ_p)$ is $(j-1)$-rough, $\deg(B_pJ_p)\le \theta n/2$, and $tL_\omega(m)>\omega(B_p)\ge tL_\omega(m)-1$. Consequently,
\als{
\rho
	&\le \sum_{j=1}^m 
	\mathop{\sum\sum}_{\substack{B_p\ j\text{-smooth} \\\deg(J_p)=j,\ \deg(B_pJ_p)\le \theta n/2 \\ tL_\omega(m)-1\le \omega(B_p)<tL_\omega(m)}}
		\P_{A_p\in \CM_p(n)}
	\bigg(\begin{array}{l} B_pJ_p|A_p \\ A_p/(B_pJ_p)\ (j-1)\text{-rough} \end{array}\bigg) .
}
It will be convenient to replace the ``$(j-1)$-rough'' above with ``$((j-1)/24)$-rough'', which, of course, only increases the probability further. Let therefore $\CI_p(j)$ denote the set of monic irreducible polynomials different from $T$ and of degree $\le (j-1)/24$. We apply Lemma~\ref{sieve} with $\ell_p=\max\{11,\lfloor j/24\rfloor\}$ to each summand and get
\al{
\rho
	&\le 2 \sum_{j=1}^m 
	\mathop{\sum\sum}_{\substack{B_p\ j\text{-smooth},\ \deg(J_p)=j \\  \omega(B_p)\ge tL_\omega(m)-1}}\frac{1}{\|B_pJ_p\|_p}
	\prod_{I_p\in\CI_p(j)}\bigg(1-\frac{1}{\|I_p\|_p}\bigg)\nn
	&\quad	+\sum_{j=1}^m 	
	\mathop{\sum\sum\sum}_{B_p,\ J_p,\ G_p}
	\bigg| \P_{A_p\in\CM_p(n)}(B_pJ_pG_p|A_p)- \frac{1}{\|B_pJ_pG_p\|_p}\bigg| \nn
	&\eqqcolon M+R,\label{smooth-normal-setup}
}
where the remainder term $R$ runs over triplets $(B_p,J_p,G_p)$, where $B_p$ is $j$-smooth, $J_p$ is irreducible of degree $j$, $G_p|\CI_p(j)$, $\deg(B_pJ_p)\le \theta n/2$, $\omega(G_p)\le 6\log(\max\{\lfloor  j/24\rfloor,11\})$ and $\omega(B_p)\le tL_\omega(m)$. 

First, we deal with the remainder term $R$. Since $G_p|\CI_p(j)$, the polynomial $G_p$ must be square-free. Hence, the product $B_pJ_pG_p$ is a $j$-smooth polynomial $D_p$ with 
\[
\deg(D_p)=\deg(B_pJ_p)+\deg(G_p)\le \theta n/2+6\cdot (j/24)\log(\max\{j/24,11\})\le  \theta n
\]
for $j\le m\le 2\theta n/\log n$ and $m$ sufficiently large. Let us now estimate how many ways to write $D_p=B_pJ_pG_p$ exist, for a given $D_p$. For $J_p$ we have no more than $\omega(D_p)$ possibilities because it is irreducible. Once $J_p$ is chosen, $D_p/J_p$ can be written as $B_pG_p$ in no more than $2^{\omega(D_p/J_p)}$ ways, because $G_p$ is square-free. Note that
\[
\omega(D_p/J_p)=\omega(B_pG_p)\le\omega(B_p)+\omega(G_p) .
\]
Hence, our assumptions on $B_p$ and $G_p$ imply that
\begin{align*}
  \omega(D_p/J_p) & \le tL_\omega(m) + 6\log(\max\{\lfloor j/24\rfloor,11\})\\
  &\le t(\log m+O(1)) + 6 \log m \le (t+6)\log m+O(C_2)
\end{align*}
for $m$ sufficiently large. We get that the number of possibilities to get $D_p$ is no more than
\[
\omega(D_p)2^{\omega(D_p/J_p)}\le (1+O(C_2)+(t+6)\log m)2^{O(C_2)+(t+6)\log m}\le m^{t+4}
\]
for $m$ sufficiently large and $t\ge2$ (note that we have here $2^{\log m}$, but the log is to base $e$). Consequently,
\eq{smooth-normal-R}{
	R	\le \sum_{1\le j\le m} m^{t+4}\sum_{\deg(D_p)\le \theta n} \bigg| \P_{A_p\in\CM_p(n)}(D_p|A_p) - \frac{1}{\|D_p\|_p}\bigg|\le n^{t+5} \Delta_p(n;\theta n) .
}
This concludes the estimate of $R$.

For the main term $M$ of \eqref{smooth-normal-setup}, we apply Lemma~\ref{lem:euler product} to get
\[
\prod_{I_p\in\CI_p(j)}\bigg(1-\frac{1}{\|I_p\|_p}\bigg)\le \frac{2}{\lfloor (j-1)/24\rfloor+1}\le \frac{100}{j}.
\]
As a consequence,
\als{
	M
	\le \sum_{j=1}^m \frac{200}{j}
	\mathop{\sum\sum}_{\substack{B_p\ j\text{-smooth},\ \deg(J_p)=j \\ \omega(B_p)>tL_\omega(m)-1}}\frac{1}{\|B_pJ_p\|_p}.
}
For the sum over $J_p$, we note that
\[
\sum_{\deg(J_p)=j}\frac{1}{\|J_p\|_p} \le \frac{1}{j},
\]
where we used Proposition~\ref{PPT} again. 
Therefore,
\begin{equation}\label{eq:Mbytau}
M
\le \sum_{j=1}^m \frac{200}{j^2}
\sum_{\substack{B_p\ j\text{-smooth} \\  \omega(B_p)>tL_\omega(m)-1}} \frac{1}{\|B_p\|_p}
\le 200 \sum_{j=1}^m \frac{e^{-s(tL_\omega(m)-1)}}{j^2} 
\sum_{B_p\ j\text{-smooth}} \frac{e^{s\omega(B_p)}}{\|B_p\|} 
\end{equation}
for any choice of real number $s\ge0$, by Rankin's trick. 
Finally, note that
\als{
	\sum_{B_p\ j\text{-smooth}} \frac{e^{s\omega(B_p)}}{\|B_p\|} 
	\le \prod_{\deg(I)\le j} 
	\bigg( \sum_{\nu=0}^\infty \frac{e^{s\omega(I^\nu)}}{\|I^\nu\|_p}\bigg) 
	=\prod_{i=1}^j
	\bigg(1+\frac{e^s}{p^i-1}\bigg)^{\#\{\deg(I_p)=i\}} .
}
Using Proposition~\ref{PPT} again, as well as the inequality $1+x\le e^x$, we conclude that
\begin{align}\label{eq:taubyj}
\sum_{B_p\ j\text{-smooth}} \frac{e^{s\omega(B_p)}}{\|B_p\|} 
\le \exp\bigg(\sum_{i=1}^j \frac{e^s(1+O(p^{-i}))}{i} \bigg) 
&=\exp(e^s \log j+O(e^s)) .\nonumber
\end{align}
Inserting the above estimates into \eqref{eq:Mbytau}, with $L_\omega(m)=\log(m)+O(1)$, \eqref{eq:945}, we arrive at the bound
\[
M \le e^{O(e^s+C_2s)} \sum_{j=1}^m  j^{e^s-2} m^{-st}	.
\]
We take $s=\log t \in[\log 2,\log C_2]$ to conclude that 
\[
M \ll_{C_2} m^{e^s-1} m^{-st}	 =   m^{-(t\log t-t+1)} .
\]
Combining the above estimate with \eqref{smooth-normal-setup} and \eqref{smooth-normal-R} completes the proof of \eqref{normal-g-target}, and hence of the special case of part (b) of the lemma when $f=\omega$ and $t\ge2$.  

\medskip

Let us now prove Lemma~\ref{normal-additive} for all $f$ and all $t$. Note that we may assume that $t>0$; the case when $t=0$ will then follow by letting $t\to0^+$. 

In general, let $X\subset\R_{\ge0}$. We want to give a bound for $\P_{A\in\CM(n)}(f(A_p^{\CS(m)})\in X)$. Fix some $t_0\ge2$ and apply Lemma~\ref{normal-smooth} with $u_{\textrm{Lemma~\ref{normal-smooth}}}=(\theta n)/(2m)\ge \tfrac{1}{4} \log m$ and $C_{\text{Lemma~\ref{normal-smooth}}}=4t_0\log t_0$ to find that the probability that $\deg(A_p^{\CS(m)})>\theta n/2$ is $\ll_{t_0} m^{1-t_0\log t_0}+\Delta_p(n;\theta n/2)$ (still for $m$ sufficiently large). In addition, the portion of Lemma~\ref{normal-additive} already proven implies that 
\[
\P_{A\in\CM(n)}\big(\omega(A_p^{\CS(m)})\ge t_0\log m\big)
	\ll_{t_0} m^{-(t_0\log t_0-t_0+1)}+n^{t_0+5}\Delta_p(n;\theta n).
\] 
Consequently, 
\[
\P_{A\in\CM(n)}(f(A_p^{\CS(m)})\in X) = 
\P_{A\in \CM(n)} \left( \begin{array}{l} \deg(A_p^{\CS(m)})\le \theta n/2 \\ 
\omega(A_p^{\CS(m)})\le t_0L_\omega(m) \\ f(A_p^{\CS(m)}) \in X  \end{array} \right) 
+ \eta,
\]
where $\eta$ is the error (which is $\ll_{t_0} m^{-t_0\log t_0+t_0-1}+ n^{t_0+5}\Delta_p(n;\theta n)$, as above). Writing $B_p=A_p^{\CS(m)}$, we infer that
\[
\P_{A\in\CM(n)}(f(A_p^{\CS(m)})\in X) 
	= \sum_{\substack{B_p\ m\text{-smooth},\, f(B_p)\in X \\ \deg(B_p)\le \tfrac{\theta n}{2} \\  \omega(B_p)\le t_0L_\omega(m)}} 
				\P_{A\in\CM(n)} \bigg(\begin{array}{l} B_p|A_p \\ A_p/B_p\ m\text{-rough} \end{array}\bigg) +\eta.
\]
Note that if $A_p/B_p$ is $m$-rough, then it is also $(m/24)$-rough. Hence, if we let $\CI$ denote the set of monic irreducible polynomials over $\F_p$ of degree $\le m/24$ that are different from $T$, then Lemma~\ref{sieve} implies that
\als{
\P_{A\in\CM(n)} \big( B_p|A_p,\, A_p/B_p\ m\text{-rough}\big) 
	&\le \frac{2}{\|B_p\|_p} \prod_{I\in \CI} \Big(1-\frac{1}{\|I\|_p}\Big)  \\
	&\qquad	+ \sum_{\substack{G_p|\CI  \\ \omega(G_p)\le 6\log \ell  }} \Big| \P_{A\in\CM(n)}(B_pG_p|A_p)-\frac{1}{\|B_pG_p\|_p} \Big| ,
}
where $\ell:=\max\{11,\lfloor m/24\rfloor \}$. In addition, the product over $I\in\CI$ is $\le 48/m$ by Lemma~\ref{lem:euler product}. Consequently, 
\[
\P_{A\in\CM(n)}(f(A_p^{\CS(m)})\in X) 
\le \frac{100}{m} S+ E +\eta,
\]
\[
S:= \sum_{\substack{B_p\ m\text{-smooth} \\  f(B_p)\in X}} \frac{1}{\|B_p\|_p}
\quad\text{and}\quad
E:= \mathop{\sum\sum}_{B_p,\ G_p} \Big| \P_{A\in\CM(n)}(B_pG_p|A_p)-\frac{1}{\|B_pG_p\|_p} \Big|,
\]
with the second sum running over pairs $(B_p,G_p)$ such that $B_p$ is $m$-smooth, $\deg(B_p)\le \theta n/2$, $G_p|\CI$, $\omega(B_p)\le t_0L_\omega(m)$ and $\omega(G_p)\le 6\log\ell$ (we dropped the condition $f(B_p)\in X$ which we do not need to get a good estimate). Setting $D_p=B_pG_p$ and adapting the argument leading to \eqref{smooth-normal-R}, we find that
\eq{smooth-normal-E}{
	E	\le n^{t_0+5} \sum_{\deg(D_p)\le \theta n} \bigg| \P_{A_p\in\CM_p(n)}(D_p|A_p) - \frac{1}{\|D_p\|_p}\bigg|\le n^{t_0+5} \Delta_p(n;\theta n) .
}
In conclusion, we have proven that
\begin{equation}\label{smooth-additive-almost there}
\begin{split}
\P_{A\in\CM(n)}(f(A_p^{\CS(m)})\in X) 
&\le \frac{100}{m} S+ \eta+n^{t_0+5}\Delta(n;\theta n)\\
&=\frac{100}{m} S+ O_{t_0}\big(m^{-t_0\log t_0+t_0-1}+ n^{t_0+5}\Delta_p(n;\theta n)\big) .
\end{split}
\end{equation}
The argument now deviates according to the exact definition of $X$.

\medskip

(a) Here, $X=[0,tL_f(m)]$. We take $t_0=3$, so that $t_0\log t_0-t_0+1>1\ge t\log t-t+1$. Since
\begin{equation}
  L_f(m)\le \sum_{\deg(I)\le m}\frac{1}{\|I\|_p}
  =\log m+O(1)\label{eq:9105},
\end{equation}
the lemma will follow if we can show that $S\ll m\cdot e^{-(t\log t-t+1)L_f(m)}$. Indeed, by Rankin's trick, we find that 
\[
S \le e^{stL_f(m)} \sum_{B_p\ m\text{-smooth}} \frac{e^{-sf(B_p)}}{\|B_p\|_p} 
		\le e^{stL_f(m)} \prod_{\deg(I)\le m} 
	\bigg( 1+ \frac{e^{-sf(I)}}{\|I\|_p} + \sum_{\nu\ge2 } \frac{1}{\|I^\nu\|_p}\bigg) 
\]
for any $s\ge0$ (for $\nu>1$ we simply estimated $e^{-sf(I^\nu)}\le 1$). Next, we use the inequality $1+x\le e^x$ and the fact that $\sum_I \sum_{\nu\ge2}1/\|I^\nu\|_p=O(1)$ to conclude that
\[
S\ll \exp \bigg( stL_f(m)+\sum_{\deg(I)\le m} \frac{e^{-sf(I)}}{\|I\|_p} \bigg) .
\]
Now, since we assumed that $f(I)\in\{0,1\}$, we have $e^{-sf(I)}=(e^{-s}-1)\cdot 1_{f(I)=1}+1$ and summing over $I$ gives
\[
\sum_{\deg(I)\le m} \frac{e^{-sf(I)}}{\|I\|_p}
	=  (e^{-s}-1)L_f(m)+\sum_{\deg(I)\le m} \frac{1}{\|I\|_p}
	= (e^{-s}-1)L_f(m)+\log m+O(1).
\]
As a consequence, 
\[
S\ll  m\cdot \exp\big( (st+e^{-s}-1)L_f(m)\big) 
\]
uniformly for all $s\ge0$. Taking $s=-\log t\ge0$ to optimize the above inequality establishes the desired inequality that $S\ll m\cdot e^{-(t\log t-t+1)L_f(m)}$. This completes the proof of part (a) of the lemma.

\medskip

(b) Here, $X=[tL_f(m),+\infty)$. We take $t_0=\max\{t,2\}$, so that \eqref{smooth-additive-almost there} reduces the proof to showing that $S\ll m\cdot e^{-(t\log t-t+1)L_f(m)}$. This is proven in a similar way to part (a), starting this time with the inequality
\[
S\le e^{-stL_f(m)} \sum_{B_p\ m\text{-smooth}} \frac{e^{sf(B_p)}}{\|B_p\|_p} 
\]
that is valid for all $s\ge0$. We leave the details to the reader, and suffice in noting that it is at this point that we use the condition $f(I^\nu)\le C_1\log\nu$.
\end{proof}	

In the next result $m$ is allowed to vary, unlike in Lemmas~\ref{normal-smooth} and~\ref{normal-additive}, where $m$ was fixed.

\begin{lem}\label{normal} Fix $\theta,\eps\in(0,1)$. Let $\CP$ be a set of $r$ primes, let $n\in\N$, and let $\mu_0,\dots,\mu_{n-1}$ be probability measures on $\Z$ such that 
\[
\Delta_p(n;\theta n)\le n^{-8}\qquad\text{for all}\ p\in\CP.
\]	
Then there is a constant $c=c(\eps)>0$ such that 
\[
\P_{A\in\CM(n)} 
	\bigg( \begin{array}{l} \deg(A_p^{\CS(m)})\le \eps m\log m \\ 
			\tau(A_p^{\CS(m)})\le m^{(1+\eps)\log 2} \end{array}
			\ 
				\begin{array}{l} 
					\forall m\in[m_0,2\theta n/\log n] \\
					\forall p\in\CP
				\end{array}
	 \bigg)\ge 1 
	 		- O_{\eps,r}    \big(m_0^{-c}\big) 
\]
for all $m_0\in[1,2\theta n/\log n]$. 
\end{lem}

\begin{proof} We may assume that $\eps$ is sufficiently small and that $m_0$ is sufficiently large in terms of $\eps$. 
Define the events
\[
\CE_{p,m} =\bigg\{A_p\in \CM_p(n):
	\begin{array}{l}
		\deg(A_p^{\CS(m)})\le (\eps/3) m\log m\\  
		\tau(A_p^{\CS(m)})\le m^{(1+\eps/3)\log 2}
		\end{array}
		\bigg\}.
\]
The condition $\deg(A_p^{\CS(m)})\le (\eps/3)m\log m$ is handled by Lemma~\ref{normal-smooth}. We apply Lemma~\ref{normal-smooth} with $u_\textrm{Lemma~\ref{normal-smooth}}=(\eps/3)\log m$ and $C_\textrm{Lemma~\ref{normal-smooth}}=6/\eps$ and get
\[
\P_{A_p\in\CM(n)}\big(\deg(A_p^{\CS(m)})>(\eps/3)m\log m\big) \le O_\eps(m^{-1})+n^{-8},
\]
where we used that $(\eps/3)m\log m<\theta n$ for all $m\le 2\theta n/\log n$ to bound the error by $\Delta_p(n;\theta n)$. As for the condition $\tau(A_p^{\CS(m)})\le m^{(1+\eps/3)\log 2}$, it is handled by Lemma~\ref{normal-additive}(b). Indeed, note that the function $\log\tau/\log 2$ is an additive function satisfying the conditions of Lemma~\ref{normal-additive} with $C_1=3$. We wish to use Lemma~\ref{normal-additive}(b) with
\[
t_\textrm{Lemma~\ref{normal-additive}}=\frac{(1+\eps/3)\log m}{L_{\log \tau/\log 2}(m)}.
\]
Since $L_{\log\tau/\log 2}(m)=\sum 1/\|I\|_p$ over all irreducible $I$ with degree $\le m$, we have $L_{\log\tau/\log 2}(m)=\log m+O(1)$ and hence $t=1+\eps/3+O(1/\log m)$. In particular, for $m$ sufficiently large we have $t_\textrm{Lemma~\ref{normal-additive}}\in (1,2)$. We may therefore take the $C_2$ of Lemma~\ref{normal-additive} to be 2 and get
\[
\P_{A_p\in\CM(n)}(\tau(A_p^{\CS(m)})> m^{(1+\eps/3)\log 2})
\ll e^{(-t\log t-t+1)(\log m+O(1))}+n^{-1}.
\]
Summing both estimates we find that 
\eq{normal-e1}{
	\P_{A\in\CM(n)}	(A_p\in \CE_{p,m})\ge 1-Cm^{-c}
	\quad\text{for all}\ m\in[m_0, 2\theta n/\log n],\ p\in\CP
}
where $c=(1+\eps/3)\log(1+\eps/3)-\eps/3\in(0,1)$ and $C$ is some constant depending at most on $\eps$ and $\theta$. We will use this bound for carefully selected values of $m$ only. To this end, we define the checkpoints 
\[
m_j=\fl{\min\{2^j m_0,2\theta n/\log n\}},
\]
and let $J$ be the smallest index with $m_J=\fl{2\theta n/\log n}$. Note that
\eq{normal-e2}{
\bigg\{ A\in \CM(n) : A_p\in \bigcap_{j=0}^J \CE_{p,m_j}\ \forall p\in\CP \bigg\}
	\subset 
		\left\{
			A\in \CM(n):
			\begin{array}{l} 
			\deg(A_p^{\CS(m)})\le \eps m\log m \\ 
			\tau(A_p^{\CS(m)})\le m^{(1+\eps)\log 2} \\
			\forall m\in[m_0,2\theta n/\log n]\\  \forall p\in\CP
			\end{array}
		\right\}.
}
Indeed, for each $m\in[m_0,2\theta n/\log n]$, there is $j\in [J]$ such that $m_{j-1}\le m\le m_j$. Hence, if $A$ lies in the intersection of all $\CE_{p,m_j}$, then
\[
\deg(A_p^{\CS(m)})\le \deg(A_p^{\CS(m_j)}) \le (\eps/3) m_j \log m_j \le \eps m\log m
\]
and
\[
\tau(A_p^{\CS(m)})\le \tau(A_p^{\CS(m_j)})\le m_j^{(1+\eps/3)\log 2} \le m^{(1+\eps)\log 2}
\]
for all $p\in\CP$, provided that $m_0$ is sufficiently large in terms of $\eps$.

Now, to complete the proof note that \eqref{normal-e1} implies that
\[
\P_{A\in \CM(n)} 
	\bigg( A_p\in \bigcap_{j=0}^J \CE_{p,m_j}\ \forall p\bigg) 
	 \ge 1-\frac{rC}{2^c-1} \cdot m_0^{-c}.
\]
Together with \eqref{normal-e2}, this completes the proof with the implicit constant in the big-Oh term equal to $rC/(2^c-1)$.
\end{proof}

We are finally ready to establish the key estimate in our proof of Proposition~\ref{from distr to irr}.

\begin{lem}\label{lem:one k} 
Let $\theta\in(0,1/2]$, $\delta\in(0,1]$, $\lambda\in(0,1)$, $\CP=\{p_1,\dotsc,p_r\}$ be a set of primes, $n\in\Z_{\ge 2}$, and $\mu_0,\dotsc,\mu_{n-1}$ be probability measures on $\Z$ satisfying
\begin{equation}\label{eq:Delta<n^r}
\Delta_\CP(n;\theta n+n^\lambda) \le n^{-7r}
\qquad\textnormal{and}\qquad 
\sup_{1\le j<n}\sum_{a\equiv 0\mod p}\mu_j(a)\le 1-\delta\quad\forall p\in\CP.
\end{equation}
Fix, in addition, $\eps\in(0,1)$ and $k\in\Z\cap[1,\theta n]$, and let $\CE_{k,\lambda,\eps,\theta}$ be the event of the statement of Lemma~\ref{normal} with $m_0=k^{\lambda/2}$, namely, the event that $\deg(A_p^{\CS(m)})\le\eps m\log m$ and $\tau(A_p^{\CS(m)})\le m^{(1+\eps)\log 2}$ for all $m\in\Z\cap [k^{\lambda/2},2\theta n/\log n]$ and all $p\in\CP$.

Then, we have that
 \begin{equation}\label{eq:one k}
  \P_{A\in\CM(n)}
  	\big(\CE_{k,\lambda,\eps,\theta}\cap \{\forall p\in\CP,\, \exists D_p|A_p\ \textrm{with}\ \deg(D_p)=k\}\big)
  		\ll_{r,\eps,\lambda} 
  			\bigg(\frac{\log^2 n}{\delta k^{(1-\log 2-\eps)\lambda}}\bigg)^r.
 \end{equation}
 \end{lem}

\begin{proof} All implicit constants in Vinogradov's notation $\ll$ may depend on $r$, $\eps$ and $\lambda$. Let us write $\CE$ instead of $\CE_{k,\lambda,\eps,\theta}$ for simplicity.

We may assume without loss of generality that $\eps<1-\log(2)$, that $k$ is sufficiently large (depending on $r$, $\eps$ and $\lambda$), because for small $k$ the claim holds trivially by adjusting the implied constant in \eqref{eq:one k}. Similarly, we may assume $k^\lambda\ge100(\log n)^2$ and $k\ge100(1+\ceil{r\delta^{-1}\log n})$. This also means that $n$ can be taken to be sufficiently large, as otherwise there might not be any $k\in[1,\theta n]$ for which the claim is nontrivial.

We first consider the power of $T$ that divides $A_p$. By the right-hand side of \eqref{eq:Delta<n^r}, we infer that
\begin{align}
\P_{A\in\CM(n)}(T^\nu |A_p) 
&=\P_{A\in\CM(n)}(p|a_0,a_1,\dots,a_{\nu-1}) 
= \smash[b]{\prod_{j=0}^{\nu-1}\bigg( \sum_{a\equiv0\mod p}\mu_j(a) \bigg)} \nn
&\le (1-\delta)^{\nu-1} \le e^{-\delta\cdot (\nu-1)} .\label{eq:16half}
\end{align}
Choosing
\[
\nu=1+\lceil r\delta^{-1}\log n\rceil,
\]
for which we have $\nu\le k/100$ by our assumptions on $k$, we find that
\begin{equation}\label{power of T}
\P_{A\in\CM(n)}(T^\nu|A_p)\le n^{-r}.
\end{equation}
This is negligible quantity compared to the right-hand side of \eqref{eq:one k}. We therefore assume for the rest of the proof that all our polynomials satisfy $T^\nu\nmid A_p$. We deduce that $A_p$ has a divisor $D_p$ coprime to $T$ of degree $k_p\in(k-\nu,k]$ (this is not the same $D_p$ from the statement of the lemma, hopefully no confusion will arise). Therefore, if we denote
  \[
  \rho:=\P_{A\in\CM(n)}\big(\CE\cap \{\forall p\in\CP,\, \exists D_p|A_p\ \textrm{with}\ \deg(D_p)=k\}\cap \{T^\nu\nmid A_p\}\big)
  \]
(essentially the left-hand side of \eqref{eq:one k}), then
\eq{from distr to irr - rho}{
\rho\le \sum_{\substack{k-\nu<k_p\le k\\ p\in \CP}}
	\rho(\bs k)
	\ll (\delta^{-1}\log n)^r\max_{\substack{k-\nu<k_p\le k\\ p\in \CP}}
	\rho(\bs k),
}
where
\[
\rho(\bs k):=\P_{\bs A\in \CM_\CP(n)}
	\Big(\CE \cap
		\big\{ \forall p\in\CP,\, \exists D_p|A_p\ \textrm{with}\ T\nmid D_p\ \textrm{and}\ \deg(D_p)=k_p\big\}
	\Big) .
\]

We fix for the rest of the proof a tuple $\bs k=(k_p)_{p\in\CP}\in (k-\nu,k]^r$ maximizing $\rho(\bs k)$. In addition, we define 
\[
m=\fls{k^\lambda/8\log n},
\]
for which we have $k^{\lambda/2}\le m\le 2\theta n/\log n$ by our assumptions that $k^\lambda\ge100(\log n)^2$ and $k\le\theta n$. Hence for all polynomials $A\in\CE=\CE_{k,\lambda,\eps,\theta}$ and all primes $p\in\CP$, we have $\deg(A_p^{\CS(m)})\le\eps m\log m$ and $\tau(A_p^{\CS(m)})\le m^{(1+\eps)\log 2}$. If we let $B_p=A_p^{\CS(m)}$ and we assume that $D_p$ divides $A_p$, then $D_p^{\CS(m)}$, the $m$-smooth part of $D_p$, must divide $B_p$. Consequently,
\[
\rho(\bs k)\le 
\mathop{\sum\sum}_{(\bs B,\bs D)\in\CX_{\bs k}}
	\P_{\bs A\in \CM_\CP(n)}\bigg(\begin{array}{l} [B_p,D_p]\, |\, A_p  \\ A_p/[B_p,D_p]\ m\text{-rough}\end{array}
		\ \forall p\in\CP\bigg)
\]
where $\CX_{\bs k}$ is the set of all couples $(\bs B,\bs D)$ such that $B_p$ is $m$-smooth, $\deg(B_p)\le\eps m\log m$, $\tau(B_p)\le m^{(1+\eps)\log 2}$, $D_p^{\CS(m)}\mid B_p$, $\deg(D_p)=k_p$ and $T\nmid D_p$, for all $p\in\CP$. We apply Lemma~\ref{sieve} with $\CI_p$ the set of monic irreducible polynomials $I_p\neq T$ with $\deg(I_p)\le m$ to each couple $(\bs B, \bs D)$ and sum over them. This yields that
\eq{from distr to irr - rho(k)}{
\rho(\bs k)   \le M+R,
}
where $M$ is the main term given by
\[
M=2^r\mathop{\sum\sum}_{(\bs B,\bs D)\in\CX_{\bs k}}
	\frac{\prod_{p\in\CP}\prod_{I_p\in\CI_p}(1-1/\|I_p\|_p)}{\|[\bs B,\bs D]\|_\CP} 
\]
and $R$ is the remainder term given by 
\[
R = \mathop{\sum\sum\sum}_{\substack{(\bs B,\bs D)\in\CX_{\bs k}\\ G_p\ m\text{-smooth, squarefree},\\ 
	\omega(G_p)\le6\log m \,\forall p\in\CP}}
	\Big|\P_{\bs A\in \CM_\CP(n)}\big(\bs A\equiv \bs0\mod{[\bs B,\bs D]\bs G}\big) - \frac{1}{\|[\bs B,\bs D]\bs G\|_\CP}\Big|.
\]

We first deal with the remainder term $R$. We make the change of variables $H_p=[B_p,D_p]G_p$ for each $p\in\CP$. Notice that $T\nmid H_p$ for all $p$ (recall that the definition of the smooth part of a polynomial precludes the factor $T$), as well as that
\[
\deg(H_p)\le \deg(B_p)+\deg(D_p)+\deg(G_p)\le \eps m\log m+\theta n+6m\log m,
\]
since $\deg(D_p)=k_p\le k\le \theta n$ and we know that $G_p$ is a square-free and $m$-smooth polynomial with $\le 6\log m$ irreducible factors. We have $\eps<1$ and $m\le n^\lambda/8\log n$, and thus
\[
\deg(H_p)\le \theta n+n^\lambda\quad\text{for all}\  p\in\CP 
\]
for $n$ sufficiently large. This inequality will allow us to bound $R$ in terms of $\Delta_\CP(n;\theta n+n^\lambda)$. But first we must also understand how many times each choice of $H_p$ occurs.

Note that the $m$-rough part of $H_p$ is always given by the $m$-rough part of $D_p$, so there is no multiplicity created there. Adding to this the fact that $D_p^{\CS(m)}$ divides $B_p$ gives that $H_p^{\CS(m)}=G_pB_p$. The number of ways to write $H_p^{\CS(m)}$ as a product of two polynomials is $\tau(H_p^{\CS(m)})$, and if there is even one way to write  $H_p^{\CS(m)}=G_pB_p$ with our restrictions on $G_p$ and $B_p$  then we would get that
\[
\tau(H_p^{\CS(m)}) = \tau(B_pG_p)\le \tau(B_p)\tau(G_p)\le m^{(1+\eps)\log 2} \tau(G_p).
\]
Since $G_p$ is square-free, we have $\tau(G_p)= 2^{\omega(G_p)}\le m^{6\log 2}$.

Once $G_p$ and $B_p$ are chosen, we must also choose $D_p^{\CS(m)}$, and since it divides $B_p$, the number of possibilities for that is at most $\tau(B_p)\le m^{(1+\eps)\log 2}$. All in all, we get that the number of appearances of each $H_p$ is bounded by $m^{(8+2\eps)\log2}$. Since there are $r$ different $p\in\CP$ we get that the total number of appearances of each $\bs H$ is bounded by
\[
m^{r(8+2\eps)\log 2}\le m^{6r}.
\]

Putting everything together, we arrive at the inequality
\al{
R &\le m^{6r} \mathop{\sum\cdots\sum}_{\substack{\deg(H_p)\le \theta n+n^\lambda  \\ T\nmid H_p\ \forall p\in\CP}}
	\Big|\P_{\bs A\in \CM_\CP(n)}\big(\bs A\equiv \bs0\mod{\bs H} \big) - \frac{1}{\|\bs H\|_\CP}\Big| \nn
	&\le m^{6r}\Delta_\CP(n; \theta n+n^\lambda)\le n^{-r} ,
	\label{from distr to irr - R}
}
where the last relation follows from \eqref{eq:Delta<n^r}.

It remains to bound the main term $M$ of \eqref{from distr to irr - rho(k)}. Appealing to Lemma~\ref{lem:euler product}, we have that
\eq{euler product 2}{
\prod_{I_p\in\CI_p} \bigg(1-\frac{1}{\|I_p\|_p}\bigg) \le \frac{2}{m}
}
for all $p\in\CP$. Consequently,
\[
M \le \frac{4^r}{m^r} \mathop{\sum\sum}_{(\bs B,\bs D)\in \CX_{\bs k}} \frac{1}{\|[\bs B,\bs D]\|_\CP} .
\]

Writing $D_p'=D_p^{\CS(m)}$ and $D_p''=D_p^{\CR(m)}$, we find that $[B_p,D_p]=B_pD_p''$. Fix for the moment $B_p$ and $D_p'|B_p$. We then find that $\deg(D_p'')=k_p-\deg(D_p')$ is fixed and positive, say equal to $j$. Note that $j\ge k-\nu-\eps m\log m>6m\log m$, because $\nu\le k/100$, $m\le k^\lambda/8\log n$,  $\eps< 1$ and $k$ is sufficiently large.

To find an upper bound for
\[
\sum_{\substack{\deg(D_p'')=j \\ D_p''\ m\text{-rough}}} \frac{1}{\|D_p''\|_p} 
=\frac{\#\{D_p''\in\CM_p(j): D_p''\ m\text{-rough}\}}{\#\{D_p''\in\CM_p(j)\}} 
\]
we apply Lemma~\ref{sieve} with $\CP_{\textrm{Lemma~\ref{sieve}}}=\{p\}$, $n_{\textrm{Lemma~\ref{sieve}}}=j$, $\P_{\textrm{Lemma~\ref{sieve}}}$ being the probability measure coming from the uniform counting measure on $\CM_p(j)$,  $D_{\textrm{Lemma~\ref{sieve}}}=1$, and the $\CI_p$ of Lemma~\ref{sieve} being as here, i.e., all irreducible polynomials of degree $\le m$, except for $T$. Since $j>6m\log m$, the error term vanishes identically, and we find that
\[
\sum_{\substack{\deg(D_p'')=j \\ D_p''\ m\text{-rough}}} \frac{1}{\|D_p''\|_p} \le 2\prod_{I_p\in\CI_p}(1-1/\|I_p\|_p) .
\]

The conclusion of the above discussion is that
\[
M \le \frac{8^r}{m^r}\prod_{p\in\CP}\prod_{I_p\in\CI_p}(1-1/\|I_p\|_p)
	 \mathop{\sum\sum}_{\substack{B_p\ m\text{-smooth},\ D_p'|B_p \\ 
	 \tau(B_p)\le m^{(1+\eps)\log 2}\  \ \forall p\in\CP}}
	\frac{1}{\|\bs B\|_\CP}.
\]
Obviously, there are $\le \tau(B_p)\le m^{(1+\eps)\log2}$ choices for $D_p'$. As a consequence,
\[
M \le\frac{8^{r}m^{r(1+\eps)\log 2}}{m^r}\prod_{p\in\CP}\prod_{I_p\in\CI_p}(1-1/\|I_p\|_p)
		\sum_{B_p\ m\text{-smooth}\ \forall p\in\CP}   \frac{1}{\|\bs B\|_\CP}.
\]
Since
\[
\sum_{B_p\ m\text{-smooth}\ \forall p\in\CP} 
	\frac{1}{\|\bs B\|_\CP} =  \prod_{p\in\CP}\prod_{I_p\in\CI_p}\bigg(1-\frac{1}{\|I_p\|_p}\bigg)^{-1}
\]
the two terms in the estimate of $M$ cancel perfectly. Using also $m=\lceil k^\lambda/8\log n\rceil$, we arrive at the bound
\[
M \le  \frac{8^{r}}{m^{r(1-\log(2)-\eps\log 2)}} \ll \frac{(\log n)^r}{k^{r\lambda(1-\log(2)-\eps)}} .
\]
Together with \eqref{from distr to irr - rho(k)} and \eqref{from distr to irr - R}, this implies that
\[
\rho(\bs k) \ll \frac{(\log n)^r}{k^{r\lambda(1-\log(2)-\eps)}} . 
\]
With \eqref{from distr to irr - rho}, the proof of the lemma is done.
\end{proof}

\section{Proof of Proposition~\ref{from distr to irr}} 
Without loss of generality, we may assume that $n$ is sufficiently large. In addition, we may assume that $\supp(\mu_0)\neq\{0\}$; otherwise, the conclusion of Proposition~\ref{from distr to irr} is trivial.

Let $\eps\in(0,1/100]$, $\mu_0,\dots,\mu_{n-1}$ and $\CP$ be as in Proposition~\ref{from distr to irr}. Let $A(T)=a_0+a_1T+\cdots+a_{n-1}T^{n-1}+T^n$ be a random polynomial with $a_0\neq0$ sampled according to the measure $\P_{\CM(n)}$. By Proposition~\ref{large irr factors}, all irreducible factors of $A$ have degree $\ge n^{1/10}$ with probability $1-O(n^{-7/20})$, so let us assume that this is the case. 

We apply Lemma~\ref{normal} with the parameters $\eps_{\textrm{Lemma~\ref{normal}}}=\eps/10$, $m_0=n^{1/30}$, and $\theta$ and $\CP$ being as here. Letting $c_1=c_{\textrm{Lemma~\ref{normal}}}/30>0$, we get that, with probability $1-O_\eps(n^{-c_1})$, we have
\eq{normal-e3}{
\deg(A_p^{\CS(m)})\le \tfrac1{10}\eps m\log m
	\quad\text{and}\quad
	\tau(A_p^{\CS(m)})\le m^{(1+\eps/10)\log 2}
}
for all $m\in \Z\cap [n^{1/30},2\theta n/\log n]$ and all $p\in \CP$. Denote this event by $\CE$. 

Next, we apply Lemma~\ref{lem:one k} for each integer $k\in[n^{1/10},\theta n]$ with the parameters $\eps_{\textrm{Lemma~\ref{lem:one k}}}=\eps/10$, $r_{\textrm{Lemma~\ref{lem:one k}}}=4$, $\delta_{\textrm{Lemma~\ref{lem:one k}}}=n^{-\eps/200}$, $\lambda_{\textrm{Lemma~\ref{lem:one k}}}=\lambda_0+\eps$, and $\theta$ and $\CP$ as here. We get
\begin{equation}\label{eq:no-divisors}
  \P_{A\in\CM(n)}\big(\{\forall p\in\CP,\, \exists D_p|A_p\ \textrm{with}\ \deg(D_p)=k\}\cap \CE^*\big)\ll_\eps 
  \bigg(\frac{n^{\eps/200}\log^2n}{k^{(1-\log 2-\eps/10)(\lambda_0+\eps)}}\bigg)^4
\end{equation}
where $\CE^*$ is from Lemma~\ref{lem:one k}. But $\CE^*$ contains $\CE$ since the only difference between them is the range of $m$ involved, $[n^{1/30},2\theta n/\log n]$ for $\CE$ and $[k^{(\lambda_0+\eps)/2},2\theta n/\log n]$ for $\CE^*$ (recall that $\lambda_0>0.8$ and $k\ge n^{0.1}$). Hence we may replace $\CE^*$ with $\CE$ in \eqref{eq:no-divisors}. Since $4(1-\log2-\eps/10)(\lambda_0+\eps)\ge 1+0.8\eps$, we find that
\[
\sum_{n^{1/10}\le k\le \theta n} 
\bigg(\frac{n^{\eps/200}\log^2n}{k^{(1-\log 2-\eps/10)(\lambda_0+\eps)}}\bigg)^4
\ll_\eps  \frac{n^{\eps/50}\log^8n}{(n^{1/10})^{0.8\eps}}
	\ll_\eps n^{-\eps/20} .
\]
We conclude that
\begin{align*}
&\P_{\CM(n)}\big(\exists D|A\ \text{with}\ \deg D\le \theta n,\, a_0\ne0\big)\\
 	&\quad \le \P_{A\in \CM(n)}\big(\exists D|A\ \text{with}\ \deg D\le n^{1/10} \,\big|\, a_0\neq 0\big)  
	+ \P_{A\in \CM(n)}\big(\exists D|A : \deg(D)\in(n^{1/10},\theta n]\big) \\
&\quad \le O(n^{-2/5})+\P_{\CM(n)}(\CE^c) \\
&\qquad +\sum_{n^{1/10}\le k\le \theta n} \P_{\CM(n)}\big(\CE\cap\{A\,:\, \exists D|A\ \text{with}\ \deg(D)=k\}\big)\\
&\quad\ll_\eps n^{-2/5}+n^{-c_1}+ n^{-\eps/20},
\end{align*}
thus proving Proposition~\ref{from distr to irr} with $c=\min\{2/5,c_1,\eps/20\}$.

\part{The Galois group}\label{part-galois}

In this final part of the paper, we prove Proposition~\ref{from distr to galois}. We must show that if we sample a polynomial $A\in\CM(n)$ according to the measure $\P_{\CM(n)}$, then the odds that $A$ is irreducible and, at the same time, its Galois group $\CG_A$ is different from $\CA_n$ and $\CS_n$ are small.

\section{Galois theory}\label{sec:GaloisTheory}

Recall that $A$ is irreducible if and only if $\CG_A$ is transitive. Thus, if we set
\begin{equation}\label{eq:def Tn}
\CT_n:=\bigcup_{\substack{G\leqslant \CS_n \\ G\ \text{transitive}\\ G\neq \CA_n,\CS_n}} G,
\end{equation}
then Proposition~\ref{from distr to galois} is reduced to showing that
\begin{equation}
\label{eq:galois-goal}
\P_{A\in \CM(n)}\big( \CG_A\subset  \CT_n\big) \ll n^{-c}
\end{equation}
under its assumptions, where $c$ is some appropriate absolute constant.

To prove \eqref{eq:galois-goal}, we will reduce our polynomial $A$ modulo the prime $p$ of the statement of Proposition~\ref{from distr to galois}, for which we know that
\begin{equation}
\label{distr mod p}
\Delta_p(n;n/2+n^\lambda) \le n^{-10}
\qquad\text{and}\qquad \sup_{1\le j<n} \sum_{a\equiv0\mod p} \mu_j(a) \le 1-1/(\log n)^2
\end{equation}
for some $\lambda\in(0,1)$. In particular, $A_p$, which denotes the reduction of $A$ mod $p$, is approximately uniformly distributed in $\CM_p(n)$. We will then factor $A_p$ in $\F_p[T]$ and deduce \eqref{eq:galois-goal} from a result about the distribution of random partitions.

\subsection{The factorization type of $A_p$}
Recall that a partition of $n$ is an increasing sequence $\rho=(\rho_1,\rho_2,\dots,\rho_r)$ of positive integers (for some $r$) such that $\sum_{i=1}^r \rho_i=n$, and that this is denoted by $\rho\vdash n$. 

The polynomial $A_p$ can be factored as a product of irreducible elements of $\F_p[T]$, say $A_p=\prod_{i=1}^r I_i$ with the factors arranged so that $\deg(I_1)\le\cdots\le \deg(I_r)$. Hence, the tuple
\[
\tau_{A_p}:= (\deg(I_1),\dots,\deg(I_r))
\]
is a partition of $n$ that we shall refer to as the {\it factorization type} of $A_p$.

The above observation implies that the probability measure $\P_{\CM(n)}$ naturally induces a probability measure $\nu$ on the set of partitions of $n$. This measure is defined by
\begin{equation}
\label{nu-dfn}
\nu(\CE) := \P_{A\in\CM(n)}(\tau_{A_p}\in \CE) 
\end{equation}
for all sets $\CE$ of partitions of $n$.

The following lemma records some of the key properties of $\nu$ (and, thus, of the distribution of $\tau_{A_p}$). To state it, it will be convenient to use set notation for partitions (even though they are multisets rather than sets). Thus, for example, $k\in\rho$ will mean that for some $i$, $\rho_{i}=k$, while $\{k,k\}\subseteq\rho$ will mean that for some $i\ne j$, $\rho_{i}=\rho_{j}=k$. If $U\subset \rho$, then $\sum_{u\in U} f(u)$ means that we sum the elements of $U$ according to their multiplicity, and so on and so forth.

\begin{lem}\label{lem:PartitionMeasure}
Let $\nu$ be the measure defined by \eqref{nu-dfn}, where $n\ge16$ and $p$ is a prime satisfying \eqref{distr mod p} for some $\lambda>0$. 
We write $\rho$ for a partition of $n$ sampled according to $\nu$. Then

\begin{enumerate}
\item For all $k,\ell\in[2,n/4]\cap\Z$, we have
\[
	 		\nu(\{k,\ell \}\subseteq \rho ) \le \frac{2}{k\ell} .
\]
\item There is an absolute constant $c>0$ such that 
\[
\nu\Big(\exists U\subseteq\rho \text{ such that }\sum_{u\in U}u=k\Big)\ll_\lambda k^{-c\lambda }
\quad\text{for all}\ k\in[n^{1/10},n/2]\cap\Z.
\]
\item Let $f\colon \N\to\{0,1\}$, $m\in[1,n/\log n]\cap \Z$, $t\in(0,1)$, and set $L=\sum_{k=1}^m f(k)/k$. Then
\[
\nu\bigg( \sum_{k\in \rho\cap[1,m]}f(k)\le tL \bigg)  \ll  e^{-(t\log t-t+1)L} .
\]
\end{enumerate}
\end{lem}

\begin{proof} (a) Let $\CI_k$ be the set of monic irreducible polynomials of degree $k$, and consider $k,\ell\in[2,n/4]$, so that $k+\ell\le n/2$ and the polynomial $I(T)=T$ is not contained in $\CI_k\cup\CI_\ell$. Thus
\[
	\nu(\{k,\ell\} \subseteq \rho) \leq 
		 \mathop{\sum\sum}_{I\in \CI_k,\, J\in\CI_\ell } \P_{A\in\CM(n)}(IJ | A_p)
		 \leq 	 \mathop{\sum\sum}_{I\in \CI_k,\, J\in\CI_\ell} \frac{1}{\|IJ\|_p} 
		 +2\Delta_p(n;k+\ell) .
\]
Since $\sum_{I\in\CI_k}1/\|I\|_p\le 1/k$ by Proposition~\ref{PPT} and $\Delta_p(n;k+\ell)\le n^{-10}\le 1/(2k\ell)$ by \eqref{distr mod p}, we conclude that $	\nu(\{k,\ell\} \subseteq \rho) \le 2/(k\ell)$ as needed.

\medskip

(b) Note that
\als{
\nu\Big(\exists U\subseteq\rho\text{ such that }\sum_{u\in U}u=k\Big)
		= \P_{A\in\CM(n)}\big(\exists D_p | A_p\ \text{such that}\ \deg D_p=k\big) .
}
Now, let $\CE=\CE_{k,\lambda,1/100,1/2}$ denote the event described in Lemma~\ref{lem:one k} with $\eps_{\mathrm{Lemma~\ref{lem:one k}}}=1/100$, $\theta_{\mathrm{Lemma~\ref{lem:one k}}}=1/2$, $\CP_{\mathrm{Lemma~\ref{lem:one k}}}=\{p\}$ and $\delta_\textrm{Lemma~\ref{lem:one k}}=1/\log^2n$. Assumption \eqref{distr mod p} ensures that the conditions of Lemma~\ref{lem:one k} are met, so we infer that
\[
 \P_{A\in\CM(n)}
\big(\CE\cap \{ \exists D_p|A_p\ \textrm{with}\ \deg(D_p)=k\}\big)
\ll_\lambda k^{-0.2\lambda}
\]
for $k\in[n^{1/10},n/2]$. In addition, Lemma~\ref{normal} implies that $\P_{\CM(n)}(\CE) \ge 1-O_\lambda(k^{-c_1\lambda})$ 
for an absolute constant $c_1>0$. Putting together the above estimates completes the proof of clause (b) of the lemma with $c=\min\{c_1,0.2\}$.

\medskip

(c) We may assume that $L\ge1$; otherwise, the result is trivially true. Note that
\[
\nu\bigg( \sum_{k\in \rho\cap[1,m]}f(k)\le tL \bigg) 
	= \P_{A\in\CM(n)}\bigg( \sum_{I^r\|A_p,\,\deg(I)\le m} r f(\deg(I))\le tL \bigg) ,
\]
where $I$ denotes a generic monic irreducible polynomial over $\F_p$ and where, as usual, $I^r\|A_p$ means that $I^r\mid A_p$ but $I^{r+1}\nmid A_p$. Let $g$ denote the additive function over $\F_p[T]$ defined by
\[
g(I^r) = f(\deg(I)).
\]
Recall the notation $A_p^{\CS(m)}$, which we introduce in relation \eqref{smooth/rough}. We then observe that $g(A_p^{\CS(m)})\le  \sum_{I^r\|A_p,\, \deg(I)\le m} r f(\deg(I))$, and thus
\begin{equation}\label{eq:nu of sum}
\nu\bigg(\sum_{k\in \rho\cap[1,m]}f(k)\le tL \bigg) 
\le \P_{A\in\CM(n)}\big( g(A_p^{\CS(m)})\le tL\big) .
\end{equation}
Recall the notation $L_g(m)$ from Lemma~\ref{normal-additive}. We then have 
\[
L_g(m) = \sum_{\substack{1\le k\le m \\ f(k)=1}} \sum_{\deg(I)=k} \frac{1}{p^k} 
			= \sum_{\substack{1\le k\le m \\ f(k)=1}} \bigg(\frac{1}{k}+O\big(p^{-k/2}\big)\bigg) 
			= L+O(1)
\]
by Proposition~\ref{PPT}. We then define $t^*$ by the relation $t^*L_g(m)=tL$, so that $t^*=t+O(1/L)$. If $t^*<1$, then Lemma~\ref{normal-additive}(a) with $\theta=1/2$ implies that 
\als{
\nu\bigg( \sum_{k\in \rho\cap[1,m]}f(k)\le tL \bigg) 
&\stackrel{\mathclap{\textrm{\eqref{eq:nu of sum}}}}{\le} \P_{A\in\CM(n)}\big( g(A_p^{\CS(m)})\le t^*L_g(m) \big)\\
\textrm{by Lemma~\ref{normal-additive}(a)}\qquad &\ll e^{-(t^*\log t^*-t^*+1) L_g(m)} + n^8 \Delta_p(n;n/2) \\
&\ll e^{-(t\log t-t+1)L},
}
where in the last step we used \eqref{distr mod p} to bound $\Delta$ and the facts that $L\le\log n+1$ and that $0<t\log t-t+1<1$ for $t\in(0,1)$. 
This completes the proof of the lemma in the case when $t^*<1$. Lastly, when $t^*\ge1$, we must have that $t=1+O(1/L)$, so that $(t\log-t+1)L=O(1)$. Hence, the lemma holds trivially in this case.
\end{proof}

\subsection{Lifting the Frobenius automorphism} 
Now that we understand the basics about the distribution of $\tau_{A_p}$, we use some standard Galois theory to relate $\tau_{A_p}$ to a certain conjugacy class of the  Galois group $\CG_A$ of $A$, namely the class of the Frobenius automorphism at $p$.

Recall that conjugacy classes of $\CS_n$ are in one-to-one correspondence with partitions of $n$. Indeed, if $g\in \CS_n$, then it has a unique decomposition as a product of disjoint cycles. Its conjugacy class is then completely determined by the partition $(\ell_1,\ell_2,\dots,\ell_r)$ whose parts $\ell_j$ are the lengths of the cycles of $g$ listed in increasing order. We call this partition the {\it cycle type} of $g$. 

It turns out that the the cycle type of the Frobenius automorphism at $p$ can be obtained by $\tau_{A_p}$ after {\it merging} certain equal parts of the latter. The following definition makes this notion precise.

\begin{dfn}\label{dfn-merge}
Let  $\rho=(\rho_1,\dots,\rho_r)$ and $\sigma=(\sigma_1,\dots,\sigma_s)$ be two partitions of $n$. In addition, let $y\in\R_{\ge1}$. We say that $\sigma$ is a {\it $y$-merging} of $\rho$ if there are sets $B_1,\dots,B_s$ such that\footnote{As usual, $\cupdot$ is a union of disjoint sets.}
\begin{enumerate}
	\item $B_1\cupdot \cdots\cupdot B_s=[r]$;
	\item $\#B_i\le y$ for all $i\in[s]$;
	\item $\sigma_i=\sum_{j\in B_i} \rho_j$ for all $i\in[s]$;
	\item $\rho_j=\rho_k$ for all $j,k\in B_i$ and all $i\in[s]$.
\end{enumerate}
\end{dfn}

\begin{exm*}
	The partitions $(1,1,2,3,4)$ and $(2,2,3,4)$ are 2-mergings of $(1,1,2,2,2,3)$. However, the partition $(2,3,6)$ is not a 2-merging of $(1,1,2,2,2,3)$.
\end{exm*}

\begin{lem}\label{lem:liftingFrob}
	Let $A\in \mathbb{Z}[T]$ be a monic square-free polynomial of degree $n$, let $p$ be a prime number, and let
	\[
	M=\max\{m\in\N: \exists\ \text{irreducible}\ I\in \F_p[T]\ \text{such that}\ I^m|A_p\} .
	\]
	Then the Galois group of $A$ contains an element whose cycle type is an $M$-merging of $\tau_{A_p}$.
\end{lem}

\begin{proof} Write $A=\prod_{i=1}^n(T-x_i)$ with $\Omega=\{x_1,\ldots,x_n\} \subseteq \mathbb{C}$ its set of roots. Let $F$ be the splitting field of $A$, that is to say, $F=\Q(x_1,\dots,x_n)$. In particular, $F$ is a Galois extension of $\Q$. Let us also write $\CO_F$ for the ring of integers of $F$.
	
	Now, consider a prime ideal $\mathfrak P$ of $\CO_F$ lying above $p$. Given $x\in\CO_F$, we write $\bar{x}$ for its reduction mod $\mathfrak{P}$. We then have 
	\[
	A_p \equiv A \equiv \prod_{i=1}^n (T-x_i) \quad \mod{\mathfrak P}.
	\]
	Thus, the polynomial $A_p$ splits completely in the field $\CO_F/\mathfrak P$ with roots $\bar{x}_1,\dots,\bar{x}_n$ listed with multiplicity. In particular, we may partition the elements of $\Omega$ according to their reduction mod $\mathfrak P$: for each root $\bar{x}$ of $A_p$, we let
	\[
	\Omega_{\bar x} = \{ x_i\in \Omega: x_i\equiv \bar x \mod{\mathfrak P} \}  .
	\]
	Thus, if we let $\bar{\Omega}=\{\bar{x}_i: i=1,2,\dots,n\}$, we have
	\begin{equation}\label{eq:partitionmodp}
	\Omega = \bigcupdot_{\bar x\in \bar{\Omega}} \Omega_{\bar x} .
	\end{equation}
	
	Now, let us consider the  Frobenius automorphism $\phi_p\colon \CO_F/\mathfrak P\to \CO_F/\mathfrak P$, defined by $\phi_p(\bar{x}):=\bar{x}^p$. A classical result from algebraic number theory \cite[Theorem 32, p. 77]{marcus} states that $\phi_p$ can be lifted to an element of $\CG_A$, that is to say there is some $\phi\in \CG_A$ such that
	\[
	\phi(x) \equiv  x^p \mod{\mathfrak P}
	\qquad \forall x\in \CO_F .
	\]
	In particular, $\phi(\Omega_{\bar{x}}) = \Omega_{\bar{x}^p}$. This will allow us to relate the factorization type of $A_p$ to the cycle type of $\phi$.
	
	Indeed, let $I\in \mathbb{F}_p[T]$ be an irreducible polynomial of degree $d$ that divides $A_p$ exactly $m>0$ times. In particular, we have $\#\Omega_{\bar{x}} =m$ for all $x\in\Omega$ with $I(\bar{x})=0$. The Frobenius automorphism $\phi_p$ acts transitively on the roots of $I$, so there is an ordering of them, say $\bar{\alpha}_1,\ldots, \bar{\alpha}_d$ with  $\alpha_1,\dots,\alpha_d\in\Omega$, such that  $\phi(\bar{\alpha}_i)=\bar{\alpha}_{i+1}$ with the convention that $\bar{\alpha}_{d+1}=\bar{\alpha}_1$. We will use this fact to prove the following statement.
	
	\begin{claim}\label{claim-galois} Let $i\in[d]$ and $y_i\in \Omega_{\bar{\alpha}_i}$. The orbit of $y_i$ under $\phi$ has length equal to $dm'$, where $m'=m'(y_i)$ is an integer $\le m$.
	\end{claim}
	
	The above claim will clearly complete the proof, since it implies that the cycle type of $\phi$ is an $M$-merging of the factorization type of $A_p$.
	
	To prove Claim~\ref{claim-galois}, fix some $y_i\in \Omega_{\bar{\alpha}_i}$, where $i\in[d]$. Since $\phi$ sends $\Omega_{\bar{\alpha}_j}$ to $\Omega_{\bar{\alpha}_{j+1}}$, we find that $\phi^k(y_i)\in \Omega_{\bar{\alpha}_i}$ if, and only if, $k\equiv 0\mod d$. So the length of the orbit of $y_i$ is $\ell=dm'$ for some $m'>0$. In addition, the numbers $y_i,\phi^d(y_i),\dots,\phi^{(m'-1)d}(y_i)$ are distinct elements of $\Omega_{\bar{x}_i}$. Since $\#\Omega_{\bar{x}_i}=m$, we conclude that $m'\le m$. This completes the proof of Claim~\ref{claim-galois}, and hence of Lemma~\ref{lem:liftingFrob}.
\end{proof}

\subsection{Reduction of Proposition~\ref{from distr to galois} to two lemmas}
We may assume that $A$ is irreducible, in particular separable. 
In view of Lemma~\ref{lem:liftingFrob}, we have two possibilities: 
\begin{itemize}
	\item[(i)] either there is some irreducible polynomial $I\in\F_p[T]$ that divides $A_p$ to a power higher than $(\log n)^3$;
	\item[(ii)] or $\CG_A$ contains an element whose cycle type is a $(\log n)^3$-merging of $\tau_{A_p}$.
\end{itemize}
Since $A$ is irreducible, $\CG_A$ is transitive, so option (ii) implies that:
\begin{itemize}
	\item[(ii')] $\exists g\in \CT_n$ whose cycle type is a $(\log n)^3$-merging of $\tau_{A_p}$ (recall the definition of $\CT_n$, \eqref{eq:def Tn}).
\end{itemize}

The above discussion reduces the proof of \eqref{eq:galois-goal} (and hence of Proposition~\ref{from distr to galois}) to showing that conditions (i) and (ii') occur with low probability. This is the context of the following two lemmas.

\begin{lem}\label{lem:powerfree} Let $p$ be a prime and let $\mu_0,\mu_1,\dots,\mu_{n-1}$ be a sequence of probability measures such that
\[
\Delta_p(n;n/\log n)\le 1/n
	\qquad\text{and}\qquad 
\sup_{0\le j<n} \sum_{a\equiv0\mod p} \mu_j(a) \le 1-1/(\log n)^2. 
\]
Let $\CE$ be the set of $A\in\CM(n)$ for which there is an irreducible polynomial $I\in\F_p[T]$ dividing $A_p$ to a power higher than $(\log n)^3$. Then 
\[
\P_{\CM(n)}(\CE)  \ll 1/n. 
\]
\end{lem}

\begin{lem}\label{lem:LP} 
	Let $\nu$ be the measure defined by \eqref{nu-dfn}, where $n\ge16$ and $p$ is a prime satisfying \eqref{distr mod p} for some $\lambda>0$. 
	Then there is some absolute constant $c>0$ such that
	\[
	\nu\big(\{\rho\vdash n \,:\,  \exists g\in \CT_n\ \mbox{whose cycle type is a $(\log n)^3$-merging of $\rho$}\}\big)  \ll_\lambda n^{-c\lambda} .
	\]
\end{lem}

Lemma~\ref{lem:powerfree} has a simple proof that we give below. On the other hand, Lemma~\ref{lem:LP} is significantly more complicated, with its proof comprising the entirety of Section~\ref{proof of thm:LP}.

\begin{proof}[Proof of Lemma~\ref{lem:powerfree}] 	
	The probability that $T^m|A_p$ with $m> (\log n)^3$ is $\ll1/n$ by \eqref{eq:16half} applied with $\delta=(\log n)^{-2}$. Hence, 
	\[
	\P_{\CM(n)}(\CE) = \P_{\CM(n)}(\CE')+O(1/n),
	\]
	where $\CE'$ is the set of $A\in \CM(n)$ for which there is an irreducible polynomial $I\in\F_p[T]$ that is different than $T$ and that divides $A_p$ to a power higher than $(\log n)^3$. Note that if there is such an $I$, it must satisfy that $\deg(I)\le \deg(A)/(\log n)^3\le n/(\log n)^3$ and $I^{\ell^2}|A_p$ with $\ell:=\lfloor \log n\rfloor$. Thus, if we write $\CI_k$ for the set of monic irreducible polynomials of $\F_p[T]$ of degree $k$, we find that
	\[
	\P_{\CM(n)}(\CE') 
	\leq \sum_{k\le n/(\log n)^3} \sum_{I\in\CI_k} \P_{A\in\CM(n)}\big(I^{\ell^2}|A_p\big) 
	\leq \sum_{k\le n/(\log n)^3} \sum_{I\in\CI_k} \frac{1}{\|I\|_p^{\ell^2}} + \Delta_p(n;n/\log n) .
	\]
Using Proposition~\ref{PPT} and our assumption that $\Delta_p(n;n/\log n)\le 1/n$, we conclude that
	\[
	\P_{\CM(n)}(\CE') 
		\leq \sum_{k\le n/(\log n)^3} \frac{p^k/k}{p^{k\ell^2}}+\frac{1}{n}
	\ll \frac{1}{p^{\ell^2-1}}+\frac{1}{n}\ll \frac{1}{n} .
	\]
	This completes the proof of the lemma.
\end{proof}

\section{A \L uczak-Pyber style theorem}\label{proof of thm:LP}

In 1993, \L uczak and Pyber \cite{LP93} proved that 
\[
\#\CT_n/\#\CS_n \ll n^{-c}
\] 
for some absolute constant $c>0$. The order of magnitude of the ratio $\#\CT_n/\#\CS_n$ was determined in various cases by Eberhard, Ford and Koukoulopoulos \cite{EFK} with the exact answer depending on certain arithmetic properties of $n$. In \cite{BSK}, the first and third author of the present paper strengthened the \L uczak-Pyber estimate in a different direction: they showed that if we choose a permutation $g\in\CS_n$ uniformly at random, then with high probability we have that $h\notin \CT_n$ for {\it any} permutation $h\in\CS_n$ that differs from $g$ only in cycles of length $\le n^\theta$, with $\theta<1-(1+\log\log2)/\log2=0.08607\dots$. We will prove Lemma~\ref{lem:LP} by rehashing the argument from \cite{BSK} in the broader setting of our paper. As a matter of fact, we will establish the following even more general result which, when combined with Lemma~\ref{lem:PartitionMeasure}, implies Lemma~\ref{lem:LP} immediately.

\begin{prop}[A generalized \L uczak-Pyber result]\label{thm:LP}
Let $\mu$ be a probability measure on the set of partitions of $n$, and write $\rho$ for a random partition of $n$ sampled according to $\mu$. Assume that there are constants $C\ge1$, $t\in(0,1)$, $\kappa\in(0,1]$ and $\delta\in(0,1/10]$ such that the following hold:
	\begin{enumerate}
		\item For any $k,\ell\in[2, n/4] \cap \Z$, we have
		\[
			\mu(\{k,\ell\}\subseteq\rho) \le C/(k\ell) .
		\]
		\item For all $k\in[n^{1-\delta/2},n/2]\cap\Z$, we have
		\[
		\mu\Big(\exists U\subseteq\rho \text{ such that }\sum_{u\in U}u=k\Big)\le Ck^{-\delta}.
		\]
		\item Let $f\colon \N\to\{0,1\}$ and $m\in[1,n/\log n]\cap \Z$, and set $L=\sum_{k=1}^m f(k)/k$. We then have
		\[
		\mu\bigg( \sum_{k\in \rho\cap[1,m]}f(k)\le tL \bigg)  \le C \cdot  e^{-\kappa L} ,
		\]
		where the parts of $\rho$ are summed according to their multiplicity.
	\end{enumerate}
Then, for any fixed $\eps\in(0,\delta/2)$, we have that
	\[
	\mu\big(\exists g\in\CT_n\ \mbox{whose cycle type is an $n^{\theta}$-merging of $\rho$}\big)  \ll_{C,t,\kappa,\delta,\eps} (\log n)^2 n^{-\kappa (\delta/4-\theta/2)}
	\]
	uniformly for $\theta\in[0,\delta/2-\eps]$. 
\end{prop}

\begin{rem*}Condition (c) is only necessary to get a polynomial estimate for the probability. It can be replaced by a stronger version of (a), where $C=1+\eps$, but the the resulting estimate will be worse. Condition (b), on the other hand, is necessary to preclude groups like $S_{n/2}^2\rtimes(\Z/2\Z)$ (when $n$ is even, in this example).
\end{rem*}
\begin{notat*}
As in \S~\ref{sec:GaloisTheory}, we use multi-set notation for partitions. Throughout the proof, we use the notation $\mathbb{P}(E):=\mu(E)$ and $\mathbb{E}(X):=\int X\,d\mu$. A random partition will be denoted by $\rho$.  In addition, we set
\begin{equation}\label{alpha-dfn}
\alpha:=\delta/4-\theta/2 \in[\eps/2,1/40]. 
\end{equation}
All implied constants in the big-Oh notation might depend on $C,t,\kappa,\delta$ and $\eps$ without further notice. 
Finally, we will be assuming without loss of generality, that $n\ge n_0$, where $n_0$ is a constant that is sufficiently large in terms of $C,t,\kappa,\delta$ and $\eps$.
\end{notat*}

\subsection{The anatomy of a typical partition} In this subsection, we collect various lemmas that establish that a randomly sampled partition satisfies various properties with high probability.

\begin{lem}
	\label{glem:log3} 
	Let $\mu$ be a measure on partitions of $n$ satisfying condition (a) of Proposition~\ref{thm:LP}. Let $\CE_1$
	be the set of $\rho\vdash n$ satisfying that there are no integers $k,\ell\le n/4$ with $\gcd(k,\ell)\ge n^{\kappa\alpha}$ such that $\{k,\ell\}\subset\rho$. Then
	\[
	\mathbb{P}(\CE_1)  \ge1- O((\log n)^2n^{-\kappa\alpha}) .
	\]
\end{lem}

\begin{rem*}
The case $k=\ell$ is included in the definition of $\CE_1$. So, if $\rho\in\CE_1$, then every integer $k\in[ n^{\kappa\alpha},n/4]$ occurs with multiplicity $\le1$ in $\rho$.
\end{rem*}

\begin{proof} Note that $\CE_1^c=\bigcup_{r\ge n^{\kappa\alpha}}\CB_r$, where $\CB_r$ denotes the event that there exist integers $i,j\le n/(4r)$ such that $\{ri,rj\}\subset \rho$. Then
\[
	\mathbb{P}(\CB_r)
			\le\sum_{i,j \le n/(4r)}
		\mathbb{P}(\{ri,rj\}\subseteq\rho)
			\le\sum_{i,j \le n/4}\frac{C}{r^2ij}  \le \frac{C}{r^2}\cdot (\log n)^2 ,
\]
where we used the fact that $\sum_{j\le x}1/j\le 1+\log x$ for all $x\ge1$. Summing the above estimate over $r\ge n^{\kappa\alpha}$ completes the proof of the lemma.
\end{proof}

\begin{lem}\label{lem:nosubsums} Let $\mu$ be a measure on partitions of $n$ satisfying condition (b) of Proposition~\ref{thm:LP}. Let $\CE_2$ be the set of $\rho\vdash n$ such that $\sum_{u\in U}u \neq nj/r$ whenever $U\subseteq \rho$, $r|n$, $2\le r\le n^{\delta/2}$ and $j\in\{1,2,\dots,r-1\}$. Then
\[
\P(\CE_2) \ge 1- O(n^{-\delta/4}) .
\]
\end{lem}

\begin{proof} Note that if there is $U\subset \rho$ such that $\sum_{u\in U}u=nj/r$, then there is also $V\subset \rho$ (consisting of the parts of $\rho$ that are not in $U$) such that $\sum_{v\in V}v= n(r-j)/r$. Hence, we may assume that $j\le r/2$ in the definition of $\CE_2$ so that $nj/r\le n/2$. Since we also have that $nj/r\ge n^{1-\delta/2}$, condition (b) of Proposition~\ref{thm:LP} implies that
	\[
	\P\bigg(\exists U\subset \rho\ \text{such that}\ \sum_{u\in U}u=\frac{nj}{r}\bigg)  \ll (nj/r)^{-\delta} .
	\]
	Summing the above estimate over $r|n$ with $2\le r\le n^{\delta/2}$, and over $j\in[1,r/2]\cap\Z$, we find that
	\[
	\P(\CE_2^c)\ll n^{-\delta} \sum_{\substack{r\le n^{\delta/2} \\ r|n}}  r^{\delta} \sum_{j\le r/2} j^{-\delta} 
		\ll n^{-\delta} \sum_{\substack{r\le n^{\delta/2} \\ r|n}}   r \le n^{-\delta/2}\cdot \#\{r|n\}.
	\]
	Since $n$ has $\ll n^{\delta/4}$ divisors, the lemma follows.
\end{proof}

\begin{lem} \label{glem:2elems}
	Let $\mu$ be a measure on partitions of $n$ satisfying condition (c) of Proposition~\ref{thm:LP}. 
	Let $\CE_3$ denote the event that, counting with multiplicity, there are at least $\tfrac{\alpha t}{2}\log n$ parts of $\rho$ that lie in $[n^{1-\alpha},n/\log n]$. Then
	\[
	\mathbb{P}(\CE_3)\ge 1-O\big( (\log n)^{\kappa}n^{-\kappa\alpha} \big) .
	\]
\end{lem}

\begin{proof} We shall apply condition (c) of Proposition~\ref{thm:LP} with $f(k)=1_{k\ge n^{1-\alpha}}$ and $m=n/\log n$. We have that
	\[
	\sum_{k=1}^m \frac{f(k)}{k} = \sum_{n^{1-\alpha}\le k\le n/\log n}\frac{1}{k} 
		= \alpha\log n-\log\log n+O(1).
	\]
Hence the lemma follows by condition (c) of Proposition~\ref{thm:LP}. 
\end{proof}

\begin{lem} \label{glem:smooth-elems}
	Let $\mu$ be a measure on partitions of $n$ satisfying condition (c) of Proposition~\ref{thm:LP}. 
	Let $\CE_4$ denote the event that, counting with multiplicity, there are at least $\tfrac{t}{4} \log n$ parts of $\rho$ lying in the set $\{k\le \sqrt{n}/3:\exists p>n^{1/8}\ \text{such that}\ p|k\}$. Then
	\[
	\mathbb{P}(\CE_4)\ge 1-O(n^{-\kappa/4}) .
	\]
\end{lem}

\begin{proof} We may assume $n$ is sufficiently large. Given an integer $k$, let $P^+(k)$ denote its largest prime factor with the convention that $P^+(1)=1$. We shall apply condition (c) of Proposition~\ref{thm:LP} with $f(k)=1_{P^+(k)>n^{1/8}}$ and $m=\sqrt{n}/3$. We have that
	\als{
	\sum_{k=1}^m \frac{f(k)}{k} = \sum_{\substack{k\le \sqrt{n}/3 \\ P^+(k)>n^{1/8} }}\frac{1}{k} 
		&= \sum_{k\le \sqrt{n}/3} \frac{1}{k}  - \sum_{P^+(k)\le n^{1/8} }\frac{1}{k} \\
		&\ge \frac{\log n}{2}+O(1) -\prod_{p\le n^{1/8}}\Big(1-\frac{1}{p}\Big)^{-1} \\
		&= (1/2-e^\gamma/8)\log n+O(1)
	}
by Mertens' estimate \cite[Theorem 3.4(c)]{kou}, where $\gamma$ denotes the Euler constant. Since $1/2-e^\gamma/8>1/4$, we conclude that $\sum_{k=1}^mf(k)/k \ge (\log n)/4$ for $n$ sufficiently large. Hence the lemma follows by  condition (c) of Proposition~\ref{thm:LP}. 
\end{proof}

\begin{lem}
	\label{glem:not divisible by r} 
	Let $\mu$ be a measure satisfying conditions (a) and (c) of Proposition~\ref{thm:LP}. Let $\CE_5$ be the event that for all $r\ge2$
	there exists a $k\in\rho\cap[n^{1-2\alpha},n/\log n]$ such that $r\nmid k$. Then
	\[
	\mathbb{P}(\CE_5) \ge 1-O\big((\log n)^2n^{-\kappa\alpha} \big) .
	\]
\end{lem}

\begin{proof} Let $\CB_5$ denote the complement of $\CE_5$, so that our goal is to show that $\P(\CB_5)\ll  (\log n)^2 n^{-\kappa\alpha}$. Let $\CE_1$ and $\CE_3$ be the events of Lemma~\ref{glem:log3} and~\ref{glem:2elems} for which we know that $\P(\CE_1^c),\P(\CE_3^c) \ll (\log n)^2n^{-\kappa\alpha}$. Hence, the lemma will follow if we prove that 
	\begin{equation}\label{eq:not divisible by r-goal}
	\P(\CB_5\cap\CE_1\cap\CE_3)\ll n^{-\kappa\alpha}.
	\end{equation}
	
	If a partition $\rho$ lies in $\CE_1\cap\CE_3$, then all parts in $[n^{1-2\alpha},n/\log n]$ are distinct, and there are at least two such parts, say $k$ and $\ell$. In addition, for each $r\ge n^{\kappa\alpha}$, at most one of $k$ and $\ell$ are divisible by $r$, so $\rho$ has at least one part in $[n^{1-2\alpha},n/\log n]$ not divisible by $r$. This implies that
	\begin{equation}\label{eq:not divisible by r-reduction}
	\CB_5\cap \CE_1\cap\CE_3 \subseteq \bigcup_{2\le r\le n^{\kappa\alpha}}\CB_5(r),
	\end{equation}
	where $\CB_5(r)$ denotes the event that $\rho\in\CE_1\cap\CE_3$ but there is no $k\in\rho\cap[n^{1-2\alpha},n/\log n]$ such that $r\nmid k$. 
	We bound the probability of occurrence of $\CB_5(r)$ using condition (c) of Proposition~\ref{thm:LP}. 
	
	Consider the function $f_r(k)=1_{k\ge n^{1-2\alpha},\,r\nmid k}$. We then have that
	\als{
	\sum_{k\le n/\log n} \frac{f_r(k)}{k}= \sum_{\substack{n^{1-2\alpha}\le k\le n/\log n \\ r\nmid k}} \frac{1}{k}
	 &= \sum_{n^{1-2\alpha}\le k\le n/\log n}\frac{1}{k}- \sum_{\max\{1,n^{1-2\alpha}/r\}\le \ell \le (n/\log n)/r} \frac{1}{r\ell}\\
	&= 2\alpha(1-1/r)\log n - (1-1/r)\log\log n + O(1)
	}
	uniformly for $r\ge2$ and $n\ge3$. Hence,
	\[
	\P(\CB_5(r))\le \P\bigg( \sum_{k\in\rho\cap [1,n/\log n]} f_r(k) \le t\sum_{k\le n/\log n} \frac{f_r(k)}{k}\bigg)
		\ll   (\log n)^{\kappa} n^{-2\kappa\alpha(1-1/r)} 
	\]
by condition (c) of Proposition~\ref{thm:LP}. Using the union bound, we conclude that
	\als{
		\P\Big(\bigcup_{2\le r\le n^{\kappa\alpha}}\CB_5(r)\Big) 
			&\ll \sum_{2\le r\le n^{\kappa\alpha}}  (\log n)^{\kappa} n^{-2\kappa\alpha (1-1/r)}  \\
			&\le (\log n)^{\kappa} \bigg( n^{-\kappa\alpha}+ \sum_{3\le r\le \log n}n^{-4\kappa\alpha /3} + \sum_{\log n<r\le n^{\kappa\alpha}} (e/n)^{2\kappa\alpha}\bigg) \\
			&\ll (\log n)^{\kappa} n^{-\kappa\alpha} .
	}
	Together with \eqref{eq:not divisible by r-reduction} this shows that \eqref{eq:not divisible by r-goal} does hold, and so the proof is complete.
\end{proof}

\subsection{Group theory}
We now move to the group-theoretic part of the proof.

\begin{notat*} Given $\rho\vdash n$ and $y\ge1$, we let $\M(\rho;y)$ denote the set of all permutations $g\in\CS_n$ whose cycle type is a $y$-merging of $\rho$.

Given any permutation $g\in \CS_n$, we define $\deg g=\#\{i\in[n]:g(i)\ne i\}$. Then, for each $G\leqslant S_{n}$, we let 
$\min\deg G=\min_{g\in G\smallsetminus\{1\}}\deg g$.
\end{notat*}

\begin{lem}
	\label{glem:deg}
	If $G$ is a primitive transitive subgroup of $\CS_n$ that is different than $\CA_n$ and $\CS_n$, then 
	\[
	\min\deg G \ge(\sqrt{n}-1)/2.
	\]
\end{lem}

\begin{proof}
	See \cite[Claim 1]{BSK}.
\end{proof}

\begin{lem}
	\label{glem:prim} 
		There exists $n_0$ such that if $g\in \M(\rho;n^{1/8})$ with $n\ge n_0$ and $\rho\in \CE_1\cap \CE_4$, then $g$ cannot belong to a transitive primitive group $G\leqslant \CS_n$ that is different than $\CA_n$ and $\CS_n$.
\end{lem}

\begin{rem*}
Here and below, $n_0$ can depend on the parameters $C$, $t$, $\kappa$ and $\delta$ of Proposition \ref{thm:LP}.
\end{rem*}

\begin{proof} Let $\CP$ be the set of primes $>n^{1/8}$ that divide a part of $\rho$ lying in $(n/4,n]$. Since there are at most three such parts, and since an integer $\le n$ has $\le 8$ prime factors $>n^{1/8}$, we have that $\#\CP\le 24$. 
	
	Our partition $\rho$ lies in $\CE_1$. Hence, for each $p\in\CP$, there is at most one part in $\rho\cap[1,n/4]$ that is divisible by $p$ (the condition in $\CE_1$ holds for all $p\ge n^{\kappa\alpha}\ge n^{1/40}$, so it applies for $p\in\CP$). So, all in all, there are $\le24$ parts in $\rho\cap[1,n/4]$ that are divisible by some prime in $\CP$. On the other hand, our assumption that $\rho\in\CE_4$ implies that, counting with multiplicities, there are $\ge \tfrac{t}{4}\log n$ parts in $\rho\cap[1,\sqrt{n}/3]$ whose largest prime factor is $>n^{1/8}$. In fact, each such part is $>n^{1/8}$, so its multiplicity of occurrence in $\rho$ must equal 1 because $\rho\in\CE_1$. Hence, there are $\ge\tfrac{t}{4}\log n$ distinct parts in $\rho\cap(n^{1/8},\sqrt{n}/3]$. Comparing cardinalities, and assuming that $n$ is sufficiently large, we conclude that there is at least one part $k\in\rho\cap[1,\sqrt{n}/3]$ that is coprime to all elements of $\CP$, and that has largest prime factor $>n^{1/8}$. Call $p$ this prime. By construction, $p\mid k$ and $p\nmid \ell$ for each $\ell\in \rho\cap(n/4,n]$. In addition, since $\rho\in \CE_1$, we must have that $p\nmid \ell$ for each $\ell\in\rho\cap[1,n/4]$ that is different from $k$. We conclude that $p$ divides $k$ but no other part of $\rho$. 
	
	Let $g\in \M(\rho;n^{1/8})$ and write $\tau$ for its cycle type. Since $k$ occurs with multiplicity 1 in $\rho$, it must also be a part of $\tau$. Any other part of $\tau$ must be of the form $m\ell$ with $m\le n^{1/8}$ and $\ell\neq k$. In particular, $p\nmid m\ell$ because $p>n^{1/8}$ and $p\nmid \ell$. We conclude that $g$ has exactly one cycle whose length is divisible by $p$, and that this cycle has length $k$. 
	
	For each prime $q$, let $a_q$ denote the largest integer such that $q^{a_q}$ divides a cycle length of $g$. In particular, $a_p$ is the $p$-adic valuation of $k$. So, if we set $m= p^{a_p-1}\prod_{q\neq p} q^{a_q}$ (which is a finite integer), then $g^m$ is the product of exactly $k/p$ cycles of length $p$. In particular, $\deg(g^m)=k\le \sqrt{n}/3<(\sqrt{n}-1)/2$ and $g^m\ne1$. Consequently, any group $G\leqslant \CS_n$ containing $g$ must have $\min\deg G<(\sqrt{n}-1)/2$. In view of Lemma~\ref{glem:deg}, such a group cannot be a primitive transitive subgroup of $\CS_n$ that is different than $\CA_n$ and $\CS_n$, and so the proof is complete.
\end{proof}

\begin{lem}
	\label{glem:imprim}
	There exists $n_0$ such that if $g\in \M(\rho;n^\theta)$ with $n\ge n_0$, $\theta\in [0,\frac{\delta}{2}-\eps]$, and $\rho\in \CE_1\cap\cdots\cap \CE_5$, then $g$ cannot belong to a transitive imprimitive group $G\leqslant \CS_n$.
\end{lem}

\begin{proof} Let $G$ be a transitive imprimitive subgroup of $\CS_n$. Hence, $G$ preserves a block structure, namely, 
	there must exist some $r|n$, $1<r<n$, and a decomposition
	of $[n]$ into disjoint sets $B_{1},\dotsc,B_{r}$ of common size
	$s=n/r$ such that for every $i\in[r]$ and every $g\in G$,
	$g(B_{i})=B_{j}$ for some $j$. (Such a collection of $B_i$'s is also called an {\it imprimitivity block system}.)
	
	Throughout we use the following observation: if $L$ is a cycle of length $\ell$ in a permutation that
	preserves a block structure of $r$ blocks, then $L$ intersects $r'\le r$ blocks,
	its intersection with each block is of size $s'\le s$, and $\ell=r's'$.
	Further, the set of blocks intersecting $L$ is an invariant set of
	$g$, and any other cycle in this set has its length divisible by
	$r'$. 
	
	Now, assume for contradiction that there is some $g\in G\cap \M(\rho;n^\theta)$. We divide the proof into cases according to the size of $r$.
	
	\medskip
	
	\emph{Case 1:} $2\le r\le n^{\delta/2}$. Since $\rho\in \CE_5$, it has
	a part of length $\ell\in[n^{1-2\alpha},n/4]$ such that $r\nmid \ell$. Since $\rho\in \CE_1$, it has no other part of length $\ell$, and hence $g$ must have a cycle of length $\ell$, denote it by $L$. Assume $L$ intersects $r'$ blocks of the imprimitivity system. We cannot have $r'=r$ because then $r$ would divide $\ell$, in contradiction to our choice of $\ell$. The union of the blocks intersecting $L$ is invariant under $g$
	and has size $nr'/r$. Thus there is some subset $V$ of the lengths of the cycles of $g$ such that $\sum_{v\in V}v=nr'/r$. Since these lengths are merely mergings of parts of $\rho$, it follows that $\rho$ too must possess a subset $U$ of its parts such that $\sum_{u\in U}=nr'/r$. But this contradicts our assumption that $\rho\in \CE_2$.

\medskip

	\emph{Case 2:} $n^{\delta/2}<r<n^{1-\alpha}$. Since $\rho\in \CE_3$, there are
	at least two parts of $\rho$ in $[n^{1-\alpha},n/\log n]$ for $n_0$ sufficiently large. Let us denote them by $\ell_{1}$
	and $\ell_{2}$. Since $\rho\in \CE_1$, these two parts must be distinct, and $\rho$ has no other parts of lengths either $\ell_1$ or $\ell_2$. We conclude that $g$ has cycles $L_{1}$ and $L_{2}$ of lengths $\ell_1$ and $\ell_2$, respectively. Let $r_{i}'$ be the
	number of blocks that $L_{i}$ intersects, and let $s_{i}'=\ell_{i}/r_{i}'$. We divide the argument into two subcases, according to the size of $s_1'$ and $s_2'$.
	
	\emph{Case 2a:} $s_1'=s_2'=s$. We then have that $s$ divides both $\ell_1$ and $\ell_2$, and since 
	$s=n/r>n^{\alpha}$, this contradicts our assumption that $\rho\in\CE_{1}$.
	
	\emph{Case 2b:} $s_i'<s$ for some $i\in\{1,2\}$. Then the set of blocks preserved by $L_i$
	contains another cycle, call it $L_3$, whose length is also divisible by $r_i'$. 
	On the one hand, we have $r_i'=\ell_i/s_i'>n^{1-\alpha}/s=r/n^{\alpha}>n^{\delta/2-\alpha}$. 
	On the other hand, since $g\in\M(\rho;n^\theta)$, the length of $L_3$ must equal $mk$, where $m\le n^\theta$ and $k\in\rho$. Since $r_i'|mk$, we conclude that $\gcd(r_i',k)\ge r_i'/m>n^{\delta/2-\alpha-\theta}=n^\alpha$. This of course implies $\gcd(k,\ell_i)>n^\alpha$ and contradicts our assumption that $\rho\in \CE_1$. 
	
	\medskip

	\emph{Case 3:} $n^{1-\alpha}\le  r<n$. Since $r|n$, we must have that $r\le n/2$. Our assumption that $\rho\in \CE_5$ implies that there is some $\ell\in\rho\cap[n^{1-2\alpha},n/\log n]$ such that $s\nmid \ell$. 	Since $\rho\in \CE_1$, there is no other part of length $\ell$. Consequently, $g$ must contain a cycle of length $\ell$, denote it by $L$. Assume $L$ intersects $r'$ blocks. Since $s\nmid \ell$, we get that $s'=\ell/r'<s$, and hence there exists another cycle $L'$ of $g$ divisible by $r'$. Since we merge no more than $n^\theta$ parts at a time, the length of $L'$ must equal $mk$, where $m\le n^\theta$ and $k\in\rho$. Since $r'|mk$, we infer that  $\gcd(k,\ell)\ge r'/m\ge r'/n^\theta$. But $r'=\ell/s'>n^{1-2\alpha}/s=r/n^{2\alpha}\ge n^{1-3\alpha}$ and again we reach a contradiction to $\rho\in \CE_1$ because $\alpha \le 1/40$.
	
	\medskip
	
	We covered all possibilities for $r$, arriving each time at a contradiction. We conclude that $G\cap\M(\rho;n^\theta)=\emptyset$. Since $G$ was chosen arbitrarily among all imprimitive transitive subgroups of $\CS_n$, the lemma is proved.
\end{proof}

\begin{proof}
	[Proof of Proposition~\ref{thm:LP}]
	Let $\mu$ be a measure
	satisfying all three conditions of  the proposition. According to Lemmas~\ref{glem:log3},
	\ref{lem:nosubsums},~\ref{glem:2elems},~\ref{glem:smooth-elems} and~\ref{glem:not divisible by r}, we have that 
	\[
	\mathbb{P}(\CE_{1}\cap \cdots \cap \CE_5) \ge 1- O\big((\log n)^2n^{-\kappa\alpha}\big) .
	\]
	Now, assume that $n\ge n_0$ and apply Lemmas~\ref{glem:prim} and~\ref{glem:imprim}. We get that
	for any $\rho\in \CE_{1}\cap\dotsb\cap \CE_{5}$, any permutation $g\in\M(\rho;n^\theta)$ cannot belong to a transitive $G\leqslant \CS_n$, primitive or imprimitive, unless	$G=\CA_n$ or $G=\CS_n$. The proposition is thus proved.
\end{proof}


\bibliographystyle{plain}

 \end{document}